\documentclass[letterpaper,10pt]{amsart}
\usepackage{amssymb}
\usepackage{amsthm}
\usepackage[utf8]{inputenc}
\usepackage[margin=1.0in]{geometry}
\usepackage{pst-node}
\usepackage{tikz-cd} 
\usepackage{thmtools}
\usepackage{thm-restate}
\usepackage{hyperref}
\usepackage{amsrefs}
\usepackage[noabbrev,capitalize]{cleveref}
\usepackage{amsmath}
\usepackage{amsfonts}
\usepackage{tikz-cd}
\usepackage{mathtools}
\usepackage{bbm}
\usepackage{hyperref}
\usepackage{blindtext, xcolor}
\usepackage{comment}
\usepackage{romannum}
\usepackage{multirow}
\usepackage{array}
\usepackage{wrapfig}
\usepackage{graphicx}
\usepackage{caption}

\DeclareMathOperator{\Covol}{Covol}

\DeclareMathOperator{\GL}{GL}
\DeclareMathOperator{\GO}{GO}

\DeclareMathOperator{\Gal}{Gal}
\DeclareMathOperator{\Hom}{Hom}

\DeclareMathOperator{\Tr}{Tr}

\DeclareMathOperator{\Sym}{Sym}
\DeclareMathOperator{\disc}{disc}

\DeclareMathOperator{\Vol}{Vol}

\DeclareMathOperator{\Len}{Len}

\DeclareMathOperator{\lv}{\lvert}
\DeclareMathOperator{\rv}{\rvert}

\DeclareMathOperator{\Disc}{Disc}
\DeclareMathOperator{\bin}{bin}
\DeclareMathOperator{\poly}{Poly}
\DeclareMathOperator{\mon}{mon}
\DeclareMathOperator{\irr}{irr}
\DeclareMathOperator{\redu}{red}
\DeclareMathOperator{\rlinec}{\ell_3}
\DeclareMathOperator{\llinec}{\ell_2}
\DeclareMathOperator{\flinec}{\ell_1}
\DeclareMathOperator{\dett}{Det}
\DeclareMathOperator{\rank}{Rank}
\DeclareMathOperator{\Cusp}{Cusp}
\DeclareMathOperator{\MT}{Main}
\DeclareMathOperator{\Ree}{Re}
\DeclareMathOperator{\Imm}{Im}

\DeclareMathOperator{\Ind}{Ind}

\newcommand{\ZZ}{\mathbb{Z}}      % for Integers
\newcommand{\QQ}{\mathbb{Q}}      % for Rationals
\newcommand{\RR}{\mathbb{R}}      % for Reals
\newcommand{\cO}{\mathcal{O}}      % for Ring of Integers
\newcommand{\CC}{\mathbb{C}}      % for Complex Numbers
\newcommand{\PP}{\mathbb{P}}      % for Projective Space
\newcommand{\la}{\langle}    
\newcommand{\ra}{\rangle}

\newcommand{\cU}{\mathcal{U}}

\newcommand{\cG}{\mathcal{G}}  
             
\newcommand{\cF}{\mathcal{F}}

\newcommand{\cM}{\mathcal{M}}    
\newcommand{\cL}{\mathcal{L}}  
\newcommand{\cS}{\mathcal{S}}

\newcommand{\eps}{M}  
\newcommand{\MO}{M_0}    
    
\newcommand{\Mod}[1]{\ (\mathrm{mod}\ #1)}

\newtheorem{theorem}{Theorem}[section]

\newtheorem{lemma}[theorem]{Lemma}
\newtheorem{proposition}[theorem]{Proposition}
\newtheorem{corollary}[theorem]{Corollary}

\theoremstyle{definition}

\newtheorem{question}[theorem]{Question}

\theoremstyle{remark}

\newtheorem{definition}[theorem]{Definition}

\usepackage[english]{babel}
\usepackage[utf8]{inputenc}

\title{The Successive Minima of Lattices Arising From Orders In Low Degree Number Fields}

\author{Sameera Vemulapalli}
\address{Department of Mathematics, Princeton University}
\email{sameerav@math.harvard.edu}

\subjclass[2010]{11H50 (primary), 11H06, 11P21, 14H05 (secondary)}
\keywords{Successive minima; Scrollar Invariants; Lattices; Geometry of numbers}

\usepackage[english]{babel}
\usepackage[utf8]{inputenc}

\begin{document}

\maketitle
\pagenumbering{arabic}

\begin{abstract}
Orders in number fields provide natural examples of lattices. We ask: what can the successive minima of lattices arising from orders in number fields be? Given an order $\cO$ of absolute discriminant $\Delta$ in a degree $n$ number field, let $1=\lambda_0,\dots,\lambda_{n-1}$ denote the successive minima. For $3 \leq n \leq 5$ and many groups $G \subseteq S_n$, we compute asymptotics of the points $(\log_{ \Delta }\lambda_{1},\dots,\log_{ \Delta }\lambda_{n-1}) \in \RR^{n-1}$ as $\cO$ ranges across orders in degree $n$ fields with Galois group $G$ as $\Delta \rightarrow \infty$. In many cases, we find that the asymptotics, normalized appropriately, are given by a piecewise linear expression and are supported on a finite union of polytopes.
\end{abstract}

\section{Introduction}
Orders in number fields of degree $n$ provide natural examples of lattices via their embeddings into $\RR^n$ using their real and complex places. Explicitly, we equip an order $\cO$ in a degree $n$ number field with the following lattice structure: denote the homomorphisms of $\cO$ into $\CC$ by $\sigma_1,\dots,\sigma_{n}$, and define 
\[
	\lvert x \rvert \coloneqq \sqrt{\frac{1}{n}\sum_{i = 1}^{n} \lvert\sigma_i(x)\rvert^2}
\]
for $x \in \cO$. In this article, we study lattices via their successive minima; for $i \in \{0,\dots, n-1\}$, the \emph{$i$-th successive minimum} $\lambda_{i}(\cO)$ of $\cO$ is defined to be the smallest positive number $r$ such that $\cO$ contains at least $i + 1$ linearly independent elements of length $\leq r$. 

Minkowski's second theorem states that $\prod_{i} \lambda_i \asymp \Delta^{1/2}$, where $\Delta$ is the absolute discriminant of $\cO$. The main motivating question of this article is: given numbers $(p_1,\dots,p_{n-1})$, how often does $\lambda_i$ approximately equal $\Delta^{p_i}$ for all $1 \leq i\leq n-1$, as we range across orders in degree $n$ number fields ordered by absolute discriminant?

More concretely, we associate a point in $\RR^{n-1}$ to each order $\cO$ in a degree $n$ number field:
\[
	p_{\cO} \coloneqq (\log_{ \Delta }\lambda_{1},\dots,\log_{ \Delta }\lambda_{n-1}) \in \RR^{n-1}.
\]

\begin{question}
Fix a point $p \in \RR^{n-1}$, fix a permutation group $G\subseteq S_n$, and let $\cO$ denote an order in a degree $n$ number field. What are the asymptotics of
\begin{equation}
\label{asympmain}
\#\{\cO \; \mid \; \Delta \leq X \text{ and $\Gal(\cO \otimes \QQ) \simeq G$ and $p_{\cO}$ is ``close'' to $p$}\}
\end{equation}
as $X \rightarrow \infty$? 
\end{question}

\subsection{History}
In the number field setting, this question was first raised by Chiche-lapierre \cite{chiche}, who gave partial results when $n = 3,4,5$ and $G = S_n$. The finite field analogue was addressed by Zhao \cite{zhao}, in the case $n = 3$ and $G = S_3$. Patel \cite{anand} has addressed the geometric analogue in the case $n = 3,4,5$ and $G = S_n$.

Related distributional questions have been also addressed by Terr \cite{Terr97}, who proved the equidistribution of shapes of cubic fields; by Bhargava and H \cite{BhaHa16}, who proved the equidistribution of shapes of $S_n$-fields for $n = 4,5$; Bhargava and H's results were improved by Yifrach \cite{yuval}, who showed the equidistribution of grids of rings of integers in $S_n$-number fields for $n = 3,4,5$; and Holmes \cite{holmes} later proved the equidistribution of shapes in pure prime degree number fields. Our approach differs from that of Terr, Bhargava, H, and Holmes in the following meaningful sense; for $n = 3,4,5$, when ordered by absolute discriminant, $100\%$ of orders in $S_n$-fields have normalized successive minima vectors ``near'' the point $\frac{1}{2(n-1)}(1,\dots,1)$. Thus, equidistribution theorems only ``see'' that one point, but give very refined information at that point. Conversely, our work is focused on computing asymptotics of phenomena that occur with density $0$.

\subsection{Overview of Results}
In this article, we compute the asymptotics of \eqref{asympmain} for various Galois groups $G$ when $n = 3,4,5$ using Delone and Faddeev's parametrization of cubic rings and Bhargava's parametrization of quartic and quintic rings respectively. We show that in many cases, there is a finite union of polytopes $\poly_G \subseteq \RR^{n-1}$ depending on $G$ such that:
\begin{enumerate}
\item for $p \in \poly_G$ we have  $\eqref{asympmain} \asymp X^{d_G(p)}$ for some continuous piecewise linear function $d_G \colon \poly_G \rightarrow \RR_{\geq 0}$;
\item and for $p \notin \poly_G$, we have $\eqref{asympmain} \ll 1$.
\end{enumerate} 

Furthermore, we prove that in many cases, the vertices of $\poly_G$ often ``correspond'' to various interesting \emph{algebraic} families of orders. 

\subsubsection{Sketch of proof}
We use the geometry of numbers methods pioneered by Davenport, Heilbronn, and Bhargava; our main point of departure from these methods is in how we ``average''. In particular, even deep in the cusp, our counting is no more complicated than counting the number of lattice points in a box. In the cubic case, we show that a nondegenerate cubic ring that is ``close'' to $p$, is equipped with a ``short'' basis, and has bounded discriminant gives rise to an integral point in a certain box in $\Sym^3 \RR^2$; conversely, we show that every integral point of the box corresponds to a nondegenerate cubic ring of bounded discriminant ``close'' to $p$ equipped with a ``short'' basis. Using Davenport's lemma, we can easily count the number of lattice points in the box. We complete the proof by estimating the number of ``short'' bases such a cubic ring has, i.e. how many times a ring appears in the box. The quartic and quintic cases are completely analogous. 

\subsection{Results on cubic rings}

\subsubsection{Defining rank $n$ rings}
Define a \emph{rank $n$ ring} to be a commutative ring $R$ whose underlying group is $\ZZ^n$. A rank $n$ ring has a natural trace map $\Tr \colon R \rightarrow \ZZ$ sending an element $x$ to the trace of multiplication by $x$. Define $\Disc(R) = \det(\Tr(x_ix_j))$ for any basis $x_1,\dots,x_n$ of $R$. If $\Disc(R) \neq 0$, we say $R$ is a \emph{nondegenerate}.  Henceforth, let $R$ be a nondegenerate rank $n$ ring of absolute discriminant $\Delta$; such a ring naturally has the structure of a lattice, and therefore we may speak of the \emph{successive minima}
\[
	\frac{1}{\sqrt{n}} \leq\lambda_0(R) \leq \dots \leq \lambda_{n-1}(R)
\]
of $R$ (see \cref{notation} for details). When $R$ is implicit, we simply write $\lambda_i$. For a permutation group $G \subseteq S_n$, we say $R$ is a \emph{$G$-ring} if $\Gal(R \otimes \QQ) \cong G$ as permutation groups.  In this section, assume that $R$ is a nondegenerate rank $3$ ring.

\subsubsection{Density functions and theorem statements}
We ask: for a point $p=(p_1,p_2) \in \RR^{2}$ and a permutation group $G\subseteq S_3$, what are the asymptotics of
\[
	\#\{R \; \mid \; \Delta \leq X \text{ and $R$ is a $G$-ring and $(\log_{\Delta}\lambda_1,\log_{\Delta}\lambda_2)$ is ``close'' to $p$}\}
\]
as $X \rightarrow \infty$? To this end, we make the following definitions. 
\begin{definition}
For any positive real number $\eps$ and any $X \in \RR_{>1}$, define $\cF(p,\eps,G,X)$ to be the set of isomorphism classes of $G$-rings $R$ such that $\Delta \leq X$ and 
\[
	\max_{i = 1,2}\{\lv \log_{\Delta}\lambda_i(R) - p_i \rv \} \leq \frac{\eps}{\log X}.
\]
\end{definition}

Note that if there exists a constant $C \geq 1$ such that $C^{-1}\Delta \leq X \leq C\Delta$, then $\#\cF(p,\eps,G,X)$ is the number of cubic rings of bounded discriminant and fixed Galois group such that $\lambda_i \asymp_{M,C} X^{p_i}$ for $i = 1,2$. In this way, the number $\#\cF(p,\eps,G,X)$ captures the number of cubic rings of bounded discriminant and fixed Galois group that are ``$(\eps,X)$-close'' to $p$. One should think of $\eps$ as being a large, fixed positive real number\footnote{We encourage the reader to think of $\eps$ as a fixed quantity that is large enough to make certain estimates work, where here ``large enough'' depends only on the degree $n$. For $3 \leq n \leq 5$, taking $\eps$ to be one billion is certainly large enough for all our theorem statements.}. In many cases, $\#\cF(p,\eps,G,X) \asymp_{\eps} X^{a}(\log X)^{b}$ for some real numbers $a,b$; in this paper, we are interested in the value $a$, which we extract using the following \emph{density function}.

\begin{definition}
Let $d_{\eps,G} \colon \RR^2 \rightarrow \RR \cup \{\ast\}$ be the function which assigns to $p \in \RR^2$ the value 
\[
	\lim_{X \rightarrow \infty}\frac{\log (\#\cF(p,\eps,G,X)+1)}{\log X }
\]
if the limit exists and else assigns $\ast$. Call $d_{\eps,G}$ the density function.
\end{definition}

We add $1$ to $\#\cF(p,\eps,G,X)$ in the definition above simply to avoid taking a logarithm of zero. The main theorems of this paper consist of computing such density functions; the first such result was obtained by Chiche-lapierre. 

\begin{restatable}{theorem}{chichecubic}{\normalfont (Chiche-lapierre \cite{chiche}, Theorem 3.1.1)} \label{chichecubic} 
There exists a positive real number $\MO$ such that for every $\eps > \MO$ and every $p=(p_1,p_2) \in \RR^2$, we have:
\[
    d_{\eps,S_3}(p) = 
\begin{cases}
    1-(p_2-p_1)& \text{if $p$ is contained in the closed line segment from $(1/6,1/3)$ to $(1/4,1/4)$ } \\
    0              & \text{else.}
\end{cases}
\]
\end{restatable}

We also obtain a theorem for maximal orders in $S_3$-cubic number fields, which we now state. 
\begin{definition}
Let $\cF^{\max}(p,\eps,G,X)$ be the subset of $\cF(p,\eps,G,X)$ consisting of maximal orders in $S_3$-cubic number fields. Let $d^{\max}_{\eps,G} \colon \RR^2 \rightarrow \RR \cup \{\ast\}$ be the function which assigns to $p \in \RR^2$ the value 
\[
	\lim_{X \rightarrow \infty}\frac{\log (\#\cF^{\max}(p,\eps,G,X)+1)}{\log X}
\]
if the limit exists and else assigns $\ast$.
\end{definition}

\begin{restatable}{theorem}{sthreemaxthm}
\label{s3maxthm}
There exists a positive real number $\MO$ such that for every $\eps > \MO$, we have
\[
	d_{\eps, S_3} = d^{\max}_{\eps, S_3}.
\]
\end{restatable}

Finally, we also obtain a \emph{total density function}, i.e. a density function where we do not impose any Galois condition, as follows. 
\begin{definition}
Let $\cF(p,\eps,X)$ be the set of isomorphism classes of nondegenerate cubic rings $R$ with $\Delta \leq X$ and 
\[
	\max_{i = 1,2}\{\lv \log_{\Delta}\lambda_i(R) - p_i \rv \} \leq \frac{\eps}{\log X}
\]
and let $d_{\eps,3}$ denote the corresponding density function. Similarly, let $\cF_3^{\max}(p,\eps,X)$ be the subset of $\cF_3(p,\eps,X)$ consisting of \emph{maximal} cubic rings and define the density function of maximal cubic rings $d^{\max}_{\eps,3}$ as above.
\end{definition}

\begin{restatable}{theorem}{totdensitycubic}
\label{totdensitycubic}
There exists a positive real number $\MO$ such that for every $\eps > \MO$ and every $p=(p_1,p_2) \in \RR^2$, we have:
\begin{align*}
    d_{\eps,3}(p) &= 
\begin{cases}
    1-(p_2-p_1) + \max\{0,1/2-3p_1\}& \text{if $p$ is contained in the closed line segment from } \\
                   & \;\;\;\;\;\; \;\;\;\;\;\;\;\;\;\;\;\;\;\;\; \;\;\;\;\;\;\;\;\;\;\;\;\;\;\; \;\;\;\;\;\; \text{$(0,1/2)$ to $(1/4,1/4)$} \\
    0              & \text{else.}
\end{cases} \\
    d^{\max}_{\eps,3}(p) &= 
\begin{cases}
    1-(p_2-p_1)& \text{if $p$ is contained in the closed line segment from $(1/6,1/3)$ to $(1/4,1/4)$ } \\
    1 & \text{if $p = (0,1/2)$}\\
    0              & \text{else.}
\end{cases}
\end{align*}
\end{restatable}

We do not compute the density function of $\ZZ/3\ZZ$ cubic rings but we have the following result, which follows from \cref{zerocubic}.
\begin{corollary}
\label{a3thm}
The set of limit points of the multiset $\{(\log_{\Delta}\lambda_1(R), \log_{\Delta}\lambda_2(R))\}$ as $R$ ranges across orders in cyclic cubic fields is the closed line segment from $(1/6,1/3)$ to $(1/4,1/4)$.
\end{corollary}

\subsubsection{Observations}

First, observe that the total density function $d_{\eps,3}$ has a very particular structure; it is $1$ minus piecewise differences of the coordinates (the term $p_2-p_1$), plus correction terms (the term $\max\{0,1/2-3p_1\}$). We observe this phenomenon in the quartic and quintic total density functions as well. 

Next, we can show that the rings lying near the boundary of these line segments tend to have to various interesting \emph{algebraic} properties. In particular:
\begin{enumerate}
\item the point $(1/6,1/3)$ ``corresponds'' to cubic rings generated by one element as an algebra (see \cref{monogeniccubictheorem});
\item orders in cubic number fields ``tend'' to lie ''near'' to $(1/4,1/4)$ (see \cref{genericcubiccorollary});
\item orders in reducible cubic \'etale algebras ``tend'' to lie ``near'' to $(0,1/2)$ (see \cref{genericcubiccorollary}).
\end{enumerate} 
$(2)$ is also implied by Bhargava and H's work on the equidistribution of orders in $S_3$-cubic fields \cite{BhaHa16}, although we give an independent proof.

\subsection{Results on binary $n$-ic forms}
Birch and Merriman \cite{birch} gave a construction of rank $n$ ring $R_f$ from an integral binary $n$-ic form $f$. We now give partial results on successive minima of rings arising from binary $n$--ic forms. 

\begin{definition}
For an integer $n \geq 3$, a point $p=(p_1,\dots,p_{n-1}) \in \RR^{n-1}$, any positive real number $\eps$, and any $X \in \RR_{>1}$, define $\cF_n^{\bin}(p,\eps, X)$ to be the set of $\GL_2(\ZZ)$-orbits of $\Sym^n \ZZ^2$  of absolute discriminant $\Delta$ such that $0 < \Delta \leq X$ and
\[
	\max_{1 \leq i < n}\lv \log_{\Delta} \lambda_i(R_f) - p_i \rv \leq \frac{\eps}{\log X}.
\]
\end{definition}

Let
\[
	k \coloneqq \frac{1}{2(1 + \dots + (n-1))}
\]
and define $\ell$ to be the closed line segment from $(0,k)$ to $(k/2,k/2)$. Let $L \colon \ell \rightarrow \RR^{n-1}$ be the linear map sending the point $(0,k)$ to $(k,2k,\dots,(n-1)k)$ and sending the point $(k/2,k/2)$ to $\frac{1}{2(n-1)}(1,\dots,1)$.

\begin{theorem}
\label{binthm}
There exists a positive real number $\MO$ such that for all $\eps > \MO$ and for $r = (r_1,r_2)\in \ell$, we have
\[
	\#\cF_n^{\bin}(L(r),\eps, X) \gg_{r} X^{(n+1)/(2n-2) - (r_2-r_1)}.
\]
\end{theorem}

We expect the bound in \cref{binthm} to be sharp; for $n = 3,4,5$, the proofs of \cref{totdensitycubic}, \cref{fulldensitythmquartic}, and \cref{quintictotaldensity} respectively implies that
\[
\#\cF_n^{\bin}(L(r),\eps, X) \ll_r X^{(n+1)/(2n-2) - (r_2-r_1)}.
\]
for $\eps$ sufficiently large.

\subsection{Results on quartic rings}
\subsubsection{Setup}
Consider pairs $(R,C)$, where $R$ is a nondegenerate quartic ring equipped with a cubic resolvent ring $C$. Let $\Delta$ denote the absolute discriminant of $R$. We first define certain regions that will be the support of our density functions; these are the analogues of the line segments from the cubic case. 

\begin{definition}
For a permutation group $G \subseteq S_4$, let $\poly_4(G)$ be the limit points of the multiset 
\[
	\{(\log_{\Delta}\lambda_1(R),\log_{\Delta}\lambda_2(R),\log_{\Delta}\lambda_3(R),\log_{\Delta}\lambda_1(C),\log_{\Delta}\lambda_2(C))\} \subseteq \RR^{5}
\]
as we range over isomorphism classes of pairs $(R,C)$ where $R$ is a $G$-ring. Define $\poly_4$ to be $\cup_G \poly_4(G)$ as we range over all permutation groups $G \subseteq S_4$.     
\end{definition}

As the notation implies, $\poly_4(G)$ is a finite union of polytopes (\cref{containmentlemma}). The inequalities defining $\poly_4(S_4)$, $\poly_4(D_4)$, and $\poly_4$ are given in \cref{s4polytopethm}, \cref{d4polytopethm}, and \cref{c1polytopethm}. 

\subsubsection{Defining the density functions}

\begin{definition}
For any $p = (p_1,p_2,p_3,q_1,q_2) \in \RR^5$, any positive real number $\eps$, any $X \in 
\RR_{>1}$, and any permutation group $G \subseteq S_4$, let $\cF(p,\eps,G,X)$ be the set of isomorphism classes of pairs $(R,C)$ such that $R$ is a $G$-ring, $\Delta \leq X$, and
\begin{align}
\label{condquartic}
\max_{i = 1,2,3}\{\lv \log_{\Delta}\lambda_i(R) - p_i \rv\} &\leq \frac{\eps}{\log X} \\
\label{condquartic2}
\max_{i = 1,2}\{\lv \log_{\Delta}\lambda_i(C) - q_i \rv\} &\leq \frac{\eps}{\log X}.
\end{align}
\end{definition}

\begin{definition}
Let $d_{\eps,G} \colon \RR^5 \rightarrow \RR \cup \{\ast\}$ be the function which assigns to $p \in \RR^5$ the value 
\[
	\lim_{X \rightarrow \infty}\frac{\log (\#\cF(p,\eps,G,X)+1)}{\log X}
\]
if it exists and else assigns $\ast$. Call $d_{\eps,G}$ the density function. Let $d_{M,4}$ be the density function where we don't impose a Galois condition.
\end{definition}

\begin{figure}[!htb]
\minipage{0.32\textwidth}
 \includegraphics[width=\linewidth]{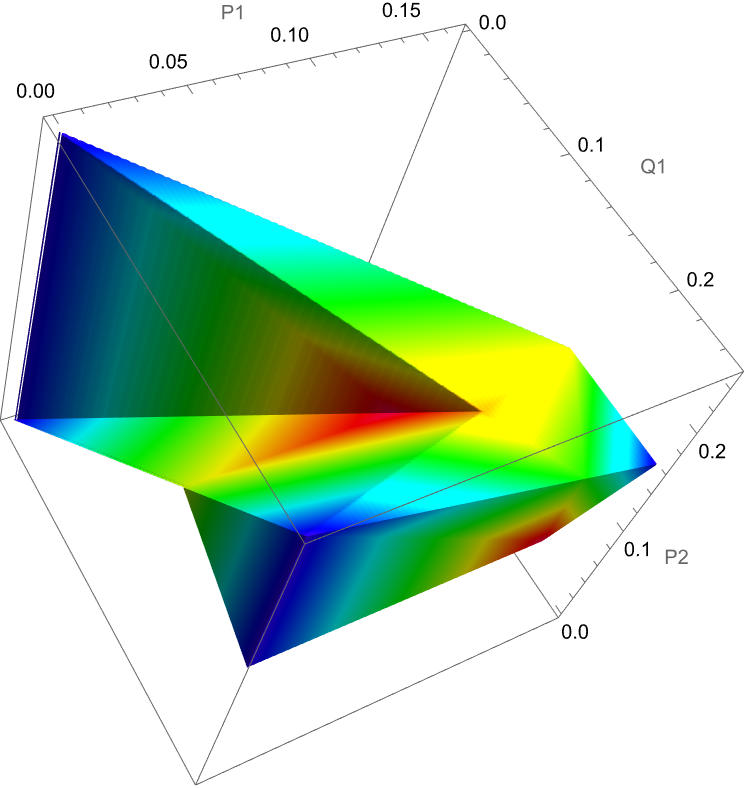}
 \caption{$\poly_4$ projected onto $p_1$, $p_2$, and $q_1$. The color indicates the value of the density function $d_{\eps,4}$; blue is the highest value.}\label{fig:awesome_image1}
\endminipage\hfill
\minipage{0.32\textwidth}
 \includegraphics[width=\linewidth]{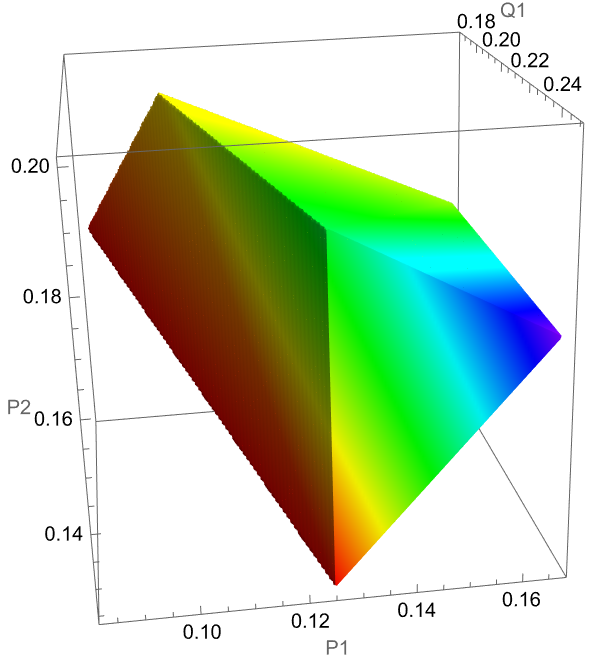}
 \caption{$\poly_4(S_4)$ projected onto $p_1$, $p_2$, and $q_1$. The color indicates the values of the density function $d_{\eps,S_4}$; blue is the highest value.}\label{fig:awesome_image2}
\endminipage\hfill
\minipage{0.32\textwidth}%
 \includegraphics[width=\linewidth]{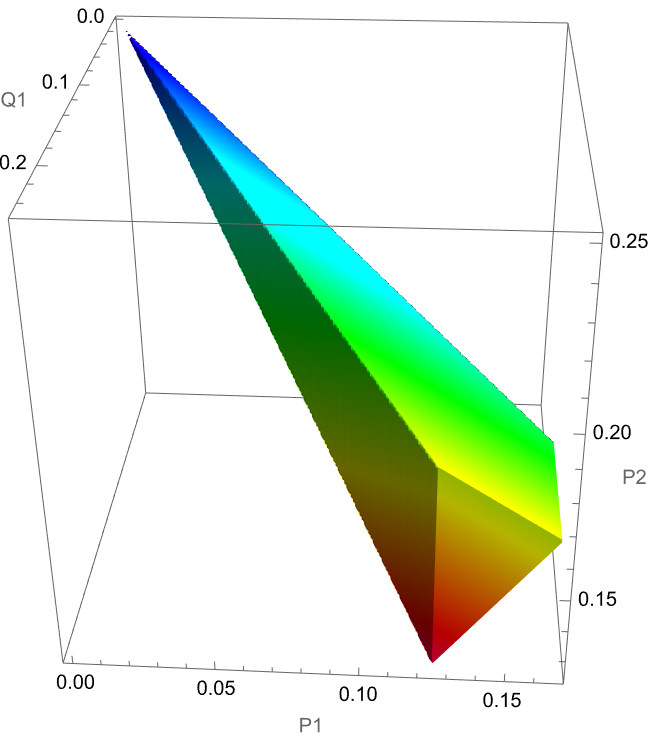}
 \caption{$\poly_4(D_4)$ projected onto $p_1$, $p_2$, and $q_1$. The color indicates the value of the lower bound on the density function $d_{\eps,D_4}$ given in \cref{d4densitythm}; blue is the highest value.}\label{fig:awesome_image3}
\endminipage
\end{figure}

\subsubsection{Theorem statements}

\begin{restatable}{theorem}{fulldensitythmquartic}
\label{fulldensitythmquartic} There exists a positive real number $\MO$ such that for every $\eps > \MO$ and every $p \in \RR^5$,
\[
    d_{\eps,4}(p)= 
\begin{cases}
    1 - (p_3 - p_2) - (p_3 - p_1) - (p_2 - p_1) - (q_2 - q_1) & \text{if $p \in \poly_4$ } \\
    \;\; \; +\sum_{i \leq j,k}\max\{0, q_k-p_i-p_j\} &  \\
    0              & \text{else.}
\end{cases}
\]
\end{restatable}

\begin{restatable}{theorem}{sfourdensitythm}
\label{s4densitythm} There exists a positive real number $\MO$ such that for every $\eps > \MO$ and every $p \in \RR^5$ we have
\[
    d_{\eps,S_4}(p) = 
\begin{cases}
    d_{\eps,4}(p) & \text{if $p \in \poly_4(S_4)$ } \\
    0              & \text{else.}
\end{cases}
\]
\end{restatable}

\begin{restatable}{theorem}{dfourdensitythm}
\label{d4densitythm} There exists a positive real number $\MO$ such that for every $\eps > \MO$ and every $p \in \RR^5$, the following statements hold. 
\begin{enumerate}
\item If $p \in \poly_4(D_4)\setminus \poly_4(S_4)$, then
\begin{align*}
d_{\eps,D_4}(p) = d_{\eps,4}(p) =& \; 1 - (p_3 - p_2) - (p_3 - p_1) - (p_2 - p_1) - (q_2 - q_1) \\
& + (q_2-2p_1) + (q_2-p_1-p_2) + (q_2-p_1-p_3). 
\end{align*}
\item If $p \notin \poly_4(D_4)$, then $d_{\eps,D_4}(p) = 0$. 
\item If $p \in \poly_4(D_4)$, then
\begin{align*}
\liminf_{X\rightarrow \infty} \frac{\log(\#\cF(p,\eps,D_4,X)+1)}{\log X} &\geq 1 - (p_3 - p_2) - (p_3 - p_1) - (p_2 - p_1) - (q_2 - q_1)\\
 &\;\;\;\;\; \;\;\;\; + (q_2-2p_1) + (q_2-p_1-p_2) + (q_2-p_1-p_3)
\end{align*}
\end{enumerate}
\end{restatable}

\begin{figure}[!htb]
\minipage{0.48\textwidth}
 \includegraphics[width=\linewidth]{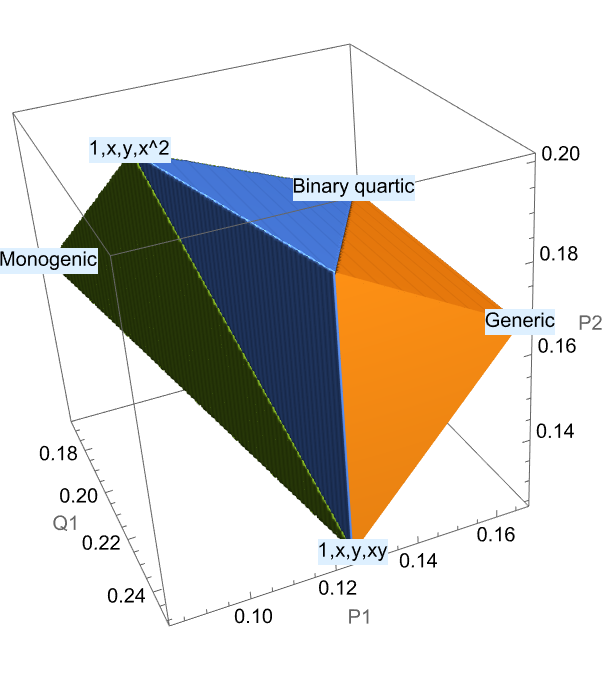}
 \caption{$\poly_4(S_4)$ projected onto $p_1$, $p_2$, and $q_1$. The colored regions are the regions on which $d_{\eps,S_4}$ is linear. The vertices are labelled by the family of rings ``corresponding'' to them.}\label{fig:awesome_image1}
\endminipage\hfill
\minipage{0.48\textwidth}
 \includegraphics[width=\linewidth]{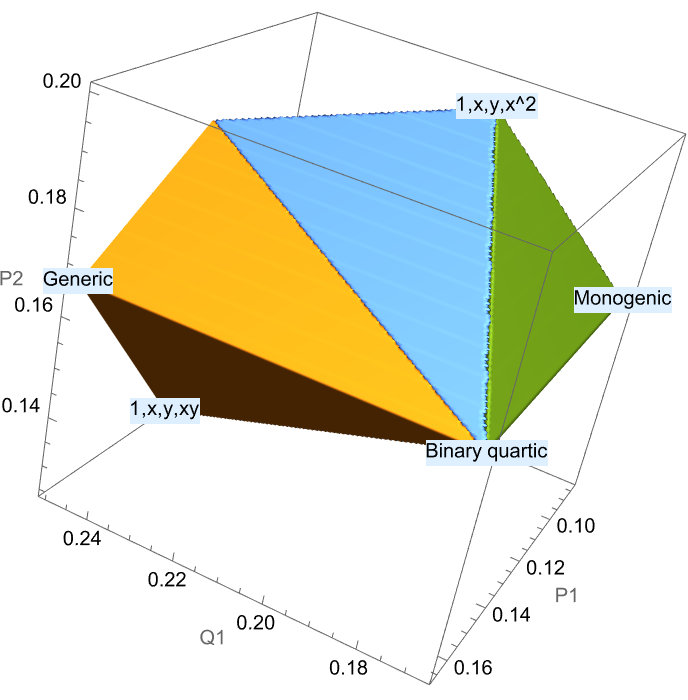}
 \caption{The other side of $\poly_4(S_4)$ projected onto $p_1$, $p_2$, and $q_1$.}\label{fig:awesome_image2}
\endminipage
\end{figure}

\subsubsection{Observations}
First, as in the cubic case and quintic case, $d_{\eps,4}$ has a very particular structure; it is $1$ minus piecewise differences of the coordinates (the terms $p_j-p_i$ and $q_j-q_i$) plus correction terms. 

Next, as in the cubic case, we observe \emph{algebraic} phenomena at the vertices of these convex bodies. In particular, $\poly_4(S_4)$ is a polytope with $6$ vertices. We give natural interpretations for $5$ of these $6$ vertices:
\begin{enumerate}
\item $S_4$-rings ``correspond'' to $(1/6,1/6,1/6,1/4,1/4)$ by \cref{genericgaloisgroupscorollary};
\item monogenic quartic rings  ``correspond'' to $(1/12,1/6,1/4,1/6,1/3)$ by \cref{monogenicquartictheorem};
\item quartic rings arising from binary quartic forms ``correspond'' to
\[
	(1/6,1/6,1/6,1/6,1/3)
\] 
by \cref{binaryquartictheorem};
\item quartic rings with a basis of the form $1,x,y,xy$ with a pairing condition ``correspond'' to $(1/8,1/8,1/4,1/4,1/4)$ by \cref{1xyxythm};
\item and quartic rings with a basis of the form $1,x,y,x^2$ ``correspond'' to $(1/10,1/5,1/5,1/5,3/10)$ by \cref{1xyx2thm}.
\end{enumerate}

The vertex missing from this list is $(1/8,3/16,3/16,1/4,1/4)$; it would be interesting to have a family of quartic rings naturally corresponding to this vertex! Moreover, one can show that $D_4$-rings ``correspond'' to $(0,1/4,1/4,0,1/2)$ by \cref{genericgaloisgroupscorollary}, and this point is the unique vertex of $\poly_4(D_4)$ that is not a vertex of $\poly_4(S_4)$.

\subsection{Results on quintic rings}
\subsubsection{Setup}
Consider pairs $(R,S)$ where $R$ is a nondegenerate quintic ring equipped with a sextic resolvent ring $S$; see \cref{quinticbackground} for background. Let $\Delta$ be the absolute discriminant of $R$. We proceed as in the quartic case; we first define certain regions that will be the support of our density functions, then compute our density functions, and then discuss special families of quintic rings. 

\begin{definition}
For a permutation group $G \subseteq S_5$, define $\poly_5(G)$ to be the limit points of the multiset 
\[
	\{(\log_{\Delta}\lambda_1(R),\dots,\log_{\Delta}\lambda_4(R),\log_{\Delta}\lambda_1(S),\dots,\log_{\Delta}\lambda_5(S))\} \subseteq \RR^{9}
\]
in the Euclidean metric as we range over all pairs $(R,S)$ where $R$ is a $G$-ring.    
\end{definition}

The region $\poly_5(S_5)$ is a polytope (\cref{s5polytopetheorem}). For a positive number $\eps$ and a permutation group $G \subseteq S_5$, define the density functions $d_{\eps,G}$,  and $d_{\eps,5}$ analogously to the quartic case; in this case they are functions from $\RR^9$ to $\RR \cup \{\ast\}$; see \cref{quintic} for more information. Let $p = (p_1,\dots,p_4,q_1,\dots,q_5) \in \RR^9$ denote a point.

\begin{figure}[!htb]
\captionsetup{width=.3\linewidth}
\minipage{0.32\textwidth}
 \includegraphics[width=\linewidth]{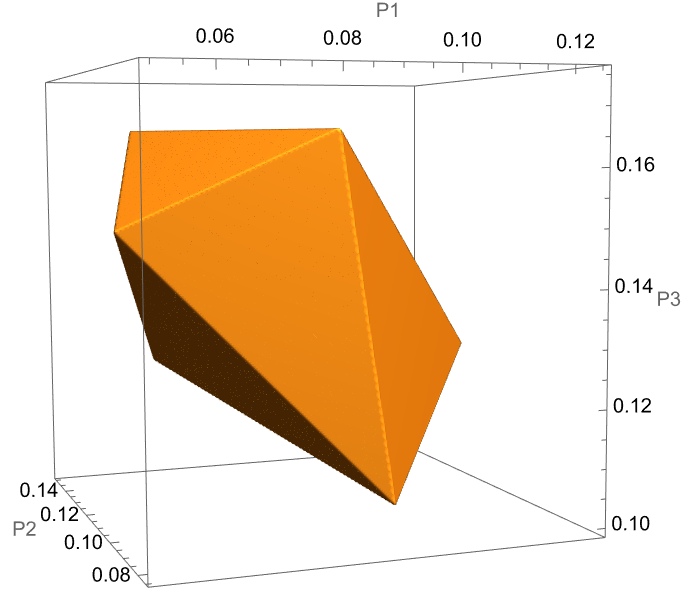}
 \caption{$\poly_5(S_5)$ projected onto $p_1$, $p_2$, and $p_3$.}\label{fig:awesome_image1}
\endminipage\hfill
\minipage{0.32\textwidth}
 \includegraphics[width=\linewidth]{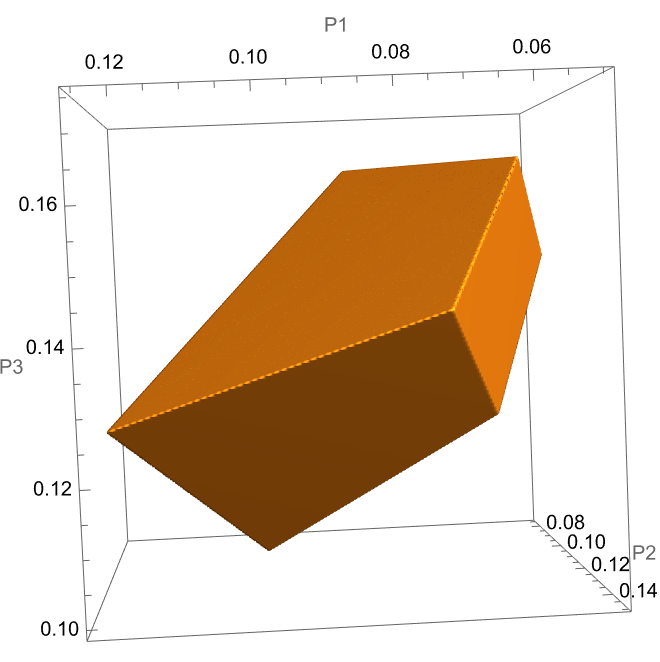}
 \caption{The other side of $\poly_5(S_5)$ projected onto $p_1$, $p_2$, and $p_3$.}\label{fig:awesome_image2}
\endminipage \hfill
\minipage{0.32\textwidth}
\includegraphics[width=\linewidth]{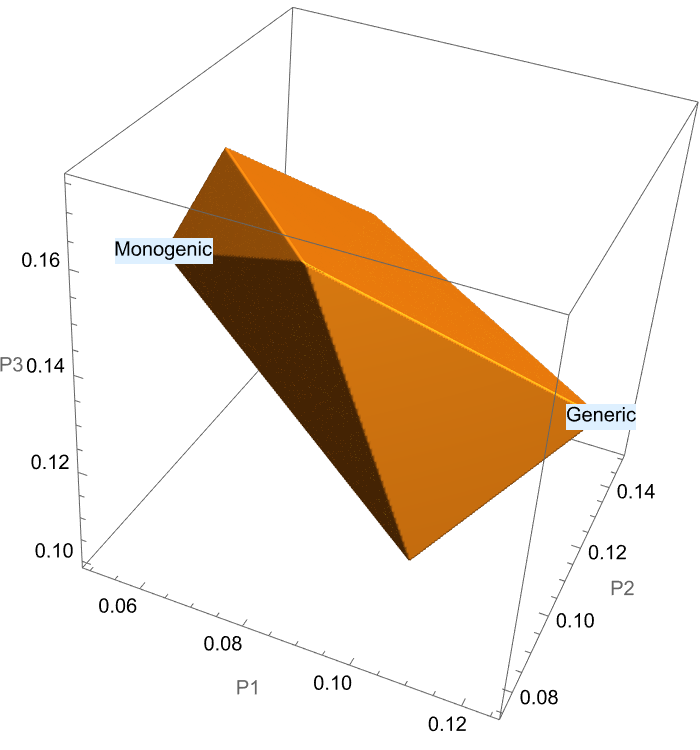}
 \caption{$\poly_5(S_5)$ projected to $p_1$, $p_2$, and $p_3$. The image of the projection of the vertex corresponding to $S_5$-rings is marked ``generic''. The projection of the vertex corresponding to monogenic rings is marked ``monogenic''.}
\endminipage

\end{figure}
\subsubsection{Theorem statements}
\begin{restatable}{theorem}{quintictotaldensity}
\label{quintictotaldensity}
There exists a positive real number $\MO$ such that for every $\eps > \MO$ and every $p \in \RR^9$,
\[
    d_{\eps,5}(p) =
\begin{cases}
    1 - \sum_{1 \leq i<j\leq 4}(p_j - p_i) - \sum_{1 \leq i<j\leq 5}(q_j - q_i) & \text{if $p \in \poly_5$ } \\
    \;\; \; +\sum_{i<j,k}\max\{0, -1/2+q_i+q_j-p_k\}&  \\
    0              & \text{else.}
\end{cases}
\]
\end{restatable}

\begin{restatable}{theorem}{sfivedensitythm}
\label{s5densitythm}
There exists a positive real number $\MO$ such that for every $\eps > \MO$ and every $p \in \RR^9$,
\[
    d_{\eps,S_5}(p) =
\begin{cases}
    d_{\eps,5}(p) & \text{if $p \in \poly_5(S_5)$ } \\
    0              & \text{else.}
\end{cases}
\]
\end{restatable}

\subsubsection{Observations}
First, observe that the structure of $d_{\eps,5}$ is similar to that of $d_{\eps,3}$ and $d_{\eps,4}$; from $1$ we subtract piecewise differences ($\sum_{1 \leq i<j\leq 4}(p_j - p_i)$ and $\sum_{1 \leq i<j\leq 5}(q_j - q_i)$); and in the ``cusp'' we add in correction terms ($\sum_{i<j,k}\max\{0, -1/2+q_i+q_j-p_k\}$).

Next, the definitions of the density functions and $\poly_5(G)$ are purely Euclidean, but we observe \emph{algebraic} phenomena along these convex bodies. In particular:
\begin{enumerate}
\item $S_5$-rings ``correspond'' to the vertex
\[
	(1/8,1/8,1/8,1/8,3/10,3/10,3/10,3/10,3/10)
\]
of $\poly_5(S_5)$ by \cref{generics5thm};
\item monogenic quintic rings ``correspond'' to the vertex 
\[
	(1/20,2/20,3/20,4/20,4/20,5/20,6/20,7/20,8/20)
\] 
of $\poly_5(S_5)$ by \cref{monogenicquintictheorem};
\item and quintic rings arising from binary quintic forms ``correspond'' to the vertex 
\[
	(1/8, 1/8, 1/8, 1/8, 1/4, 1/4, 1/4, 3/8, 3/8)
\] 
of $\poly_5(S_5)$  by \cref{binaryquintictheorem}.
\end{enumerate}

There are $18$ vertices of $\poly_5(S_5)$; it would be interesting to give interpretations of the remaining $15$ vertices!

\subsection*{Acknowledgments}
\noindent The author is very grateful to Manjul Bhargava and Ravi Vakil for countless helpful discussions throughout this project, Jordan Ellenberg for posing the question that inspired this project, Arul Shankar for helpful discussions about sieving for maximality, Akshay Venkatesh for suggesting directions for future work, and Robert Lemke-Oliver for helpful discussions about the $D_4$-case. Thank you as well to Ravi Vakil, Hannah Larson, Anand Patel, and a very helpful referee for some comments on an earlier draft of this manuscript. The author was supported by the NSF, under the Graduate Research Fellowship and an NSF Mathematical Sciences Postdoctoral Fellowship.

\section{Setup}
\label{notation}
In this section, we introduce the setup and notation used throughout this paper. Let $n \geq 2$ be any integer. A \emph{rank $n$ ring} is a commutative unital ring that is isomorphic to $\ZZ^n$ as a $\ZZ$-module. We attach to any such ring $R$ the trace function $\Tr \colon R \rightarrow \ZZ$ which assigns to any element $\alpha \in R$ the trace of the endomorphism $R \xrightarrow{\times \alpha} R$. The discriminant $\Disc(R)$ of any such ring $R$ is defined to be the determinant $\det(\Tr(\alpha_i\alpha_j)) \in \ZZ$, where $\{\alpha_i\}_{i = 1}^n$ is any $\ZZ$-basis of $R$. We say a ring of rank $n$ is \emph{nondegenerate} if its discriminant is nonzero.  Let $R$ be a nondegenerate rank $n$ ring.

We now define the successive minima of $R$. The ring $R$ has $n$ homomorphisms into $\CC$. Denote these by $\sigma_1,\dots,\sigma_{n}$, and define a quadratic form $q \colon R \rightarrow \ZZ$ by 
\[
	q(x) \coloneqq \frac{1}{n}\sum_{i = 1}^{n} \lvert \sigma_i(x)\rvert^2
\]
Observe that $q$ is positive definite and thus makes $R$ a lattice. Thus, we obtain successive minima $\lambda_0(R),\dots,\lambda_{n-1}(R)$. What constraints exist on the successive minima of nondegenerate rank $n$ rings? By definition, we have that $0 \leq \lambda_0 \leq \lambda_1 \leq \dots \leq \lambda_{n-1}$. There are two additional constraints that are not too difficult to show.

\begin{lemma}
\label{lambdazerolemma}
For any nondegenerate rank $n$ ring $R$, we have $1/\sqrt{n} \leq \lambda_{0} \leq 1$. If $R$ is an order in a number field, then $\lambda_0 = 1$.
\end{lemma}
\begin{proof}
Note that $\lvert 1 \rvert = 1$ so $\lambda_0 \leq 1$. Suppose first that $R$ is an order in a number field. Let $\sigma_1,\dots,\sigma_{r_1}$ be the nonzero homomorphisms of $R$ into the real numbers and let $\sigma_{r_1 + 1},\overline{\sigma_{r_1 + 1}},\dots,\sigma_{r_1 + r_2},\overline{\sigma_{r_1 + r_2}}$ be the nonzero homomorphisms of $R$ into the complex numbers whose image does not lie in the real numbers. Then for any nonzero $x \in R$:
\begin{align*}
\lvert x \rvert^2 &= \frac{1}{n}\bigg(\sum_{i = 1}^{r_1} \sigma_i(x)^2 + \sum_{i = r_1 + 1}^{r_1 + r_2} 2\lvert\sigma_i(x)\rvert^2\bigg) \\
&\geq \sqrt[n]{\prod_{i = 1}^{r_1} \sigma_i(x)^2\prod_{i = r_1 + 1}^{r_1 + r_2} \lvert\sigma_i(x)\rvert^4} && \text{ by the AM-GM inequality} \\
&= \sqrt[n]{\prod_{i = 1}^{r_1} \sigma_i(x)^2\prod_{i = r_1 + 1}^{r_1 + r_2} \sigma_i(x)^2\overline{\sigma_i}(x)^2} \\
&= \lvert N_{R\otimes\QQ/\QQ}(x)\rvert^{2/n}  \\
& \geq 1,
\end{align*}
Thus, $\lambda_0 = 1$ in this case.

Now suppose $R$ is any nondegenerate rank $n$ ring. Write $R \otimes \QQ$ as $\prod_{i = 1}^t K_i$ where $K_i$ are all fields. For any $x \in R$, write $x = (x_i) \in \prod_i K_i$. If $x$ is nonzero then:
\[
	n\lv x \rv^2 = \sum_{i = 1}^t \deg(K_i)\lv x_i \rv^2 \geq \sum_{i = 1}^t \lv x_i \rv^2 \geq 1.
\]
\end{proof}

There is an additional serious constraint on successive minima of lattices, coming from Minkowski's second theorem (see \cite{siegel}, Lecture 10, \S 6), which states that
\[
\prod_{i = 0}^{n-1} \lambda_i \asymp_n \lv \Disc(R)\rv.
\]

\section{Successive Minima of cubic rings}
\label{cubic}

This section is organized as follows. \cref{subsec:special-family-additional-results} states some additional results on cubic rings. In \cref{cubicbackground} to \cref{subsec:step-5-cubic}, we prove all of our theorems about cubic rings. More precisely: 
\begin{itemize}
    \item in \cref{cubicbackground}, we go over Delone and Faddeev's parametrization of cubic rings using binary cubic forms;
    \item in \cref{cubicreductiontheory}, we discuss how binary cubic forms with ``small'' coefficients correspond to cubic rings with ``short'' bases;
    \item in \Cref{subsec:density-functions-cubic-prelim}, we combine the previous two sections to get upper and lower bounds on the number of cubic rings with almost prescribed successive minima;
    \item in \cref{subsec:density-functions-cubic}, we begin to compute density functions and prove \cref{chichecubic};
    \item in \cref{subsec:correspondences-cubic}, we discuss special families of cubic rings and prove \cref{genericcubiccorollary} and \cref{monogeniccubictheorem};
    \item and in \cref{subsec:step-5-cubic}, we sieve for maximal orders in cubic fields, and prove  \cref{s3maxthm} and \cref{totdensitycubic}.
\end{itemize}

\subsection{Additional results on special families of cubic rings}
\label{subsec:special-family-additional-results}

Remarkably, there is a relationship between rings that have extremal successive minima and rings that are extremal in an \emph{algebraic} sense, as we will see in the following theorems. Via this relationship, the endpoints of our line segments ``correspond'' to various interesting algebraic families of rings.

\subsubsection{Definition of a family of orders corresponding to a point}

We'll now formalize what ``correspond'' means. Let $\cF$ be an infinite set equipped with a map $H \colon \cF \rightarrow \RR_{> 0}$ such that for all $X$, the set $\{f \in \cF \mid H(f) \leq X \}$ is finite. We call $H$ the \emph{height function} of $\cF$. Further suppose that $\cF$ is also equipped with a map $\psi$ from $\cF$ to the set of isomorphism classes of nondegenerate cubic rings\footnote{As a motivating example, consider the case when $\cF$ is the set of integral binary cubic forms of nonzero discriminant and $H$ is the naive height, i.e. the maximum of the absolute value of the coefficients.}.

\begin{definition}
For $\eps \in \RR_{>0}$ and $X \in \RR_{>1}$, we say a cubic ring $R$ and a point $p \in \RR^2$ are \emph{$(\eps,X)$-close} if
\[
	\max_{i = 1,2}\{\lv \log_{\Delta}\lambda_i - p_i \rv \} \leq \frac{\eps}{\log X}.
\] 
\end{definition}

\begin{definition}
For a point $p$ contained in the closed line segment from $(0,1/2)$ to $(1/4,1/4)$, we say the triple \emph{$(\cF,H,\psi)$ corresponds to $p$} if 
\begin{equation}
\label{cond1cubicspecialfamily}
\lim_{\eps \rightarrow \infty}\liminf_{X \rightarrow \infty}\frac{\#\{f \in \cF \; \mid \; H(f) \leq X \text{ and $\psi(f)$ is $(\eps,X)$-close to $p$}\}}{\#\{f \in \cF \; \mid \; H(f) \leq X \}} = 1
\end{equation}
and there exists a constant $\MO$ such that for all $\eps > \MO$, 
\begin{equation}
\label{cond2cubicspecialfamily}
\liminf_{X \rightarrow \infty}\frac{\#\{R \in \cF(p,\eps, X) \; \mid \; \text{$R \in \psi(\cF)$}\}}{\#\cF(p,\eps, X)} > 0.
\end{equation}
\end{definition}

When $H$ or $\psi$ are implicit, we simply write that $\cF$ corresponds to $p$. Heuristically, \eqref{cond1cubicspecialfamily} says that ``most'' elements of $\cF$, when ordered by height, are ``near'' $p$, and \eqref{cond2cubicspecialfamily} says that when ordered by absolute discriminant, a positive proportion of rings ``near'' $p$ come from $\cF$. 

\subsubsection{Theorem statements}
Let $\cF_n^{\mon}$ be the set of monic degree $n$ polynomials with integral coefficients. It has a natural height function: given a polynomial
\[
	f(x) = x^n + a_{1}x^{n-1} + \dots + a_n \in \ZZ[x],
\]
the \emph{root height of $f$} is denoted $H_r(f)$ and is defined to be $\max_i\{\lv a_i\rv^{1/i}\}$. Each $f \in \cF^n_{\mon}$ naturally gives rise to the ring $R_f \coloneqq \ZZ[x]/(f(x))$. As the following corollary shows, these rings tend to be quite skew in the cubic case.

\begin{corollary}
\label{monogeniccubictheorem}
$(\cF_3^{\mon}, H_r)$ corresponds to $(1/6,1/3)$.
\end{corollary}

Our next result shows that a general irreducible cubic ring is not very skew at all, and a general reducible ring is as skew as possible.

\begin{corollary}
\label{genericcubiccorollary}
Let $\cF^{\irr}$ be the set of isomorphism classes of nondegenerate irreducible cubic rings, and let $H$ be the absolute discriminant. Then $(\cF^{\irr},H)$ corresponds to $(1/4,1/4)$. Let $\cF^{\redu}$ be the set of isomorphism classes of nondegenerate reducible cubic rings, and again let $H$ be the absolute discriminant. Then $(\cF^{\redu},H)$ corresponds to $(0,1/2)$.
\end{corollary}

The statement above about $\cF^{\irr}$ is in fact a corollary of Terr's work \cite{Terr97} on the equidistribution of shapes of cubic orders, but we give an independent proof.

\subsection{The parametrization of cubic rings}
\label{cubicbackground}
An element $\gamma \in \GL_2(\ZZ)$ acts on a binary cubic form $f(x,y)$ by 
\[
	(\gamma f)(x,y) = \frac{1}{\det \gamma}f((x,y)\gamma).
\]
A \emph{based rank $n$ ring} is a rank $n$ ring $R$ equipped with an ordered $\ZZ$-basis of $R/\ZZ$. The group $\GL_2(\ZZ)$ naturally acts on based cubic rings by change of basis.

\begin{theorem}[Delone, Faddeev \cite{delone}]
There is a canonical discriminant--preserving bijection between the set of based cubic rings and integral binary cubic forms. Under the action of $\GL_2(\ZZ)$, this bijection descends to a bijection between the set of isomorphism classes of cubic rings and the set of $\GL_2(\ZZ)$-orbits of integral binary cubic forms.
\end{theorem}

We now describe the Delone--Faddeev bijection. Given a cubic ring $R$ with basis $\{1,\omega,\theta\}$ and an element $\alpha \in R$, let $L_{\alpha}$ be a $3\times 3$ matrix such that
\[
\begin{bmatrix}
1 & \alpha & \alpha^2\\
\end{bmatrix} = L_{\alpha}\begin{bmatrix}
1 & \omega & \theta\\
\end{bmatrix}^t
\] 
One may define a map $\varphi \colon R \rightarrow \ZZ$ given by sending $\alpha$ to $\det L_{\alpha}$. This descends to a map $\varphi \colon R/\ZZ \rightarrow \ZZ$. Upon choice of basis of $R/\ZZ$, the map $\varphi$ is represented by an integral binary cubic form.

Conversely, from an integral binary cubic form $f(x,y) = ax^3 + bx^2y + cxy^2 + dy^3$, we obtain a cubic ring $R_f$ with basis $\{1,\omega,\theta\}$ defined by the following multiplication table:
\begin{equation}
 \label{cubiccoeff}
\begin{split}
\omega\theta &= -ad \\
\omega^2 &= -ac + b\omega - a\theta\\
\theta^2 &= -bd + d\omega - c\theta
\end{split}
\end{equation}

\subsection{Reduction theory of cubic rings}
\label{cubicreductiontheory}
In this section, we develop the relationship between ``small'' bases of cubic rings and integral binary cubic forms with ``small'' coefficients. Let $V_{\RR}$ denote the set of real binary cubic forms and let $V_{\ZZ} \subseteq V_{\RR}$ denote the lattice of integral binary cubic forms.

\begin{definition}
For a point $p=(p_1,p_2) \in \RR^2$ and a positive real number $X$, define the box $B(p,X)$ to be the set of real binary cubic forms $ax^3 + bx^2y + cxy^2 + dy^3$ such that:
\begin{align*}
\lv a \rv &\leq X^{3p_1 - 1/2}\\
\lv b \rv &\leq X^{2p_1 + p_2 - 1/2}\\
\lv c \rv &\leq X^{p_1 + 2p_2 - 1/2}\\
\lv d \rv &\leq X^{3p_2 - 1/2}
\end{align*}
\end{definition}

We now fix some notation that we use throughout this section. Let
\[
	f(x,y) = ax^3 + bx^2y + cxy^2 + dy^3 \in V_{\ZZ}
\]
be a form of nonzero discriminant, let $R_f$ denote the cubic ring associated to $f$ under the Delone-Faddeev correspondence, let $\Delta = \lv \Disc(f) \rv = \lv \Disc(R_f) \rv$, let $p = (p_1,p_2)$ denote a point in $\RR^2$, let $\eps$ denote a positive real number, let $C,D$ denote real numbers $\geq 1$, and let $X \in \RR_{>1}$.

\begin{definition}
We say $f$ \emph{comes from a Minkowski basis} if there exists a Minkowski basis of $R_f$ giving rise to $f$ under the Delone--Faddeev correspondence.    
\end{definition}

See \cite{siegel}, Lecture 10, \S 6 for a definition of Minkowski basis. Recall the following definition:

\begin{definition}
For $\eps \in \RR_{>0}$ and $X \in \RR_{>1}$, we say a cubic ring $R$ and a point $p \in \RR^2$ are \emph{$(\eps,X)$-close} if
\[
	\max_{i = 1,2}\{\lv \log_{\Delta}\lambda_i - p_i \rv \} \leq \frac{\eps}{\log X}.
\] 
\end{definition}

We say a pair $(f,p)$ is \emph{$(\eps,X)$-close} if $R_f$ and $p$ are $(\eps,X)$-close.

\subsubsection{Binary cubic forms with small coefficients correspond to short bases of cubic rings}
In this subsection, we prove three lemmas which illustrate the relationship between binary cubic forms with ``small'' coefficients and ``short'' bases of cubic rings. Our first lemma shows that Minkowski bases of cubic rings give rise to binary cubic forms with ``small'' coefficients. 

\begin{lemma}
\label{minkbasisshortcubic}
There exists $C$ such that for all $\eps$, $f$, and $p$, if $(f,p)$ is $(\eps,\Delta)$-close and $f$ comes from a Minkowski basis, then $f \in Ce^{3\eps} B(p,\Delta)$.
\end{lemma}
\begin{proof}
Suppose $(f,p)$ be an $(\eps,\Delta)$-close pair coming from a Minkowski basis $\{1,\omega,\theta\}$ of $R_f$. Minkowski's second theorem implies that $\lv \omega \rv \asymp \lambda_1(R_f)$ (see \cite{siegel}, Lecture 10, \S 6 for a reference), and the definition of  $(\eps,\Delta)$-closeness implies that $\lambda_1(R_f) \leq e^{\eps} \Delta^{p_1}$. Therefore,
\[
	\lv \omega \rv \asymp \lambda_1(R_f) \leq e^{\eps} \Delta^{p_1}.
\]
Similarly, we have:
\[
	\lv \theta \rv \asymp \lambda_2(R_f) \leq e^{\eps} \Delta^{p_2}.
\]
Recall that $\lvert \omega\rvert$ and $\lvert \theta\rvert$ refer to the lengths of $\omega$ and $\theta$ respectively in the Minkowski embedding; see \Cref{notation} for details.

Let $\sigma_1,\sigma_2,\sigma_3$ be the $3$ homomorphisms of $R_f$ into the complex numbers. Recall that for any $\alpha \in R$, we have that $M_{\alpha}$ is the matrix sending the basis ${1,\omega,\theta}$ to ${1,\alpha,\alpha^2}$. By the Delone-Faddeev correspondence, we have:
\[
    f(x,y) =   \det L_{x\omega + y\theta} = \Delta^{-1/2}\begin{array}{|c c c|}
1 & \dots & 1\\
x\sigma_1(\omega) + y\sigma_1(\theta) & \dots & x\sigma_3(\omega) + y\sigma_3(\theta) \\
(x\sigma_1(\omega) + y\sigma_1(\theta))^2 & \dots & (x\sigma_3(\omega) + y\sigma_3(\theta))^2
\end{array} 
\]
Expand out and collect monomials, now treating $x$ and $y$ as variables. Because $\lv \sigma_i(\omega) \rv \ll \lv \omega \rv$ and $\lv \sigma_i(\theta) \rv \ll \lv \theta \rv$, the absolute value of the coefficient of $x^iy^{3-i}$ in the expression above is $\ll \lv \omega\rv^i\lv \theta\rv^{3-i} \Delta^{-1/2} $.  Therefore:
\begin{align*}
\lv a \rv &\ll \lv \omega \rv^3\Delta^{-1/2} \ll e^{3\eps}\Delta^{3p_1 - 1/2}\\
\lv b \rv &\ll \lv \omega \rv^2 \lv \theta \rv \Delta^{-1/2}  \ll e^{3\eps} \Delta^{2p_1 + p_2 - 1/2}\\
\lv c \rv &\ll\lv \omega \rv \lv \theta \rv^2\Delta^{-1/2} \ll e^{3\eps} \Delta^{p_1 + 2p_2 - 1/2}\\
\lv d \rv &\ll \lv \theta \rv^3\Delta^{-1/2} \ll e^{3\eps} \Delta^{3p_2 - 1/2}
\end{align*}
\end{proof}

For a form $f$, let $\{1,\omega,\theta\}$ denote the basis of $R_f$ specified by the explicit multiplication table \eqref{cubiccoeff}. The following lemma shows if $f$ has ``small'' coefficients, then $\{1,\omega,\theta\}$ is a ``small'' basis. Let $\flinec$ be the closed line segment from $(0,1/2)$ to $(1/4,1/4)$.
\begin{lemma}
\label{reducedcubicbasis}
For every $p \in \flinec$, every $C, X$, and every $f$, if $f \in CB(p,X)$ then $\lv \omega \rv \ll_C X^{p_1}$ and $\lv \theta \rv \ll_C X^{p_2}$.
\end{lemma}
\begin{proof}
Writing out the explicit multiplication coefficients, we obtain:
\begin{align*}
\lv \omega\rv^2 &\ll \lv \omega^2 \rv \\
&\ll_C \max\{\lv ac \rv, \lv b\rv\lv \omega\rv, \lv a\rv\lv \theta\rv\}  \\
&\ll_C \max\{X^{2p_1}, X^{p_1}\lv \omega\rv, X^{2p_1 - p_2}\lv \theta\rv\} 
\end{align*}
\begin{align*}
\lv \theta\rv^2 &\ll \lv \theta^2 \rv \\
&\ll_C \max\{\lv bd \rv, \lv d\rv\lv \omega\rv, \lv c\rv\lv \theta\rv\} \\
&\ll_C \max\{X^{2p_2}, X^{2p_2 - p_1}\lv \omega\rv, X^{p_2}\lv \theta\rv\}
\end{align*}
Applying \cref{generalboundlemma} completes the proof.
\end{proof}

The following lemma explicitly bounds the successive minima of $R_f$ for $f$ with ``small'' coefficients.
\begin{lemma}
\label{lambdaboundlemma}
For every $C,X$, every $p \in \flinec$, and every $f$, if $f \in CB(p,X)$ and $\frac{X}{2} \leq \Delta$ then 
\[
	 \lambda_i(R_f) \asymp_C \Delta^{p_i}
\]
for $i = 1,2$.
\end{lemma}
\begin{proof}
We have:
\begin{align*}
\lambda_1(R)\lambda_2(R) &\ll \lv \omega\rv\lv \theta\rv && \text{by definition of successive minima} \\
&\ll_C X^{p_1}X^{p_2} && \text{by \cref{reducedcubicbasis}}\\
&\ll \Delta^{p_1}\Delta^{p_2} && \text{because $X/2 \leq \Delta$}\\
&\ll \Delta^{1/2} && \text{because $p_1 + p_2 =1/2$} \\
&\asymp \lambda_1(R)\lambda_2(R) && \text{by Minkowski's second theorem}
\end{align*}
Thus, 
\[
	 \lambda_i(R_f) \asymp_C \Delta^{p_i}
\]
for $i = 1,2$.
\end{proof}

\subsubsection{The number of binary cubics with small coefficients giving rise to isomorphic cubic rings}
In this subsection, we calculate how many binary cubic forms with ``small'' coefficients give rise to isomorphic cubic rings. We begin with a lower bound on the number of distinct binary cubic forms with ``small'' coefficients giving rise to isomorphic cubic rings.
\begin{lemma}
\label{helplemma}
There exists $D$ such that for every $p \in \flinec$, every $f$, and every $C,X$, if $f \in CB(p,X)$ then $\#\{f' \in DCB(p,X) \cap V_{\ZZ} \mid R_{f'} \simeq R_{f}\} \geq  X^{p_2-p_1}$.
\end{lemma}
\begin{proof}
Act on $f$ by elements of $\GL_2(\ZZ)$ of the form
\[
	\begin{bmatrix}
    1 & 0  \\
    u & 1 
  \end{bmatrix} 
\]
where $\lv u \rv \leq X^{p_2-p_1}$.
\end{proof}

Our next lemma shows an upper bound on the number of distinct binary cubic forms with ``small'' coefficients giving rise to isomorphic cubic rings. 
\begin{lemma}
\label{reducedcubiccorol}
For every $C,X$, every $f$, and every $p \in \flinec$, if $f \in CB(p,X)$ and $X/2 \leq \Delta \leq X$, there are $\ll_C X^{p_2-p_1}$ distinct binary cubic forms $f' \in CB(p,X)$ such that $R_{f'} \simeq R_f$.
\end{lemma}
\begin{proof}
Let $\{1,\omega,\theta\}$ be the basis of $R_f$ specified by \eqref{cubiccoeff}. We have:
\begin{align*}
X^{1/2} &\asymp \Delta^{1/2} && \text{because $X/2 \leq \Delta \leq X$} \\
&\ll \lv\omega\rv\lv\theta\rv && \\
&\ll_C X^{p_1 + p_2} && \text{by \cref{reducedcubicbasis}}\\
&= X^{1/2}
\end{align*}
Therefore, $\lv \omega \rv\lv \theta \rv \asymp_C \Delta^{1/2}$.

Suppose $f'\in CB(p,X)$ is an integral binary cubic form such that $R_{f'} \simeq R_f$. Let $\{1,\omega',\theta'\}$ denote the basis corresponding to $f'$. Then \cref{reducedcubicbasis} implies that $\lv \omega' \rv \ll_C X^{p_1}$ and $\lv \theta' \rv \ll_C X^{p_2}$. Express $\omega'$ and $\theta'$ in terms of the basis $\{1,\omega,\theta\}$, i.e. write 
\begin{align*}
\omega' &= a + b\omega + c\theta \\
\theta' &= d + e\omega + f\theta
\end{align*} 
where $a,b,c,d,e,f \in \ZZ$.

Because $\lv \omega \rv \lv \theta \rv \asymp_C \Delta^{1/2}$, we have:
\begin{align*}
\lv \omega' \rv &\asymp_C \max\{\lv a \rv, \lv b \rv\lv\omega\rv, \lv c \rv\lv \theta \rv\} \\
\lv \theta' \rv &\asymp_C \max\{\lv d \rv, \lv e \rv \lv\omega\rv, \lv f \rv\lv \theta \rv\}
\end{align*}
The integral binary cubic form $f'$ is determined by the choice of $b$, $c$, $e$, and $f$. The number of choices of $b$, $c$, and $f$ is $\ll_C 1$. The number of choices of $e$ is $\ll_C X^{p_2-p_1}$. Therefore there are $\ll_C X^{p_2-p_1}$ distinct integral binary cubic forms $f' \in CB(p,X)$ such that $R_{f'} \simeq R_f$.
\end{proof}

\subsection{Upper and lower bounds}
\label{subsec:density-functions-cubic-prelim}

The main idea of this section is as follows. We proved in the previous two subsections that, roughly speaking, the integral points of $B(p,X)$ should correspond to binary cubic forms $f$ whose corresponding rings $R_f$ have normalized successive minima vectors ``near'' $p$ and absolute discriminant $\leq X$. Each ring $R$ arising this way has approximately $X^{p_2-p_1}$ forms $f \in B(p,X)$ for which $R_f \simeq R$. In this subsection, we combine these two facts together.

\subsubsection{Using Davenport's lemma to count lattice points of $B(p,X)$}
The main tool we'll use to count integral points of $B(p,X)$ is Davenport's lemma, stated below. For positive real numbers $C$ and $D$, let $\exp[C,D]$ denote $C^D$.
\begin{lemma}[Davenport's Lemma \cite{davenport}]
\label{davenport}
Let $B$ be a bounded, semi-algebraic set in $\RR^n$ which is defined by at most $k$ polynomial inequalities each having degree at most $\ell$. Then the number of integer lattice points contained in the region $B$ is 
\[
	\Vol(B) + O(\max\{\Vol(\overline{B}), 1\}),
\]
where $\Vol(\overline{B})$ denotes the greatest $d$-dimensional volume of any projection of $B$ onto a coordinate subspace obtained by equating $n-d$ coordinates to zero, where $d$ takes all values from $1$ to $n-1$. The implied constant in the second summand depends only on $n$, $k$, and $\ell$.
\end{lemma}

Recall that $\flinec$ denotes the closed line segment from $(0,1/2)$ to $(1/4,1/4)$. We will also need the following two line segments. Let $\llinec$ denote the closed line segment from $(0,1/2)$ to $(1/6,1/3)$ and let $\rlinec$ denote the closed line segment from $(1/6,1/3)$ to $(1/4,1/4)$. Now, we apply Davenport's lemma to count lattice points:
\begin{lemma}
\label{davenportcubic}
For any $C,X$ and $p \in \flinec$, we have:
\[
	\exp[X,1+\max\{0,1/2-3p_1\}] \ll_p \#\{V_{\ZZ} \cap CB(p,X)\} \ll C^4\exp[X,1+\max\{0,1/2-3p_1\}].
\]
In particular, the first implicit constant is independent of $C$ and $X$, and the second implicit constant is independent of $C$, $X$, and $p$.
\end{lemma}
\begin{proof}
First suppose that $p \in \rlinec$ and $p \neq (1/6,1/3)$. Then $\Vol(CB(p,X)) = 2^4C^4X$. In this case, one sees that the projection of $CB(p,X)$ onto $a = 0$ has volume $2^3C^3\exp[X,1-(3p_1-1/2)]$, and the volume of any other coordinate projection is bounded above by this quantity. \cref{davenport} now shows that there exists a constant $D \geq 1$, independent of $C$, $X$, and $p$, such that:
\begin{align*}
\#\{V_{\ZZ} \cap CB(p,X)\} &\leq 2^4C^4X + 2^3C^3D\exp[X,1-(3p_1-1/2)] \\
& \leq 2^5DC^4X \\
& \ll C^4X
\end{align*}
We obtain the lower bound as follows:
\begin{align*}
\#\{V_{\ZZ} \cap CB(p,X)\} &\geq \#\{V_{\ZZ} \cap B(p,X)\} \\
&\geq 2^4X - D2^3\exp[X,1-(3p_1-1/2)] \\
&\gg_p X
\end{align*}

Now suppose $p \in \llinec$ and $p \neq (0,1/2)$. Then every lattice point in $CB(p,X)$ satisfies $\lv a \rv \leq C$. For an integer $-C \leq i \leq C$, let $B_i \subseteq \RR^3$ denote the subset of $CB(p,X)$ with $a = i$, projected onto $a = 0$. We will count the number of lattice points in each region $B_i$. Applying \cref{davenport} to the projection $b = 0$ shows that there exists a constant $D \geq 1$, independent of $C$, $X$ and $p$, such that:
\begin{align*}
\#\{V_{\ZZ} \cap CB(p,X)\} &= \sum_{i = -C}^C \#\{V_{\ZZ} \cap B_i(p,X)\} \\
&\leq (2C + 1)(2^3C^3\exp[X,1+(1/2-3p_1)] + 2^2C^2D\exp[X,1+(1/2-3p_1)-p_1]) \\
& \ll C^4\exp[X,1+(1/2-3p_1)]
\end{align*}
Again, the lower bound follows from the fact that:
\begin{align*}
\#\{V_{\ZZ} \cap CB(p,X)\} &\geq \#\{V_{\ZZ} \cap B(p,X)\} \\
&\geq 2^3\exp[X,1+(1/2-3p_1)] - 2^2D\exp[X,1+(1/2-3p_1)-p_1]  \\
&\gg_p \exp[X,1+(1/2-3p_1)]
\end{align*}

An analogous argument proves the lemma when $p = (0,1/2)$. 
\end{proof}

\subsubsection{Upper and lower bounds on $\#\cF_3(p,\eps,X)$}
\begin{proposition}
\label{upperboundcubic}
We have $\#\cF_3(p,\eps,X) \ll e^{12\eps} \exp[X,1-(p_2-p_1)+\max\{0,1/2-3p_1\}]$ where the implicit constants are independent of $\eps$, $X$, and $p$.
\end{proposition}
\begin{proof}
Let $S$ be the set of binary cubic forms arising from Minkowski bases of cubic rings in $\cF_3(p,\eps,X)$. \cref{minkbasisshortcubic} implies that there exists a global constant $C$ such that $f \in Ce^{3\eps}B(p,X)$ for all $f \in S$. By \cref{helplemma}, there exists a global constant $D$ such that $\#\{f' \in DCe^{3\eps}B(p,X) \mid R_{f'} \simeq R_{f}\} \geq X^{p_2-p_1}$ for all $f \in S$. Without loss of generality, we may suppose $D \geq 1$. Therefore, for all $f \in S$, we have $f \in DCe^{3\eps}B(p,X)$ and the ring $R_f$ is obtained by at least $X^{p_2-p_1}$ distinct forms in $DCe^{3\eps}B(p,X)$. Therefore,
\begin{align*}
\#\cF_3(p,\eps,X) &\leq \frac{\#\{DCe^{3\eps}B(p,X) \cap V_{\ZZ}\}}{\exp[X,p_2-p_1]} && \text{} \\
&\ll e^{12\eps}\exp[X,1+\max\{0,1/2-3p_1\}-(p_2-p_1)]. && \text{by \cref{davenportcubic}}
\end{align*}
\end{proof}

\begin{proposition}
\label{lowerboundcubic}
There exists a constant $C$ such that for all $\eps > C$, we have:
\begin{enumerate}
\item if $p \in \flinec$, then $\#\cF_3(p,\eps,X) \gg_{p} \exp[X,1+\max\{0,1/2-3p_1\}-(p_2-p_1)]$;
\item and if $p \in \rlinec$, then $\#\cF(p,\eps,S_3,X) \gg_{p} \exp[X,1-(p_2-p_1)]$.
\end{enumerate}
\end{proposition}
\begin{proof}
We first prove $(1)$. Fix $p \in \flinec$ and let $f$ denote any integral binary cubic form. \cref{lambdaboundlemma} shows that there exists a constant $C$ such that for all $\eps > C$, if $f \in B(p,X)$ and $\frac{X}{2} \leq \Delta$, then $f$ is $(\eps,X)$-close to $p$. Let $S = \{f \in B(p,X) \cap V_{\ZZ} \mid \frac{X}{2} \leq \Delta \leq X\}$.
Therefore for all $\eps > C$, we have
\begin{align*}
\#\cF_3(p,\eps,X) &\geq \frac{\#S}{\max_{f\in S}\#\{f' \in B(p,X) \cap V_{\ZZ} \mid R_{f'}\simeq R_f\}} && \text{} \\
&\asymp_p \frac{\#\{f \in B(p,X) \cap V_{\ZZ} \}}{\max_{f\in S}\#\{f' \in B(p,X) \cap V_{\ZZ} \mid R_{f'}\simeq R_f\}} && \text{by the definition of $B(p,X)$} \\
&\gg \frac{\#\{f \in B(p,X) \cap V_{\ZZ} \}}{\exp[X,p_2-p_1]} && \text{by \cref{reducedcubiccorol}} \\
&\gg_p \frac{\exp[X,1+\max\{0,1/2-3p_1\}]}{\exp[X,p_2-p_1]} && \text{by \cref{davenportcubic}} \\
&= \exp[X,1+\max\{0,3p_1-1/2\}-(p_2-p_1)]. && 
\end{align*}

If an integral binary cubic form is irreducible, then its Galois group must be $S_3$ or $A_3$. If the Galois group is $A_3$, then the discriminant is a square; so therefore if the discriminant is not a square, the Galois group must be $S_3$. If $\Delta \equiv 2 \Mod{5}$, then $\Delta$ is not a square. If $f$ is irreducible modulo $7$, then $f$ is irreducible. As this is just a congruence condition at $2$ primes, a simple computation shows that if $p \in \ell_3$, then
\begin{equation}
\label{eqn:asymp-eqn}
  \liminf_{X \rightarrow \infty}\frac{\#(S \cap \{f \in V_{\ZZ} \mid \Delta \equiv 2 \Mod{5}, \; f\text{ is irreducible modulo } 7\})}{\#S} > 0.  
\end{equation}

Therefore, if $p \in \ell_3$, then
\begin{align*}
\cF(p,\eps,X) &\geq \frac{\#(S \cap \{f \in V_{\ZZ} \mid \Delta \equiv 2 \Mod{5}, \; f\text{ is irreducible modulo } 7\})}{\max_{f\in S}\#\{f' \in B(p,X) \cap V_{\ZZ} \mid R_{f'}\simeq R_f\}} && \text{} \\
&\asymp  \frac{\#S}{\max_{f\in S}\#\{f' \in B(p,X) \cap V_{\ZZ} \mid R_{f'}\simeq R_f\}}. && \text{by \Cref{eqn:asymp-eqn}}
\end{align*}
Applying the same analysis as for $(1)$, we obtain:
\[
    \cF(p,\eps,X) \gg \exp[X,1+\max\{0,3p_1-1/2\}-(p_2-p_1)] = \exp[X,1-(p_2-p_1)].
\]
\end{proof}

\begin{lemma} 
\label{upperboundreduciblept1}
For all $p \in \flinec$ and $\nu \in \RR_{>0}$, the number of reducible integral binary cubic forms in $B(p,X)$ is $\ll_{\nu}X^{3/2-3p_1 + \nu}$.
\end{lemma}
\begin{proof}
By Davenport's lemma, the number of forms $f \in B(p,X)$ with $a = 0$ is $\ll X^{3/2-3p_1}$. Similarly, the number of forms $f \in B(p,X)$ with $d = 0$ is $\ll X^{3/2-3p_2} \leq X^{3/2-3p_1}$.

Let's now count the number of reducible integral binary cubic forms $f(x,y) \in B(p,X)$ with $ad \neq 0$. The total number of possibilities for the triple $(a,b,d)$ is $O(X^{1-p_2})$. Suppose the values $a,b,d$ are now fixed, and consider the possible number of values $c$ such that the resulting form $f(x,y)$ is reducible. For $f(x,y)$ to be reducible, it must have some linear factor $rx + sy$ where $r,s \in \ZZ$ are relatively prime. Then $r$ must be a factor of $a$, while $s$ must be a factor of $d$.

Now, the classical divisor bound implies that the number of factors of an integer $d$ is bounded above by $O_{\nu}(X^{\nu})$. Therefore, $r$ and $s$ are both determined up to $O_{\nu}(X^{\nu})$ possibilities. Once $r$ and $s$ are determined, computing $f(-s,r)$ and setting it equal to zero then uniquely determines $c$ (if it is an integer at all) in terms of $a,b,d,r,s$. Thus the total number of reducible forms $f$ with $ad \neq 0$ is $\ll_{\nu} X^{1-p_2+\nu} \ll_{\nu}X^{3/2-3p_1 + \nu}$.
\end{proof}

\subsection{Computing density functions}
\label{subsec:density-functions-cubic}
We now easily compute the density function of $S_3$-rings.

\begin{corollary}
\label{zerocubic}
For any permutation group $G \subseteq S_3$, the function $d_{\eps,G}$ has support contained in $\flinec$. If $G$ is transitive, then $d_{\eps,G}$ has support contained in $\rlinec$.
\end{corollary}
\begin{proof}
This is a corollary of \cite[Theorem 1.1.2]{me}.
\end{proof}

\begin{proof}[Proof of \cref{chichecubic}]
\cref{zerocubic} shows that the support of $d_{\eps,S_3}$ is contained in $\rlinec$. If $p \in \rlinec$, then combining \cref{upperboundcubic} and \cref{lowerboundcubic} shows that there exists a constant $C$ such that for all $\eps > C$, we have $\#\cF(p,\eps,S_3,X) \asymp_{p,\eps} X^{1-(p_2-p_1)}$. Thus for all $p \in \rlinec$, we have
\[
	d_{\eps,S_3}(p) = \lim_{X \rightarrow \infty}\frac{\log(\#\cF(p,\eps,S_3,X)+1)}{\log X} = \lim_{X \rightarrow \infty}\frac{\log(X^{1-(p_2-p_1)})}{\log X} = 1-(p_2-p_1). 
\]
\end{proof}

\subsection{Computing correspondences of special families of cubic rings}
\label{subsec:correspondences-cubic}
In this section we will prove \cref{genericcubiccorollary} and \cref{monogeniccubictheorem}.

\begin{proof}[Proof of \cref{genericcubiccorollary}]
Let $\cF^{\irr}$ be the set of isomorphism classes of irreducible nondegenerate cubic rings, let $H \colon \cF^{\irr} \rightarrow \RR_{>0}$ be the absolute discriminant, and let $p = (1/4,1/4)$. We will prove that $(\cF^{\irr},H)$ corresponds to $p$.

We first prove \eqref{cond1cubicspecialfamily} by demonstrating that:
\begin{equation}
\label{origcond1}
	\lim_{\eps \rightarrow \infty}\limsup_{X \rightarrow \infty}\frac{\#\{R \in \cF^{\irr} \; \mid \; \Delta \leq X \text{ and $R$ is not $(\eps,X)$-close to $p$}\}}{\#\{R \in \cF^{\irr} \; \mid \; \Delta \leq X \}} = 0.
\end{equation}
We bound the denominator using \cref{lowerboundcubic} and obtain:
\begin{equation}
\label{denombound}
\#\{R \in \cF^{\irr} \; \mid \; \Delta \leq X \} \gg X.
\end{equation}
We will now upper bound the numerator; our proof strategy is as follows. We fix a positive integer $k$ and count the number of rings $R\in \cF^{\irr}$ of bounded discriminant that are $(\eps/k,X)$-close to some \emph{other} point $p'$ on the line segment $\flinec$ with the property that that $p'$ has distance at least $\eps/\log X$ from $p$. We conclude by summing up across an appropriate set of points $p'$. 

Let $k=6$ and let 
\[
	p_{i,\eps} \coloneqq \bigg (p_1 - \frac{\eps(ik^{-1}+1)}{\log X}, p_2 + \frac{\eps(ik^{-1}+1)}{\log X}\bigg ) \in \RR^2
\]
Let $j_{X,\eps}$ be the smallest integer such that $p_1 - \frac{\eps(j_{X,\eps}k^{-1}+1)}{\log X} \leq 1/6$. The set $\{p_{i,\eps}\}_{i = 1}^{j_{X,\eps}}$ will be our appropriate set of points. By Minkowski's second theorem, there exists a global constant $C$ such that for all $\eps > C$, every element of $R \in \cF^{\irr}$ is $(\eps,\Delta)$-close to some point on the closed line segment from $(1/6,1/3)$ to $(1/4,1/4)$. Thus, for any $\eps > C$, we have:
\begin{align}
&\#\{R \in \cF^{\irr} \; \mid \; \Delta \leq X \text{ and $R$ is not $(\eps,X)$-close to $(1/4,1/4)$} \} \\
\label{line2}
&\leq \sum_{i = 1}^{j_{X,\eps}}\#\{R \in \cF^{\irr} \; \mid \; \Delta \leq X \text{ and $R$ is $(\eps/k,X)$-close to $p_{i,\eps}$} \}
\end{align}

\cref{upperboundcubic} proves that each of the terms in the sum in \eqref{line2} is 
\[
	\ll e^{-2\eps(ik^{-1}+1)}e^{12\eps/k}X,
\]
where the implied constant is independent of $i$ and $\eps$. Thus, we obtain that \eqref{line2} is
\begin{align}
&\ll Xe^{12\eps/k}\sum_{i=1}^{j_{X,\eps}} e^{-2\eps(ik^{-1}+1)}\\
\label{finalboundline}
&\leq X e^{12\eps/k}\frac{e^{-2\eps(k^{-1}+1)}}{1-e^{-2\eps k^{-1}}}
\end{align}
Plugging \eqref{finalboundline} into the numerator of \eqref{origcond1}, and plugging \eqref{denombound} into the denominator of \eqref{origcond1}, we obtain: 
\begin{equation}
\label{epsbound}
\limsup_{X \rightarrow \infty}\frac{\#\{R \in \cF^{\irr} \; \mid \; \Delta \leq X \text{ and $R$ is not $(\eps,X)$-close to $p$}\}}{\#\{R \in \cF^{\irr} \; \mid \; \Delta \leq X \}} \ll 
e^{12\eps/k}\frac{e^{-2\eps(k^{-1}+1)}}{1-e^{-2\eps k^{-1}}}
\end{equation}
where the implicit constant is independent of $\eps$. Now observe that:
\[
	\lim_{\eps \rightarrow \infty}e^{12\eps/k}\frac{e^{-2\eps(k^{-1}+1)}}{1-e^{-2\eps k^{-1}}} = \lim_{\eps \rightarrow \infty}\frac{e^{-2\eps k^{-1}}}{1-e^{-2\eps k^{-1}}} = 0.
\]
This proves \eqref{origcond1}, and \eqref{cond2cubicspecialfamily} is a corollary of \cref{upperboundcubic} and \cref{lowerboundcubic}.

Now let $\cF^{\redu}$ be the set of isomorphism classes of reducible nondegenerate cubic rings, let $H \colon \cF^{\redu} \rightarrow \RR_{>0}$ be the absolute discriminant, and let $p = (0,1/2)$.  We will now prove that $(\cF^{\redu},H)$ corresponds to $p$. The second condition is easy to see; observe that for any fixed $\eps$ and $\Delta$ large enough with respect to $\eps$, \cite[Theorem 1.1.2]{me} implies that any ring $(\eps,\Delta)$-close to $p$ is reducible. Hence, \eqref{cond2cubicspecialfamily} follows immediately from \cref{upperboundcubic} and \cref{lowerboundcubic}. We now prove that \eqref{cond1cubicspecialfamily} holds by demonstrating:
\begin{equation}
\label{redorig}
	\lim_{\eps \rightarrow \infty}\limsup_{X \rightarrow \infty}\frac{\#\{R \in \cF^{\redu} \; \mid \; \Delta \leq X \text{ and $R$ is not $(\eps,X)$-close to $p$}\}}{\#\{R \in \cF^{\redu} \; \mid \; \Delta \leq X \}} = 0
\end{equation}
By considering maximal reducible orders, we obtain: $\{R \in \cF^{\redu} \; \mid \; \Delta \leq X\} \gg X$. Again let $k$ be some integer, to fixed later, and let 
\[
	p_{i,\eps} \coloneqq \bigg(p_1 + \frac{\eps(ik^{-1}+1)}{\log X}, p_2 - \frac{\eps(ik^{-1}+1)}{\log X}\bigg) \in \RR^2.
\]
Let $j_{X,\eps}$ be the largest integer such that $p_1 + \eps(j_{X,\eps}k^{-1}+1)\log X < 1/6$ and let $k_{X,\eps}$ be the smallest integer such that $p_1 + \eps(j_{X,\eps}k^{-1}+1)\log X \geq 1/4$. There exists a constant $C$ such that for all $\eps > C$, we have: 
\begin{align*}
&\#\{R \in \cF^{\redu}\; \mid \; \Delta \leq X \text{ and $R$ is not $(\eps,X)$-close to $p$}\} \\
&\leq  \sum_{i = 1}^{j_{X,\eps}}\#\{R \in \cF^{\redu} \; \mid \; \Delta \leq X \text{ and $R$ is $(\eps/k,X)$-close to $p_{i,\eps}$} \} \\
&\;\;\;\;\;\;+ \sum_{i = j_{X,\eps}+1}^{k_{X,\eps}}\#\{R \in \cF^{\redu} \; \mid \; \Delta \leq X \text{ and $R$ is $(\eps/k,X)$-close to $p_{i,\eps}$} \}
\end{align*}
Using \cref{upperboundreduciblept1}, we see the second term is
\[
	\ll_{\nu,\eps} \sum_{i = j_{X,\eps}+1}^{k_{X,\eps}} X^{5/6 + \nu} \ll X^{5/6 + \nu}\log X = o(X).
\]
As before, for $1 \leq i \leq j_{X,\eps}$, \cref{upperboundcubic} implies:
\begin{align*}
\#\{R \in \cF^{\redu} \; \mid \; \Delta \leq X \text{ and $R$ is $(\eps/k,X)$-close to $p_{i,\eps}$} \} \ll e^{12\eps/k}X e^{-2\eps(ik^{-1}+1)}
\end{align*}

Fixing $k$ to be large enough and summing the first term proves \eqref{redorig}.

\end{proof}

\begin{proof}[Proof of \cref{monogeniccubictheorem}]
Let $\cF_3^{\mon}$ be the set of monic cubic polynomials with integer coefficients and nonzero discriminant and let $p = (1/6,1/3)$. If $f \in \cF_3^{\mon}$, then a computation shows that $H_r(f) \leq X^{1/6}$ if and only if $f \in B(p,X)$. Let $C$ be any positive real number. By \cref{lambdaboundlemma}, if $\Delta \geq X^6/C$ and $H_r(f) \leq X^{1/6}$, then
\[
	\lambda_i(R_f) \asymp_C \Delta^{p_i}.
\]
To prove \eqref{cond1cubicspecialfamily}, it suffices to show:
\[
\lim_{C \rightarrow \infty} \liminf_{X \rightarrow \infty} \frac{\#\{f \in \cF_3^{\mon} \; \mid \; H_r(f) \leq X \text{ and } \Delta > X^6/C\}}{\#\{f \in \cF_3^{\mon} \; \mid \; H_r(f) \leq X\}} = 1.
\]
Because the number of monic integral cubic forms with $H_r(f) \leq X$ and $\Delta = 0$ is $o(X^6)$, the denominator is equal to $X^6 + o(X^6)$. Let $\Disc \in \ZZ[a,b,c,d]$ denote the discriminant polynomial for binary cubic forms. By Davenport's lemma, we have
\begin{align*}
&\#\{f \; \mid \; H_r(f) \leq X \text{ and } \Delta > X^6/C\} \\
&= \#\{(b,c,d) \in \ZZ^3 \; \mid \; \max\{\lv b \rv^6, \lv c \rv^3, \lv d\rv^2 \}\leq X^6 \text{ and } \lv \Disc(1,b,c,d) \rv > X^6/C\} \\
&= \Vol\{(b,c,d) \in \RR^3 \; \mid \; \max\{\lv b \rv^6, \lv c \rv^3, \lv d\rv^2 \}\leq X^6 \text{ and } \lv \Disc(1,b,c,d) \rv > X^6/C\} + O(X^{5}) \\
&= \Vol\{(b,c,d) \in \RR^3 \; \mid \; \max\{\lv b \rv^6, \lv c \rv^3, \lv d\rv^2 \}\leq X^6 \text{ and }  \\
&\;\;\;\;\;\;\;\;\;\;\;\;\;\;\;\;\;\;\;\;\;\;\;\;\;\;\;\;\;\;\;\;\;\;\; \lv \Disc(1,bX^{-1},cX^{-2},dX^{-3}) \rv > 1/C\} + O(X^{5}) \\
&= X^6 \Vol\{(b,c,d) \in \RR^3 \; \mid \; \max\{\lv b \rv, \lv c \rv, \lv d\rv \}\leq 1 \text{ and } \lv \Disc(1,b,c,d) \rv > 1/C\} + O(X^{5}).
\end{align*}

Thus:
\begin{align*}
&\lim_{C \rightarrow \infty} \liminf_{X \rightarrow \infty} \frac{\#\{f \in \cF_3^{\mon} \; \mid \; H_r(f) \leq X \text{ and } \Delta > X^6/C\}}{\#\{f \in \cF_3^{\mon} \; \mid \; H_r(f) \leq X\}} \\
&= \lim_{C \rightarrow \infty}\Vol\{(b,c,d) \in \RR^3 \; \mid \; \max\{\lv b \rv, \lv c \rv, \lv d\rv \}\leq 1 \text{ and } \lv \Disc(1,b,c,d) \rv > 1/C\} \\
&=1
\end{align*}
We obtain \eqref{cond2cubicspecialfamily} by observing that a positive proportion (depending on $\eps$) of lattice points in $B(p,X)$ lie on the hyperplane $a = 1$, and then executing the proof of \cref{upperboundcubic}.
\end{proof}

\subsection{Sieving for maximal orders in $S_3$-cubic fields}
\label{subsec:step-5-cubic}
In this section, we will prove \cref{s3maxthm} by sieving for maximal orders in $S_3$-cubic fields. Throughout this section, let $p = (p_1,p_2) \in \rlinec$ be a point and let $\ell$ denote a prime number. As before, let $f$ denote an integral binary cubic form and let $X \in \RR_{> 1}$. We will need the following tail estimate for ``small'' primes. Let $\Disc \in \ZZ[a,b,c,d]$ be the discriminant polynomial of binary cubic forms.

\begin{lemma}
\label{tail1}
If $\ell \leq \sqrt{X^{3p_2 - 1/2}}$ then
\[
	\#\{f \in B(p,X) \cap V_{\ZZ} \mid \Disc(f) \equiv 0 \Mod{\ell^2}\} \ll X/\ell^2
\]
where the implicit constant is independent of $X$, $p$, and $\ell$.
\end{lemma}
\begin{proof}
Without loss of generality suppose $\ell \neq 2,3$. Fix integer values of $a$, $b$, and $c$ such that:
\begin{align*}
	\lv a \rv &\leq X^{3p_1-1/2} \\
	\lv b \rv &\leq X^{2p_1 + p_2 - 1/2} \\
	\lv c \rv &\leq X^{p_1 + 2p_2 - 1/2}
\end{align*}
We will count the number of integer values of $d$ where $\lv d \rv \leq X^{3p_2 - 1/2}$ and $\ell^2 \mid \Disc(a,b,c,d)$. We separate into the following $3$ cases:
\begin{enumerate}
\item $\ell^2 \mid a$;
\item $\ell \mid a$ and $\ell^2 \nmid a$;
\item and $\ell \nmid a$.
\end{enumerate}
We first analyze case $(1)$. If $\ell \leq \sqrt{X^{3p_1 - 1/2}}$, Case $(1)$ occurs with probability $1/\ell^2$. If $\ell > \sqrt{X^{3p_1 - 1/2}}$, then case $(1)$ never occurs. Therefore we may bound the quadruples $(a,b,c,d)$ with $\ell^2 \mid a$ by $O(X/\ell^2)$.

Similarly, if $\ell > X^{3p_1 - 1/2}$, then case $(2)$ never occurs. If $\ell \leq X^{3p_1 - 1/2}$ case $(2)$ occurs with probability $1/\ell$. Suppose we have a quadruple $(a,b,c,d)$ such that $\ell \mid a$ and $\ell^2 \nmid a$ and write $a'\ell = a$. We have
\[
	\Disc \equiv (18a'\ell bc - 4b^3)d + (b^2c^2 - 4a'\ell c^3) \mod{\ell^2}
\]
The above expression has more than $1$ solution for $d$ modulo $\ell^2$ if and only if $\ell \mid b$ and $\ell \mid c$, which occurs with probability $1/\ell^2$; in this case, we have $\ell^2$ solutions for $d$ modulo $\ell^2$. Therefore we may bound the number of quadruples $(a,b,c,d)$ with $\ell \mid a$ and $\ell^2 \nmid a$ and $\ell^2 \mid \Disc(a,b,c,d)$ by $O(X/\ell^3)$.

Now suppose we have a quadruple $(a,b,c,d)$ with $\ell \nmid a$: then either $(i)$: $\Disc$, when considered as a monic quadratic polynomial in $d$, has no double roots modulo $\ell$; or $(ii)$: $\Disc$, when considered as a monic quadratic polynomial in $d$ has a double root modulo $\ell$. We write the discriminant of a binary cubic form $f(x,y) = ax^3 + bx^2y + cxy^2 + d^3$ as
\begin{equation*}
	\Disc(f) = (27a^2)d^2 + (18abc - 4b^3)d + (b^2c^2 - 4ac^3).
\end{equation*}

In subcase $(i)$, Hensel's lemma shows that there are at most $2$ solutions for $d$ modulo $\ell^2$; because $\ell \leq \sqrt{X^{3p_2 - 1/2}}$ and $d$ is chosen in the range $[-X^{3p_2 - 1/2},X^{3p_2 - 1/2}]$, the number of such quadruples is $O(X/\ell^2)$. In subcase $(ii)$ there are at most $\ell$ solutions for $d$ modulo $\ell^2$ because $f \not \equiv 0 \Mod{\ell}$, so it suffices to show that the second subcase occurs with probability $1/\ell$. The second case only occurs if the discriminant of $\Disc(a,b,c,d)$, when $\Disc(a,b,c,d)$ is considered as a quadratic polynomial in $d$, is divisible by $\ell$. The discriminant of $\Disc(a,b,c,d)$, when considered this way, is:
\[
	(18abc - 4b^3)^2 - 4(27a^2)(b^2c^2 - 4ac^3)
\]
which we simplify to:
\[
	16 b^6 - 144 a b^4 c + 216 a^2 b^2 c^2 + 432 a^3 c^3.
\]
Recall that we have assumed $\ell \nmid a$. Therefore, each choice of $a$ and $b$ gives at most $3$ values for $c$ modulo $\ell$. Recall that $c$ is chosen in the range $[-X^{p_1 + 2p_2-1/2},-X^{p_1 + 2p_2-1/2}]$ and $X^{p_1 + 2p_2-1/2} \geq \ell$; so the probability that the discriminant of $\Disc(a,b,c,d)$ is divisible by $\ell$ is $O(1/\ell)$. Therefore, the number of such quadruples is $O(X/\ell^2)$.
\end{proof}
For ``big'' primes, we use the following tail estimate of Bhargava, Shankar, and Tsimerman.
\begin{proposition}[Bhargava, Shankar, Tsimerman \cite{BhDH}, Proposition~29]
\label{tail2}
The number of isomorphism classes of nondegenerate cubic rings such of absolute discriminant $\leq X$ that are not maximal at a prime $\ell$ is $O(X/\ell^2)$, where the explicit constant is independent of $X$ and $\ell$.
\end{proposition}

We can now combine these two tail estimates and prove our theorem on the density function of maximal orders in $S_3$-cubic fields.
\begin{proof}[Proof of \cref{s3maxthm}]
For any prime $\ell$, let $\cU_{\ell}$ denote the subset of $V_{\ZZ_{\ell}}$ corresponding to maximal cubic rings over $\ZZ_{\ell}$. Let $\mu_{\ell}$ be the Haar measure on $V_{\ZZ_{\ell}}$, normalized such that $\mu_{\ell}(V_{\ZZ_{\ell}}) = 1$. Let $\cU = \cap_{\ell}\cU_{\ell}$. Then $\cU$ is the set of $f \in V_{\ZZ}$ corresponding to maximal cubic rings over $\ZZ$.

Let $p = (p_1,p_2) \in \rlinec$. By \cref{upperboundcubic} and the fact that $100\%$ of the integral forms in $B(p,X)$ are irreducible, it suffices to show that  a positive proportion of forms in $B(p,X)$ have squarefree discriminant. For any positive integer $Y$, we have
\begin{align*}
&\liminf_{X\rightarrow \infty}\frac{\#\{ \cU  \cap B(p,X) \cap V_{\ZZ} \}}{X}  \\
&\geq \liminf_{X \rightarrow \infty}\frac{\#\{ f \in \cap_{\ell\leq Y}\cU_{\ell}\cap B(p,X) \cap V_{\ZZ} \mid \Delta \geq X/2\} - \sum_{\ell> Y}\#\{f \in \overline{\cU_{\ell}}\cap B(p,X) \cap V_{\ZZ} \mid \Delta \geq X/2 \}}{X}
\end{align*}
We lower bound the right hand side by:
\begin{align}
\label{mainsieveeqn}
&\liminf_{X \rightarrow \infty}\frac{\#\{ f \in \cap_{\ell\leq Y}\cU_{\ell}\cap B(p,X) \cap V_{\ZZ} \mid \Delta \geq X/2\}}{X} \\
& - \limsup_{X \rightarrow \infty}\frac{\sum_{\ell> Y}\#\{f \in \overline{\cU_{\ell}}\cap B(p,X) \cap V_{\ZZ} \mid \Delta \geq X/2\}}{X}
\end{align}
We estimate the second term of \eqref{mainsieveeqn}. Using the tail estimates given in \cref{tail1} and \cref{tail2} and the bound in \cref{reducedcubiccorol}, we obtain:
\begin{align}
\label{twoparttail}
&\sum_{\ell > Y}\#\{f \in \overline{\cU_{\ell}}\cap B(p,X) \cap V_{\ZZ} \mid \Delta \geq X/2\}  \\
&\ll \sum_{Y < \ell \leq \sqrt{X^{3p_2-1/2}}}O(X/\ell^2) + \sum_{Y, \sqrt{X^{3p_2-1/2}} < \ell}O(X^{1+{p_2 - p_1}}/\ell^2)
\end{align}
The fact that $\sqrt{X^{3p_2-1/2}} < \ell$ in the second term of \eqref{twoparttail} implies that there exists a positive real number $\eta$ such that the second term of \eqref{twoparttail} is bounded above by:
\[
\sum_{Y, \sqrt{X^{3p_2-1/2}} < \ell}O(X^{1}/\ell^{1+\eta})
\]
Therefore,
\[
	\lim_{Y\rightarrow \infty}\limsup_{X \rightarrow \infty}\frac{\sum_{\ell> Y}\#\{f \in \overline{\cU_{\ell}}\cap B(p,X) \cap V_{\ZZ} \mid \Delta \geq X/2\}}{X} = 0.
\]
We now estimate the first term of \eqref{mainsieveeqn}. First suppose $p \neq (1/6,1/3)$. Then, because $B(p,X)$ is a box expanding in all coordinates with volume $X$, we have
\[
	\lim_{X \rightarrow \infty}\frac{\#\{ f\in\cap_{\ell\leq Y}\cU_{\ell}  \cap B(p,X) \cap V_{\ZZ} \mid \Delta \geq X/2 \}}{X} = \prod_{\ell \leq Y}\mu_{\ell}(\cU_{\ell}).
\]
Letting $Y$ tend to infinity completes the proof, as a computation of the local densities, given in \cite{BhDH}, shows the the product of the local densities is nonzero. An analogous proof proves the theorem for $p = (1/6,1/3)$, given that one restricts $B(p,X)$ to the subset where $a = 1$. The same estimate of the first term of \eqref{mainsieveeqn} applies, but but we obtain a different set of local densities. A computation shows that in this case, the product of relevant local densities is still nonzero, which completes the proof.
\end{proof}

\begin{proof}[Proof of \cref{totdensitycubic}]
Combine \cref{lowerboundcubic}, \cref{upperboundcubic}, \cref{s3maxthm}, \cref{zerocubic}, and the fact that all maximal reducible cubic rings can be written as $\ZZ \oplus R$ where $R$ is a maximal quadratic ring.
\end{proof}

\section{Proof of \cref{binthm}, the theorem on binary $n$-ic forms}
We now generalize the reduction theory of binary cubic forms that we developed in \cref{cubicreductiontheory} to general degree $n$. This will be useful to prove our later results on binary quartic and binary quintic forms. In this section, we will discuss Birch, Merriman, and Nakagawa's \cite{birch,naka} construction of a rank $n$ ring from an integral binary $n$-ic form. We develop the relationship between:
\begin{enumerate}
\item nondegenerate rank $n$ rings $R$ equipped with ``small'' bases arising from integral binary $n$-ic forms;
\item and integral binary $n$-ic forms with ``small'' coefficients.
\end{enumerate}

\subsection{Setup}
\label{subsec:bin-n-setup}
For the rest of this section, let $f = f_0x^n + \dots + f_ny^n \in \Sym^n\ZZ^2$ denote a binary $n$-ic form. To a form $f$, Birch, Merriman, and Nakagawa \cite{birch, naka} associated a rank $n$ ring $R_f$ equipped with a basis $\{\zeta_0=1,\zeta_1,\dots,\zeta_{n-1}\}$. The multiplication table of the basis $\{\zeta_0=1,\zeta_1,\dots,\zeta_{n-1}\}$ is given by:
\[
	\zeta_i\zeta_j = \sum_{\max\{i+j-n,1\}\leq k \leq i}f_{i+j-k}\zeta_k + \sum_{j < k \leq \min\{i+j,n\}}f_{i+j-k}\zeta_k
\]
for $1 \leq i,j \leq n-1$ where $\zeta_n \coloneqq -f_n$. An element $\gamma \in \GL_2(\ZZ)$ acts on $f$ via
\[
	(\gamma f)(x,y) \coloneqq f((x,y)\gamma);
\]
this action of $\GL_2(\ZZ)$ on $\Sym^n\ZZ^2$ induces an action of $\GL_2(\ZZ)$ on the basis $\{\zeta_1,\dots,\zeta_{n-1}\}$ of $R_f/\ZZ$. This action has an invariant called the \emph{discriminant}, which is a homogeneous polynomial of degree $2n-2$ in the coefficients of $f$. Nakagawa's construction is \emph{discriminant--preserving}: we have $\Disc(R_f) = \Disc(f)$. In this section, let $\Delta = \lv \Disc(f)\rv$. We say a form $f$ is \emph{nondegenerate} if $\Delta \neq 0$. For the rest of this section, any $f$ we consider will be nondegenerate.

\subsection{Proof overview}
Recall that in \cref{cubicreductiontheory}, for any point $p \in \RR^2$ and any positive real number $X$, we defined a box $B(p,X) \subseteq \Sym^3 \RR^2$. We roughly showed that integral binary cubic forms lying in this box correspond to cubic rings of discriminant $\leq X$ equipped with a ``short'' basis that were ``close'' to $p$. We'll now generalize our definition of the box $B(p,X) \subseteq \Sym^3 \RR^2$ for general binary $n$-ic forms.
Let
\[
	k \coloneqq \frac{1}{2(1 + \dots + (n-1))},
\]
and define $\ell$ to be the closed line segment from $(0,k)$ to $(k/2,k/2)$.

Let $L \colon \ell \rightarrow \RR^{n-1}$ be the linear map sending the point $(0,k)$ to $(k,2k,\dots,(n-1)k)$ and sending the point $(k/2,k/2)$ to $\frac{1}{2(n-1)}(1,\dots,1)$. 

\begin{definition}
For a point $r = (r_1,r_2) \in \ell$ and a positive real number $X$, define $B_n^{\bin}(r,X)$ to be the subset of $\Sym^n\RR^2$ such that:
\[
	\lv f_i \rv \leq X^{(n-i)r_1 + ir_2}
\]
for all $0 \leq i \leq n$.     
\end{definition}

The boxes $B_n^{\bin}(r,X) \subseteq \Sym^n \RR^2$ are the natural generalization of the box $B(p,X) \subseteq \Sym^3 \RR^2$ that we defined in \cref{cubicreductiontheory}. if $n = 3$, then a computation shows that
\[
	B_n^{\bin}(r,X) = B(L(r),X).
\]
In what follows, we will relate a form being in $B_n^{\bin}(r,X)$ to the ring $R_f$ being ``close'' to $L(r)$, and use the same proof strategy we used for cubic rings. 

\subsection{Reduction theory of binary $n$-ic forms}

\subsubsection{A ``short'' basis gives rise to a form $f$ with ``small'' coefficients}
We'll first show that if $R_f$ is ``close'' to $L(r)$ and $\{1,\zeta_1,\dots,\zeta_{n-1}\}$ is ``small'', then $f$ is roughly contained in this box $B_n^{\bin}(r,X)$. Our proof will require the following beautiful formula of Fess. Given a nondegenerate rank $n$ ring $R$ equipped with a basis $1,\zeta_1,\dots,\zeta_{n-1}$ and an element $\alpha \in R$, let $M_{\alpha}$ be an $n\times n$ matrix such that
\[
\begin{bmatrix}
1 & \alpha & \dots & \alpha^{n-1}\\
\end{bmatrix} = M_{\alpha}\begin{bmatrix}
1 & \zeta_1 & \dots & \zeta_{n-1}\\
\end{bmatrix}^t
\] 
A nondegenerate rank $n$ ring $R$ equipped with a basis $1,\zeta_1,\dots,\zeta_{n-1}$ naturally gives rise to an \emph{index form}, defined by:
\begin{align*}
I \colon R/\ZZ &\longrightarrow \ZZ \\
\alpha &\longrightarrow \det(M_{\alpha})
\end{align*}
The basis $\{\zeta_1,\dots,\zeta_{n-1}\}$ of $R/\ZZ$ presents the index form as a homogeneous polynomial with integer coefficients of degree $n \choose 2$ in $n-1$ variables. In this way, a nondegenerate integral binary $n$-ic form $f$ naturally gives rise to a homogeneous polynomial $I_f$ of degree $n \choose 2$ in $n-1$ variables. 

\begin{theorem}[Fess \cite{fess}, Theorem~3.2.4]
\label{indexthm}
For any $f \in \Sym^n \ZZ^2$, we have:
\[
	I_f(x^{n-2},x^{n-3}y,\dots,y^{n-2}) = f(x,y)^{n-1 \choose 2}
\]
\end{theorem}

Our proof strategy is as follows. Suppose $R_f$ is ``close'' to $L(r)$ and $\{1,\zeta_1,\dots,\zeta_{n-1}\}$ is ``small''. Then we use this information to bound the sizes of the coefficients of the index form $I_f$, and hence the sizes of the coefficients of the binary form $I_f(x^{n-2},x^{n-3}y,\dots,y^{n-2})$. By Fess' index formula, we have obtained a bound on the coefficients of $f(x,y)^{n-1 \choose 2}$; a classical theorem of Gelfand (\cref{gelfand}) then gives a bound on the coefficients of $f(x,y)$. 

Let $\theta$ denote a positive real number. We say $f$ \emph{gives rise to a $\theta$-reduced basis} if every pair of elements of the basis $1,\zeta_1,\dots,\zeta_{n-1}$ of $R_f$ have angle greater than $\theta$ in the Minkowski embedding of $R_f$. For a positive real number $\eps$, we say $f$ \emph{is $(\eps,\Delta)$-close to a point $p \in \RR^{n-1}$} if $\Delta > 1$ and $R_f$ is $(\eps,\Delta)$-close to $p$. Let $r = (r_1,r_2) \in \ell_1$ denote a point.

\begin{proposition}
For all $\theta, \eps \in \RR_{>0}$ and all $n \in \ZZ_{>1}$, there exists $C_{n,\theta,\eps} \in \RR_{>0}$ such that for all nondegenerate $f \in \Sym^n \ZZ^2$ and all $r \in \ell$, if $\Delta > 0$ and $f$ is $(\eps,\Delta)$-close to $L(r)$ and $f$ gives rise to a $\theta$-reduced basis, then $f \in C_{n,\theta,\eps}B_n^{\bin}(r,\Delta)$.
\end{proposition}
\begin{proof}
View $I_f$ as a homogeneous polynomial of degree ${n \choose 2}$ with integer coefficients in $n-1$ variables $x_1,\dots,x_{n-1}$. For any $\mathbf{k} = (k_1,\dots,k_{n-1}) \in \ZZ_{\geq 0}^{n-1}$ such that $\sum_{\ell = 1}^{n-1}k_{\ell} = {n \choose 2}$, let $a_{\mathbf{k}}$  denote the coefficient of $\prod_{\ell = 1}^{n-1}x_{\ell}^{k_{\ell}}$ in $I_f$. We now bound $a_{\mathbf{k}}$. Let $\sigma_1,\dots,\sigma_n$ be the $n$ homomorphisms of $R_f$ into the complex numbers. For an element $\alpha = \sum_{\ell = 0}^{n-1}x_{\ell}\zeta_{\ell}$, we have:
\[
	\lv \Disc(\ZZ[\alpha])\rv^{1/2} =  \begin{array}{|c c c|}
1 & \dots & 1\\
\sum_{\ell = 0}^{n-1}x_{\ell}\sigma_1(\zeta_{\ell}) & \dots & \sum_{\ell = 0}^{n-1}x_{\ell}\sigma_1(\zeta_{\ell}) \\
\dots & \dots & \dots \\
(\sum_{\ell = 0}^{n-1}x_{\ell}\sigma_1(\zeta_{\ell}))^{n-1} & \dots & (\sum_{\ell = 0}^{n-1}x_{\ell}\sigma_n(\zeta_{\ell}))^{n-1} 
\end{array}
\]
Expand the determinant above and collect monomials. Observe that because $\lv \sigma_i(\zeta_{\ell}) \rv \ll_n \lv \zeta_{\ell} \rv$, the coefficient of $\prod_{\ell = 1}^{n-1}x_{\ell}^{k_{\ell}}$ in the determinant of the matrix above is $\ll_n \prod_{\ell = 1}^{n-1}\lv \zeta_{\ell} \rv^{k_{\ell}}$. Because
\[
\rv I_f(\alpha)\lv = \lv \Disc(\ZZ[\alpha])\rv^{1/2}\Delta^{-1/2},
\]
we have:
\begin{equation}
 \label{akbound}
\begin{split}
\lv a_{\mathbf{k}} \rv & \ll_n \Delta^{-1/2} \prod_{\ell = 1}^{n-1}\lv \zeta_{\ell} \rv^{k_{\ell}} \\
 & \ll_{\eps,\theta,n} \Delta^{-1/2}\prod_{\ell = 1}^{n-1}\Delta^{k_{\ell}L(r)_{\ell}} \\
 & \leq \exp[\Delta, -1/2 + \sum_{\ell = 1}^{n-1}k_{\ell}L(r)_{\ell}] \\
 & \leq \exp[\Delta, -1/2 + \sum_{\ell = 1}^{n-1}\frac{k_{\ell}}{n-2}((n-1-{\ell})L(r)_1 + (\ell-1)L(r)_{n-1})] \\
 & \leq \exp[\Delta, -1/2 + \frac{L(r)_1}{n-2}\sum_{\ell = 1}^{n-1}k_{\ell}(n-1-\ell) + \frac{L(r)_{n-1}}{n-2}\sum_{\ell = 1}^{n-1}k_{\ell}(\ell-1)]
\end{split}
\end{equation}
We now use \eqref{akbound} to bound the coefficients of $I_f(x^{n-2},x^{n-3}y,\dots,y^{n-2})$. For any $0 \leq i \leq (n-2){n \choose 2}$, let $b_{i}$ denote the coefficient of $x^{(n-2){n \choose 2} - i}y^{i}$ in $I_f(x^{n-2},x^{n-3}y,\dots,y^{n-2})$. We can express $b_i$ as a linear polynomial in the variables $a_{\mathbf{k}}$. Suppose $a_{\mathbf{k}} = a_{(k_1,\dots,k_{n-1})}$ appears in the expansion of $b_i$. Then substituting in $x^{n-1-\ell}y^{\ell-1}$ for $x_{\ell}$ in $\prod_{\ell=1}^{n-1}x_{\ell}^{k_{\ell}}$ gives $x^{(n-2){n \choose 2} - i}y^{i}$, so:
\begin{equation}
\label{sumeqn}
\begin{split}
(n-2){n \choose 2} - i &= \sum_{\ell = 1}^{n-1}k_{\ell}(n-1-\ell) \\
i &=  \sum_{\ell = 1}^{n-1}k_{\ell}(\ell - 1)
\end{split}
\end{equation}
Substituting \eqref{sumeqn} into \eqref{akbound}, we obtain:
\begin{equation}
\label{monomialbound}
	\lv a_{\mathbf{k}} \rv \ll_{n,\eps,\theta} \exp[\Delta, -1/2 + \frac{L(r)_1}{n-2}((n-2){n \choose 2} - i) + \frac{L(r)_{n-1}}{n-2}i].
\end{equation}
Note that \eqref{monomialbound} is \emph{not} dependent on the choice of $a_{\mathbf{k}}$! Thus:
\begin{equation}
\label{combpound}
	\lv b_i \rv \ll_{n,\eps,\theta} \exp[\Delta, -1/2 + \frac{L(r)_1}{n-2}((n-2){n \choose 2} - i) + \frac{L(r)_{n-1}}{n-2}i]
\end{equation}
We can write out the linear map $L$ explicitly and obtain:
\begin{equation}
\label{sub0}
\frac{n-1}{2}(L(r)_1 + L(r)_{n-1}) = 1/2
\end{equation}
\begin{equation}
\label{sub1}
	L(r)_1 = \frac{1}{2{n \choose 2}} + \frac{\frac{1}{2(n-1)}-\frac{1}{2{n \choose 2}}}{\frac{1}{4{n \choose 2}}}r_1
\end{equation}
\begin{equation}
\label{sub2}
	L(r)_{n-1} = \frac{n-1}{2{n \choose 2}} + \frac{\frac{1}{2(n-1)}-\frac{n-1}{2{n \choose 2}}}{\frac{1}{4{n \choose 2}} - \frac{1}{2{n \choose 2}}}\bigg(r_2 - \frac{1}{2{n \choose 2}}\bigg).
\end{equation}
Substitute \eqref{sub0}, \eqref{sub1}, and \eqref{sub2} into \eqref{combpound} and simplify to obtain:
\begin{equation}
\label{finalbbound}
	\lv b_i \rv \ll_{n,\eps,\theta} \exp[\Delta, ((n-2){n \choose 2}-i)r_1 + ir_{2}]
\end{equation}
Recall that \cref{indexthm} says:
\[
	I_f(x^{n-2},x^{n-3}y,\dots,y^{n-2}) = f(x,y)^{n-1 \choose 2}
\]
Let $M = \lfloor \Delta^{r_2-r_1} \rfloor$ and let $g(x,y)  \coloneqq f(Mx,y)$. For any $0 \leq i \leq (n-2){n \choose 2}$, let $c_i$ denote the coefficient of $x^{(n-2){n \choose 2} - i}y^i$ in $g(x,y)^{n-1 \choose 2}$. 
By \cref{indexthm}, we have
\[
	g(x,y)^{n-1 \choose 2} = I_f(M^{n-2}x^{n-2},M^{n-3}x^{n-3}y,\dots,y^{n-2})
\]
Note that $c_i = M^{(n-2){n \choose 2}-i}b_i$; substituting in \eqref{finalbbound} gives:
\[
	\lv c_i \rv \ll_{n,\eps,\theta} \exp[\Delta, (n-2){n \choose 2}r_{2}].
\]
For $0 \leq i \leq n$ let $g_i$ be the coefficient of $x^{n-i}y^{i}$ in $g(x,y)$. By \cref{gelfand}, for all $0 \leq i \leq n$ we have
\[
	\lv g_i\rv \ll_{n,\eps,\theta} \Delta^{nr_2}.
\]
Because $g_i = M^{n-i}f_i$, we have
\[
	\lv f_i \rv \ll_{n,\eps,\theta} \Delta^{(n-i)r_1 + ir_2}
\]
for $0 \leq i \leq n$, completing the proof of the proposition.
\end{proof}

For a polynomial $f(x_1,\dots,x_m) \in \ZZ[x_1,\dots,x_m]$, let $H(f)$ be the maximum of the absolute value of the coefficients of $f$.

\begin{lemma}[Gelfand's Inequality \cite{silverman}, Proposition~B.7.3]
\label{gelfand}
Let $d_1,\dots,d_m$ be integers, and let $f_1,\dots,f_r \in \ZZ[x_1,\dots,x_m]$ be polynomials whose product satisfies $\deg_{x_i}(f_1\dots f_r) = d_m$ for each $1 \leq i \leq r$. Then
\[
	\prod_{i}H(f_i) \leq e^{d_1 + \dots + d_m}H(f_1\dots f_r).
\]
\end{lemma}

\subsubsection{A form with ``small'' coefficients gives rise to a short basis}
For the rest of this section, let $C,D$ denote positive real numbers, let $X \in \RR_{>1}$, let $n \in \ZZ_{>1}$, and let $r \in \ell$. 
\begin{lemma}
\label{boxlemmabin}
If $f \in CB_n^{\bin}(r,X)$ then
\[
	\lv \zeta_i \rv \ll_{C,n} X^{L(r)_i}
\]
for all $1 \leq i < n$.
\end{lemma}
\begin{proof}
Let $L(r)_0 = 0$. The explicit multiplication coefficients imply that 
\[
	\lv \zeta_i \rv^2 \ll_{C,n} \max_{0 \leq j < n}\{X^{2L(r)_i - L(r)_j}\lv \zeta_j\rv\}.
\]
Applying \cref{generalboundlemma} completes the proof.
\end{proof}

\begin{lemma}
\label{lambdaboundlemmabin}
If $f \in CB_n^{\bin}(r,X)$ and $X/2 \leq \Delta$ then
\[
	\lambda_i(R_f) \asymp_{C,n} X^{L(r)_i}
\]
for $1 \leq i \leq n-1$.
\end{lemma}
\begin{proof}
Follows from \cref{boxlemmabin}, Minkowski's second theorem, and the fact that
\[
	\sum_{i = 1}^{n-1}L(r)_i = 1/2.
\]
\end{proof}

\subsection{Counting lattice points and proof of \cref{binthm}}
\begin{lemma}
\label{autlemma}
If $f \in CB_n^{\bin}(r,X)$ and $\Delta \geq X/2$, we have
\[
	\#\{f' \in CB_n^{\bin}(r,X) \cap \Sym^n \ZZ^2 \mid \exists \gamma \in \GL_2(\ZZ) \text{ s.t. } f' = \gamma f\} \ll_{C,n} X^{r_2-r_1}
\]
\end{lemma}
\begin{proof}
By \cref{boxlemmabin}, the form $f$ gives rise to a basis $\{\zeta_1,\dots,\zeta_{n-1}\}$ of $R_f/\ZZ$ with the property that 
\[
	\lv \zeta_i \rv \ll_{C,n} X^{L(r)_i} \asymp_{C,n} \lambda_i(R_f) 
\]
for $1 \leq i \leq n-1$. Let $\gamma \in \GL_2(\ZZ)$ be such that $f' = \gamma f$ and $f' \in B_n^{\bin}(r,X)$. Let $\{\zeta'_1,\dots,\zeta'_{n-1}\}$ be the basis of $R_f/\ZZ$ be the basis corresponding to $f'$. Then again by \cref{boxlemmabin}, we have
\[
	\lv \zeta'_i \rv \ll_{C,n} X^{L(r)_i} \asymp_{C,n} \lambda_i(R_f).
\]
Write
\[
	\gamma = \begin{bmatrix}
    a & b \\
    c & d 
  \end{bmatrix}
\]
for an element $\gamma \in \GL_2(\ZZ)$. As given by Nakagawa \cite{naka}, the action of $\gamma$ on the bases is given by the following matrix equation:
\[
	\begin{bmatrix}
	a^{n-2} & a^{n-3}b & \dots & ab^{n-3} & b^{n-2} \\
	\dots & \dots & \dots & \dots & \dots \\
	c^{n-2} & c^{n-3}d & \dots & cd^{n-3} & d^{n-2} \\
	\end{bmatrix}
	\begin{bmatrix}
	\zeta_1 \\
	\dots \\
	\zeta_{n-1} 
	\end{bmatrix}
	= \begin{bmatrix}
	\zeta'_1 \\
	\dots \\
	\zeta'_{n-1} 
	\end{bmatrix}
\]
For our purposes, it is not important what the $\dots$ entries are. Thus:
\begin{align*}
\lv \zeta_1 \rv \asymp_{C,n}\lv \zeta'_1 \rv &\asymp_{C,n} \max_{1 \leq i \leq n-1}\{\lv a^{n-1-i}b^{i-1}\rv\lv \zeta_i \rv\} \\
\lv \zeta_{n-1} \rv \asymp_{C,n} \lv \zeta'_{n-1} \rv &\asymp_{C,n} \max_{1 \leq i \leq n-1}\{\lv c^{n-1-i}d^{i-1}\rv\lv \zeta_i \rv\}
\end{align*}
Thus, the number of choices of $a,b,d$ is $\ll_{C,n} 1$, and the number of choices of $c$ is $\ll_{C,n} X^{r_2 - r_1}$.
\end{proof}

\begin{lemma}
There exists $D$ such that for all $f$, $n$, $r$, and $X$, if $f \in CB_n^{\bin}(r,X)$ then
\[
	\#\{f' \in DCB_n^{\bin}(r,X) \cap \Sym^n \ZZ^2 \mid \exists \gamma \in \GL_2(\ZZ) \text{ s.t. } f' = \gamma f\} \geq X^{r_2-r_1}
\]
\end{lemma}
\begin{proof}
Act on $f$ by elements of $\GL_2(\ZZ)$ of the form 
\[
	\gamma = \begin{bmatrix}
    1 & 0 \\
    u & 1 
  \end{bmatrix}
\]
where $\lv u \rv \leq X^{r_2-r_1}$.
\end{proof}
The following theorem will be useful later.
\begin{theorem}[Evertse, Gy\H{o}ry \cite{evertse}]
\label{evertse}
For any nondegenerate rank $n$ ring $R$, there are at most $2^{5n^2}$ classes of integral binary $n$-ic forms $f$ such that $R \simeq R_f$.
\end{theorem}

\begin{proof}[Proof of \cref{binthm}]
By \cref{lambdaboundlemmabin}, there exists a constant $C$ such that for all $\eps > C$, if $f \in B_n^{\bin}(r,X)$ and $X/2 \leq \Delta$ and $1 < \Delta$, then $R_f$ is $(\eps,\Delta)$-close to $L(r)$. By \cref{autlemma}, for a form $f \in B_n^{\bin}(r,X)$ with $X/2 \leq \Delta$, there are  $\ll_n X^{r_2-r_1}$ $\GL_2(\ZZ)$ conjugates of $f$ contained in $B_n^{\bin}(r,X)$. Thus, 
\begin{equation}
\begin{split}
\#\cF_n^{\bin}(L(r),\eps, X) &\gg_{n} \frac{\#\{B_n^{\bin}(r,X) \cap \Sym^n \ZZ^2 \mid \; X/2 \leq \Delta \leq X\}}{X^{r_2-r_1}} \\
&\gg_n \frac{\#\{B_n^{\bin}(r,X) \cap \Sym^n \ZZ^2\}}{X^{r_2-r_1}} \\
\label{final}
&\gg_{r,n} \frac{\exp[X, \sum_{i = 0}^n(n-i)r_1 + ir_2]}{\exp[X, r_2-r_1]}
\end{split}
\end{equation}
Now observe that:
\begin{equation}
\label{eval}
	\sum_{i = 0}^n(n-i)r_1 + ir_2 = {n + 1 \choose 2}(r_1 + r_2) = \frac{n+1}{2n-2}.
\end{equation}
Substituting \eqref{eval} into \eqref{final} proves the theorem.
\end{proof}

\section{Successive minima of quartic rings}
\label{quartic}
In this section let $R$ denote a nondegenerate quartic ring and let $C$ denote a cubic resolvent ring of $R$. Let $\Delta = \lv \Disc(R) \rv = \lv \Disc(C) \rv$.

\subsection{Additional results on quartic rings}
Our results are threefold: we explicitly compute $\poly_4(G)$ for various $G$, we compute various density functions, and then we discuss how various interesting families of quartic rings correspond to various vertices of $\poly_4(G)$.

\subsubsection{Computing $\poly_4(G)$}

\begin{theorem}
\label{containmentlemma}
For every permutation group $G \subseteq S_4$, the set $\poly_4(G)$ is a finite union of polytopes of dimension $3$. 
\end{theorem}

\cref{containmentlemma} is a natural generalization of \cite[Theorem 1.1.2]{me} when $n = 4$. See \cref{quarticstructuretheorem} for an even more explicit description of these polytopes. For many applications it is useful to provide explicit linear inequalities defining $\poly_4(G)$; in the following two theorems, we do this when $G = S_4$ or $G=D_4$. Let $p=(p_1,p_2,p_3,q_1,q_2)$ be the coordinates of $\RR^5$. By the definition of successive minima and Minkowski's second theorem, the set $\poly_4(G)$ satisfies the following constraints: 
\begin{equation}
\label{basicineq}
\begin{tabular}{c c c c}
$0 \leq p_1 \leq p_2 \leq p_3$; & $1/2 = p_1 + p_2 + p_3$; & $0 \leq q_1 \leq q_2$; & and $1/2 = q_1 + q_2$.
\end{tabular}
\end{equation}

\begin{theorem}
\label{s4polytopethm}
The set $\poly_4(S_4)$ consists of the points $p \in \RR^5$ satisfying \eqref{basicineq} and 
\begin{center}
\begin{tabular}{c c c}
 $q_2 \leq p_1 + p_3$; & $q_1 \leq 2p_1 $; & and $q_2 \leq 2p_2$.
\end{tabular}
\end{center}
\end{theorem}

\begin{theorem}
\label{d4polytopethm}
The set $\poly_4(D_4)$ consists of the points $p \in \RR^5$ satisfying \eqref{basicineq} and 
\begin{center}
\begin{tabular}{c c}
 $q_1 \leq 2p_1 $ & and $q_2 \leq 2p_2$.
\end{tabular}
\end{center}
\end{theorem}

\begin{theorem}
\label{c1polytopethm}
The set $\poly_4$ is the union of three polytopes. The first polytope consists of the points $p \in \RR^5$ satisfying \eqref{basicineq} and 
\begin{center}
\begin{tabular}{c c}
 $q_2 \leq p_2 + p_3$ and $q_1 \leq 2p_1$.
\end{tabular}
\end{center}
The second polytope consists of the points $p \in \RR^5$ satisfying \eqref{basicineq} and 
\begin{center}
\begin{tabular}{c c}
 $q_2 \leq p_1 + p_3$ and $q_1 \leq p_1 + p_2$.
\end{tabular}
\end{center}
The third polytope consists of the points $p \in \RR^5$ satisfying \eqref{basicineq} and 
\begin{center}
\begin{tabular}{c c}
 $q_2 \leq p_1 + p_3$ and $q_2 \leq 2p_2$.
\end{tabular}
\end{center}
\end{theorem}

\subsubsection{Density functions of quartic rings}
In this section we will prove \cref{fulldensitythmquartic} and \cref{s4densitythm} as well as some other results discussed below. We compute the density function for $D_4$ on the complement of $\poly_4(S_4)$; this is not an empty statement because $\poly_4(D_4)$ strictly contains $\poly_4(S_4)$.

\dfourdensitythm*

\subsection{Additional results on special families of quartic rings}
Unsurprisingly, rings with extremal successive minima tend to also be extremal in an \emph{algebraic} sense; in other words the vertices of our polytopes ``correspond'' to various interesting families of quartic rings. As in the cubic case, we first formalize what the word ``correspond'' means for families of quartic rings. We say a pair $(R,C)$ is $(\eps,X)$-close to $p$ if:
\begin{align*}
\max_{1\leq i \leq 3}\{\lv \log_{\Delta}\lambda_i(R) - p_i \rv\} &\leq \frac{\eps}{\log X} \\
\max_{1\leq i \leq 2}\{\lv \log_{\Delta}\lambda_i(C) - q_i \rv\} &\leq \frac{\eps}{\log X}
\end{align*}

Let $\cF$ be an infinite set equipped with a map $H \colon \cF \rightarrow \RR_{> 0}$ such that for all $X$, we have $\#\{f \in \cF \mid H(f) \leq X \} < \infty$. We call $H$ the \emph{height function} of $\cF$. Further suppose that $\cF$ is equipped with a map $\psi$ from $\cF$ to isomorphism classes of nondegenerate pairs of quartic rings equipped with a cubic resolvent ring. For a point $p \in \poly_4$, we say the triple \emph{$(\cF,H,\psi)$ corresponds to $p$} if 
\begin{equation}
\label{correspondquartic1}
\lim_{\eps \rightarrow \infty} \liminf_{X \rightarrow \infty}\frac{\#\{f \in \cF \; \mid \; H(f) \leq X \text{ and $\psi(f)$ is $(\eps,X)$-close to $p$}\}}{\#\{f \; \mid \; H(f) \leq X \}} = 1
\end{equation}
and there exists a constant $M$ such that for all $\eps > M$, we have
\begin{equation}
\label{correspondquartic2}
\liminf_{X \rightarrow \infty}\frac{\#\{(R,C) \in \cF(p,\eps,X) \; \mid \; (R,C) \in \psi(\cF)\}}{\#\cF(p,\eps,X)} > 0.
\end{equation}
When $H$ or $\psi$ are implicit, we simply write that $\cF$ corresponds to $p$.

\begin{restatable}{theorem}{genericgaloisgroupscorollary}
\label{genericgaloisgroupscorollary}
The set of isomorphism classes of $S_4$-rings (where $H$ is the absolute discriminant) corresponds to
\[
	(1/6,1/6,1/6,1/4,1/4).
\]
The set of isomorphism classes of $D_4$-rings (where $H$ is the absolute discriminant) corresponds to
\[
	(0,1/4,1/4,0,1/2).
\]
\end{restatable}

The points above are vertices of $\poly_4(S_4)$ and $\poly_4(D_4)$ respectively. 

\subsubsection{Vertices of $\poly_4(S_4)$}
We now discuss various special families of quartic rings that correspond to the vertices of $\poly_4(S_4)$. The vertices of $\poly_4(S_4)$ are:
\begin{enumerate}
\item $(1/6,1/6,1/6,1/4,1/4)$;
\item $(1/6,1/6,1/6,1/6,1/3)$;
\item $(1/8,1/8,1/4,1/4,1/4)$;
\item $(1/8,3/16,3/16,1/4,1/4)$;
\item $(1/10,1/5,1/5,1/5,3/10)$;
\item and $(1/12,1/6,1/4,1/6,1/3)$.
\end{enumerate}

We will give families that correspond to each vertex except $(1/8,3/16,3/16,1/4,1/4)$. Recall that we define the root height $H_r$ of a monic degree $n$ polynomial $f(x) = x^n + a_{1}x^2 + \dots + a_n \in \ZZ[x]$ by $H_r(f) = \max_i\{\lv a_i\rv^{1/i}\}$. Given a binary $n$-ic form $f(x,y) = a_0x^n + \dots + a_ny^n$, define the \emph{coefficient height of $f$} to be $H_c(f) = \max_i\{\lv a_i\rv\}$. Wood \cite{woodquartic} constructed a map from integral binary quartic forms into pairs of integral ternary quadratic forms; thus from an integral binary quartic form, we canonically obtain a quartic ring and a cubic resolvent ring. This construction is compatible with that of Birch and Merriman \cite{birch}, i.e. both constructions give isomorphic quartic rings.

\begin{theorem}
\label{monogenicquartictheorem}
Let $\cF_4^{\mon}$ be the set of integral monic quartic polynomials with nonzero discriminant. Then $(\cF_4^{\mon}, H_r)$ corresponds to $(1/12,1/6,1/4,1/6,1/3)$.
\end{theorem}
As in the monogenic cubic case, one might guess at the above result by guessing that for a monogenic quartic ring $\ZZ[\alpha]$ satisfying
\[
	\lv \alpha\rv = \min_{i \in \ZZ}\{\lv \alpha + i\rv\},
\]
a ``short'' basis of $\ZZ[\alpha]$ might ``look like'' $1,\alpha,\alpha^2,\alpha^3$. Moreover, the (unique) cubic resolvent ring of a monogenic quartic ring is monogenic as well.

\begin{theorem}
\label{binaryquartictheorem}
Let $\cF_4^{\bin}$ be the set of integral binary quartic forms with nonzero discriminant. Then $(\cF_4^{\bin}, H_c)$ corresponds to $(1/6,1/6,1/6,1/6,1/3)$.
\end{theorem}
One might guess at the above result by recalling from work of Wood \cite{woodquartic} that a pair $(R,C)$ arises from a binary quartic form if and only if $C$ is monogenic. We now discuss the other vertices of $\poly_4(S_4)$, first considering the vertex $(1/8,1/8,1/4,1/4,1/4)$. One might look at the coordinates of this vertex and immediately consider the family of nondegenerate quartic rings having a basis of the form $1,x,y,xy$; however this is not quite the right family to consider. This is because monogenic rings \emph{also} have a basis of the form $1,x,y,xy$, and for $\eps$ sufficiently large,
\[
	d_{\eps,4}((1/8,1/8,1/4,1/4,1/4)) = d_{\eps,4}((1/12,1/6,1/4,1/6,1/3)) = 3/4.
\]
In fact, these two vertices are adjacent in $\poly_4(S_4)$ and the entire edge between these two vertices evaluates to $3/4$ under the density function $d_{\eps,4}$. Thus, to restrict to the vertex $(1/8,1/8,1/4,1/4,1/4)$, we must impose one more condition. Consider triples $(R,x,y)$ where $R$ is a quartic ring and $x,y \in R$ have the property that $R = \ZZ\la 1,x,y,xy \ra$ and $\varphi(x+y) = \varphi(x) + \varphi(y)$ (here, $\varphi$ denotes the cubic resolvent map). We say two such triples $(R,x,y)$ and $(R',x',y')$ are \emph{isomorphic} if there exists a ring isomorphism $\alpha \colon R \rightarrow R'$ such that $\alpha(x) = x'$ and $\alpha(y) = y'$. We say such a triple is \emph{nondegenerate} if $R$ is nondegenerate. Let $\cF^{x,y,xy}$ be the set of isomorphism classes of nondegenerate such triples $(R,x,y)$.

We will embed $\cF^{x,y,xy}$ in the space of pairs of integral ternary quadratic forms, and use this embedding to define a height function on $\cF^{x,y,xy}$. In particular, consider the set of pairs of ternary quadratic forms of the form\footnote{The $\ast$ symbol means any real number is permitted to be in this entry, and $-$ means that this entry is determined by symmetry.}:
\begin{equation}
\label{matrixform1xyxyresolventintro}
  (A,B) = \begin{bmatrix}
    1 & 0 & \ast \\
    - &  0 & \ast  \\
    - & - & \ast
  \end{bmatrix},
  \begin{bmatrix}
    0 & 0 & \ast \\
    - &  1 & \ast  \\
    - & - & \ast
  \end{bmatrix}.
\end{equation}

Similarly, now consider quadruples $(R,x,y,\varepsilon)$ where $x,y \in R$, $\varepsilon = \pm 1$, and $1, x,y,\varepsilon xy$ is a basis of $R$. Again, we say such a quadruple is \emph{nondegenerate} if $R$ is nondegenerate. We say two such quadruples $(R,x,y,\varepsilon)$ and $(R',x',y',\varepsilon')$ are isomorphic if $\varepsilon = \varepsilon'$ and there exists a ring isomorphism $\alpha \colon R \rightarrow R'$ such that $\alpha(x) = x'$ and $\alpha(y) = y'$. 

\begin{restatable}{theorem}{onexyxyparam}
\label{1xyxyparam}
There is a canonical discriminant-preserving bijection between:
			\[
				\left \{ \begin{array}{c}
				\text{isomorphism classes of} \\
				\text{nondegenerate quadruples $(R,x,y,\varepsilon)$} \\
				\text{where $x,y \in R$ are such that } \\
				\text{$R = \ZZ\la 1,x,y,xy\ra$ and} \\
				\text{$\varphi(x+y) = \varphi(x) + \varphi(y)$}
				\text{ and $\eps = \pm 1$}
			\end{array}	 \right \}
			\longleftrightarrow
			\left \{ \begin{array}{c}
				\text{nondegenerate pairs of } \\
				\text{integral ternary quadratic forms} \\
				\text{of the form \eqref{matrixform1xyxyresolventintro}}
				\end{array}
			\right \}
			\]	
\end{restatable}

For a triple $T \in \cF^{x,y,xy}$, we attach a pair of nondegenerate ternary quadratic forms via the map above (setting $\varepsilon = 1$); define the \emph{height} to be 
\[
	H(T) \coloneqq \max\{\lv a_{13} \rv^2, \lv a_{23} \rv^2,\lv b_{13} \rv^2,\lv b_{23} \rv^2,\lv a_{33} \rv,\lv b_{33} \rv\},
\]
where here $a_{ij},b_{ij}$ are the entries of $(A,B)$; see \eqref{coordsofquarticmatrix} for more details\footnote{Another natural height function would be the $H((R,x,y)) \coloneqq \max\{\lv x\rv, \lv y\rv\}$. The author expects that with respect to this height function, the family $\cF^{x,y,xy}$ still corresponds to $(1/8,1/8,1/4,1/4,1/4)$.}. Moreover, observe that the theorem above canonically attaches a cubic resolvent ring to a triple $(R,x,y) \in \cF^{x,y,x^2}$ via Bhargava's parametrization of quartic rings. Refer to the forward direction of the map from \cref{1xyxyparam} as $\psi$; via Bhargava's parametrization of quartic rings \cite{BhargavaHCL3}, it can be interpreted as map from such triples to pairs of based quartic rings equipped with a based cubic resolvent ring.

\begin{theorem}
\label{1xyxythm}
The pair $(\cF^{x,y,xy},H,\psi)$ corresponds to $(1/8,1/8,1/4,1/4,1/4)$.
\end{theorem}

We now discuss the vertex $(1/10,1/5,1/5,1/5,3/10)$ of $\poly_4(S_4)$. The coordinates of this vertex suggest considering the family of quartic rings having a basis of the form $1,x,y,x^2$. We consider a slight variant; consider triples $(R,x,y)$ where $R$ is a quartic ring, $x \in R$, $y \in R/\ZZ$, and $x,y,x^2$ forms a basis of $R/\ZZ$. We say two such triples $(R,x,y)$ and $(R',x',y')$ are \emph{isomorphic} if there exists a ring isomorphism $\alpha \colon R \rightarrow R'$ with the property that $\alpha(x) = x'$ and $\alpha(y) \equiv \alpha(y') \Mod{\ZZ}$. We say such a triple is \emph{nondegenerate} if $R$ is nondegenerate. Let $\cF^{x,y,x^2}$ be the set of isomorphism classes of nondegenerate such triples $(R,x,y)$. We employ the same strategy as previously.

Consider the set of pairs of ternary quadratic forms of the form: 
\begin{equation}
\label{matrixform1xyx2intro}
  (A,B) = \begin{bmatrix}
    1 & 0 & \ast \\
    - &  \ast & \ast  \\
    - & - & \ast
  \end{bmatrix},
    \begin{bmatrix}
    0 & 1/2 & 0 \\
    - &  \ast & \ast  \\
    - & - & \ast
  \end{bmatrix}.
\end{equation}

We consider quintuples $(R,x,y,\varepsilon,\nu)$ where $x \in R$, $y \in R/\ZZ$, $\varepsilon = \pm 1$,  $\nu \in \{0,1\}$, and $x,y,x^2$ is a basis of $R/\ZZ$. We say such a quintuple is \emph{nondegenerate} if $R$ is nondegenerate. We say two such quintuples $(R,x,y,\varepsilon,\nu)$ and $(R',x',y',\varepsilon',\nu')$ are isomorphic if $\varepsilon = \varepsilon'$, $\nu = \nu'$, and there exists a ring isomorphism $\alpha \colon R \rightarrow R'$ such that $\alpha(x)=x'$ and $\alpha(y)  = y' \Mod{\ZZ}$.

\begin{restatable}{theorem}{onexyxtwoparam}
\label{1xyx2param}
There is a canonical discriminant-preserving bijection between:
			\[
				\left \{ \begin{array}{c}
				\text{isomorphism classes of} \\
				\text{nondegenerate quintuples $(R,x,y,\eps,\nu)$} \\
				\text{where $x \in R$ and $y \in R/\ZZ$ are such that } \\
				\text{$x,y,x^2$ is a basis of $R/\ZZ$}
				\text{ and $\varepsilon = \pm 1$ and $\nu \in \{0,1\}$}
			\end{array}	 \right \}
			\longleftrightarrow
			\left \{ \begin{array}{c}
				\text{nondegenerate pairs of } \\
				\text{integral ternary quadratic forms} \\
				\text{of the form \eqref{matrixform1xyx2intro}}
				\end{array}
			\right \}
			\]	
\end{restatable}

For a triple $T \in \cF^{x,y,x^2}$ we attach a nondegenerate pair of integral ternary quadratic forms via the map above (setting $\varepsilon = 1$ and $\nu = 0$); define the \emph{height} to be 
\[
	H(T) \coloneqq \max\{\lv a_{13} \rv^{2},\lv a_{22} \rv,\lv a_{23} \rv,\lv a_{33} \rv, \lv b_{22} \rv^{2},\lv b_{23} \rv^{2},\lv b_{2} \rv^{2} \}.
\]
As in the previous case, the theorem above canonically attaches a cubic resolvent ring to a triple $(R,x,y)$ in $\cF^{x,y,x^2}$; we will prove that this is the \emph{unique} cubic resolvent ring of $R$. Refer to the forward direction of map from the theorem above as $\psi$.

\begin{theorem}
\label{1xyx2thm}
The pair $(\cF^{x,y,x^2},H,\psi)$ corresponds to $(1/10,1/5,1/5,1/5,3/10)$.
\end{theorem}

\subsection{The parametrization of quartic rings}
\label{quarticbackground}
Our primary tool will be Bhargava's parametrization of quartic rings, which we now briefly summarize. Let $V_{\RR} = \RR^2 \otimes \Sym^2 \RR^3$. We write an element of $(A,B) \in V_{\RR}$ as a pair of $3\times 3$ real symmetric matrices as follows: 
\begin{equation}
\label{coordsofquarticmatrix}
	2(A,B) =\Bigg(\begin{bmatrix}
    2a_{11} & a_{12} & a_{13}  \\
    a_{12} & 2a_{22} & a_{23}  \\
    a_{13} & a_{23} & 2a_{33} 
  \end{bmatrix},
    \begin{bmatrix}
    2b_{11} & b_{12} & b_{13}  \\
    b_{12} & 2b_{22} & b_{23}  \\
    b_{13} & b_{23} & 2b_{33} 
  \end{bmatrix}\Bigg).
\end{equation}

An element $(A,B)$ is said to be \emph{integral} if $a_{ij},b_{ij} \in \ZZ$, and the lattice of integral elements is denoted $V_{\ZZ}$. The group $G_{\RR} = \GL_2(\RR) \times \GL_3(\RR)$ acts on $V_{\RR}$ in the following way; an element $(g_2,g_3) \in G_{\RR}$ sends $(A,B)$ to
\[
	(rg_3Ag_3^t + sg_3Bg_3^t,tg_3Ag_3^t + ug_3Bg_3^t)
\]
where we write
\[
	g_2 = \begin{bmatrix}
    r & s \\
    t & u 
  \end{bmatrix}.
\]
The subgroup $G_{\ZZ}$ acts on $V_{\RR}$ and preserves the lattice $V_{\ZZ}$. The action of $G_{\ZZ}$ on $V_{\RR}$ has a unique polynomial invariant given by the formula $\Disc(4\dett(Bx-Ay))$. This fundamental invariant is known as the \emph{discriminant} and is denoted $\Disc(A,B)$. The orbits of $G_{\ZZ}$ acting on $V_{\ZZ}$ parametrize quartic rings equipped with a cubic resolvent ring; see \cite{BhargavaHCL3} for more details.

\begin{theorem}[Bhargava, \cite{BhargavaHCL3}]
There is a canonical, discriminant-preserving bijection between $V_{\ZZ}$ and isomorphism classes of pairs $(R,C)$ where $R$ is a based quartic ring and $C$ is a based cubic resolvent ring of $R$. Under the action of $G_{\ZZ}$, this bijection descends to a bijection between the set of $G_{\ZZ}$ orbits of $V_{\ZZ}$ and the set of isomorphism classes of pairs $(R,C)$ where $R$ is a quartic ring and $C$ is a cubic resolvent ring of $R$.
\end{theorem}

We now describe a few important details of this bijection. A cubic resolvent ring $C$ of a quartic ring $R$ comes equipped with a quadratic map $\varphi \colon R/\ZZ \rightarrow C/\ZZ$ called a \emph{cubic resolvent map}. Choosing bases for $R/\ZZ$ and $C/\ZZ$ gives rise to an element of $V_{\ZZ}$. Conversely, for an element $(A,B) \in V_{\ZZ}$, Bhargava constructed certain polynomials in the $a_{ij},b_{ij}$ that form the structure coefficients for a basis of $R$. Moreover, the binary cubic form $4\dett(Bx-Ay)$ via \eqref{cubiccoeff} gives structure coefficients for a basis of $C$.

\subsection{Reduction theory of quartic rings}
We now fix some notation we use throughout this section. Unless otherwise specified, $(A,B)$ will denote a nondegenerate element of $V_{\ZZ}$ and $(R,C)$ will denote the corresponding quartic ring and cubic resolvent ring respectively. Let $\Delta = \lv \Disc(A,B) \rv$. Analogously to the cubic case, in this section we develop the relationship between:
\begin{enumerate}
\item pairs $(R,C)$ equipped with ``small'' bases;
\item and pairs $(A,B)$ with ``small'' coefficients.
\end{enumerate}

Our proofs are almost identical to those in the cubic case.

\subsubsection{``Short'' bases of $(R,C)$ gives rise to matrices $(A,B)$ with ``small'' coefficients}
\begin{lemma}
\label{matrixbound}
Let $\{1,v_1,v_2,v_3\}$ and $\{1,w_1, w_2\}$ be Minkowski bases of any nondegenerate quartic ring $R$ and any cubic resolvent ring $C$ of $R$. Let $(A,B)$ be the corresponding matrices given by the cubic resolvent map. Then $\lv a_{ij} \rv \ll \lv v_i\rv \lv v_j \rv/ \lv w_1 \rv$ and $\lv b_{ij} \rv \ll \lv v_i\rv \lv v_j \rv/\lv w_2 \rv$.
\end{lemma}
\begin{proof}
Let $\sigma_1,\dots, \sigma_4$ be the $4$ nonzero homomorphisms of $R$ into the complex numbers.
The cubic resolvent map $\varphi \colon R/\ZZ \otimes R/\ZZ \rightarrow \frac{1}{2} C/\frac{1}{2}\ZZ$ is given by
\[
	\varphi(x \otimes y) \rightarrow \frac{1}{2}(\sigma_1(x)\sigma_2(y) + \sigma_2(x)\sigma_1(y)  + \sigma_3(x)\sigma_4(y) + \sigma_3(y)\sigma_4(x))
\]
For an element $z \in C \otimes \QQ$, let $\la w_i\ra z$ denote the coefficient of $z$ with respect to the basis  $\{1,w_1, w_2\}$. Because $\{1,v_1,v_2,v_3\}$ and $\{1,w_1, w_2\}$ are Minkowski bases,
\begin{align*}
\lv a_{ij} \rv &\asymp \la w_1\ra(\varphi(v_i \otimes v_j)) \\
&\ll \frac{\lv \varphi(v_i \otimes v_j) \rv}{\lv w_1 \rv} \\
&\ll \frac{\lv v_i \rv\lv v_j \rv}{\lv w_1 \rv}.
\end{align*}
An analogous argument shows that $\lv b_{ij} \rv \ll \lv v_i\rv \lv v_j \rv/\lv w_2 \rv$.
\end{proof}

\begin{definition}
For any $p=(p_1,p_2,p_3,q_1,q_2) \in \RR^5$ and any positive real number $X$, define $B(p,X)$ to be the set of elements $(A,B) \in V_{\RR}$ such that $a_{ij} \leq X^{p_i+p_j - q_1}$ and $b_{ij} \leq X^{p_i+p_j - q_2}$.     
\end{definition}

Let $p=(p_1,p_2,p_3,q_1,q_2) \in \RR^5$ denote a point. Moreover $\eps$ will denote a positive real number, $M,D$ will denote real numbers $\geq 1$, and $X$ will denote a real number $>1$. Recall that we say the pair $(R,C)$ is $(\eps,X)$-close to $p$ if
\begin{align*}
\max_{i = 1,2,3}\{\lv \log_{\Delta}\lambda_i(R) - p_i \rv\} &\leq \frac{\eps}{\log X} \\
\max_{i = 1,2}\{\lv \log_{\Delta}\lambda_i(C) - q_i \rv\} &\leq \frac{\eps}{\log X}.
\end{align*}

\begin{definition}
We say $((A,B),p)$ are $(\eps,X)$-close if the corresponding rings $(R,C)$ and $p$ are $(\eps,X)$-close. We say $(A,B)$ \emph{comes from Minkowski bases} if there exist Minkowski bases of the corresponding rings $(R,C)$ giving rise to $(A,B)$.    
\end{definition}

\begin{lemma}
\label{ublemma1}
There exists $M$ such that for every pair $(A,B)$, point $p$, and positive real number $\eps$, the following statement holds: if $\Delta > 1$ and $((A,B),p)$ are $(\eps, \Delta)$-close and $(A,B)$ comes from Minkowski bases, then $(A,B) \in Me^{3\eps}B(p,\Delta)$.
\end{lemma}
\begin{proof}
Let $\{1,v_1,v_2,v_3\}$ and $\{1,w_1,w_2\}$ be the Minkowski bases of $(R,C)$ giving rise to $(A,B)$. 
We have
\begin{align*}
    \lv a_{i,j} \rv &\ll \frac{\lv v_i \rv \lv v_j \rv}{\lv w_1 \rv} && \text{by \cref{matrixbound}} \\
    &\asymp \frac{\lambda_i(R)\lambda_j(R)}{\lambda_1(C)} && \\
    &\leq e^{3\eps}\Delta^{p_i+p_j-q_1} && \text{because $((A,B),p)$ is $(\eps, \Delta)$-close}.
\end{align*}
An analogous argument shows that $\lv b_{i,j} \rv \ll e^{3\eps}\Delta^{p_i+p_j-q_2}$.
\end{proof}

\subsubsection{Matrices $A,B$ with ``small'' coefficients gives rise to ``short'' bases of $R,C$}
For a pair $(A,B)$, let $\{1,v_1,v_2,v_3\}$ and $\{1,w_1,w_2\}$ denote the explicit bases of $R$ and $C$ obtained from the structure coefficients.
\begin{lemma}
\label{quarticboundlemma}
For every $p$ satisfying \eqref{basicineq}, every $(A,B)$, and every $M, X$ the following statement holds: if $(A,B) \in MB(p,X)$, then $\lv v_i \rv \ll_M X^{p_i}$ for $i = 1,2,3$ and $\lv w_i \rv \ll_M X^{q_i}$ for $i = 1,2$.
\end{lemma}
\begin{proof}
The binary cubic form corresponding to the basis $\{1,w_1,w_2\}$ is given by $4\dett(Bx - Ay)$. \cref{reducedcubicbasis} implies that $\lv w_i \rv \ll_M X^{q_i}$ for $i = 1,2$. The formulae in \cite{BhargavaHCL3} of the explicit structure coefficients for the basis $\{1=v_0,v_1,v_2,v_3\}$ imply that if $(A,B) \in MB(p,X)$ then the $v_j$-coefficient of $v_i^2$ is $\ll_M X^{2p_i - p_j}$. Let $p_0 = 0$. Thus,
\begin{equation}
\label{mainubprelim}
\lv v_i \rv^2 \ll_M \max_{j = 0,1,2,3}\{X^{2p_i-p_j}\lv v_j\rv\}.
\end{equation}
Applying \cref{generalboundlemma} completes the proof.
\end{proof}

\begin{lemma}
\label{lambdaboundlemmaquartic}
For all $p$ satisfying \eqref{basicineq} and all $(A,B)$, $M$, and $X$ the following statement holds: if $(A,B) \in MB(p,X)$ and $\frac{X}{2} \leq \Delta$ then 
\[
	 \lambda_i(R) \asymp_M \Delta^{p_i}.
\]
for $i = 1,2,3$ and
\[
	 \lambda_i(C) \asymp_M \Delta^{q_i}.
\]
for $i = 1,2$.
\end{lemma}
\begin{proof}
\cref{lambdaboundlemma} immediately implies that $\lambda_i(C) \asymp_M \Delta^{q_i}$ for $i = 1,2$. We have
\begin{align*}
\lambda_1(R)\lambda_2(R)\lambda_3(R) &\ll \lv v_1\rv\lv v_2\rv\lv v_3\rv && \text{by definition of successive minima} \\
&\ll_M X^{p_1}X^{p_2}X^{p_3} && \text{by \cref{quarticboundlemma}}\\
&= X^{1/2} && \text{because $p_1 + p_2 + p_3 = 1/2$}\\
&\asymp_M \Delta^{1/2} && \text{because $\Delta \geq X/2$ and $(A,B) \in MB(p,X)$} \\
&\asymp \lambda_1(R)\lambda_2(R)\lambda_3(R). && \text{by Minkowski's second theorem}
\end{align*}
Thus $\lambda_i(R) \asymp_M \Delta^{p_i}$ for $i = 1,2,3$.
\end{proof}

\subsection{Bounding the number of different matrices with ``small coefficients'' giving rise to isomorphic $(R,C)$}
\begin{lemma}
\label{ublemma2}
There exists $D$ such that for all $p$ satisfying \eqref{basicineq}, all $(A,B)$, and all $M,X$ the following statement holds: if $(A,B) \in MB(p,X)$ then
\[
	\#\{(A',B') \in DMB(p,X) \cap V_{\ZZ} \mid \; (R,C) \simeq (R',C') \} \geq \exp[X,(p_3-p_2)+(p_3-p_1)+(p_2-p_1)+(q_2-q_1)]\}
\]
where here $(R',C')$ is the quartic ring and cubic resolvent ring arising from the pair $(A',B')$.
\end{lemma}
\begin{proof}
Recall that $V_{\ZZ}$ admits an action of $G_{\ZZ} = \GL_2(\ZZ)\times \GL_3(\ZZ)$. Act on $(A,B)$ by elements of $G_{\ZZ}$ of the form
\[
   \Bigg(\begin{bmatrix}
    1 & m \\
    0 & 1 \\
  \end{bmatrix}, \begin{bmatrix}
    1 & 0 & 0  \\
    n_{12} & 1 & 0   \\
    n_{13} & n_{23} & 1  \\
  \end{bmatrix}\Bigg)
\]
where $\lv n_{ij}\rv \leq \lceil X^{p_j-p_i} \rceil$ and $\lv m \rv \leq \lceil X^{q_2-q_1} \rceil$.
\end{proof}

\begin{lemma}
\label{reducedquarticcorol}
For every $p$ satisfying \eqref{basicineq} and every $(A,B)$, $X$, and $M$, the following statement holds: if
\[
	(A,B) \in MB(p,X)
\]
and $X/2 \leq \Delta \leq X$, then there are $\ll_M \exp[X,(p_3-p_2)+(p_3-p_1)+(p_2-p_1)+(q_2-q_1)]$ distinct elements $(A',B') \in MB(p,X) \cap V_{\ZZ}$ such that $(R',C') \simeq (R,C)$.
\end{lemma}
\begin{proof}
Let $\{v_0=1,v_1,v_2,v_3\}$ and $\{w_0=1,w_1,w_2\}$ be the explicit bases of $R$ and $C$ obtained from $(A,B)$. By \cref{quarticboundlemma}, we have $\lv v_i \rv \ll_M X^{p_i}$ for $i = 1,2,3$ and  $\lv w_i \rv \ll_M X^{q_i}$ for $i = 1,2$. Because
\[
	X^{1/2} \asymp \Delta^{1/2} \ll \lv v_1 \rv\lv v_2 \rv\lv v_3 \rv \ll_M X^{p_1+p_2+p_3} = X^{1/2}
\]
and
\[
	X^{1/2} \asymp \Delta^{1/2} \ll \lv w_1 \rv\lv w_2 \rv \ll_M X^{q_1+q_2} = X^{1/2},
\]
we have $\lv v_i \rv \asymp_M X^{p_i}$ for $i = 1,2,3$ and  $\lv w_i \rv \asymp_M X^{q_i}$ for $i = 1,2$.

Suppose $(A',B') \in MB(p,X) \cap V_{\ZZ}$ is such that $(R',C') \simeq (R,C)$. Let $\{1,v'_1,v'_2,v'_3\}$ and $\{1,w'_1,w'_2\}$ denote the bases of $R'$ and $C'$ given by $(A',B')$. By \cref{quarticboundlemma}, we have $\lv v'_i \rv \ll_M X^{p_i}$ for $i = 1,2,3$ and  $\lv w'_i \rv \ll_M X^{q_i}$ for $i = 1,2$. Moreover, write
\begin{align*}
v'_i &= c_{i0} + \sum_{j = 0}^3 c_{ij} v_j \\
w'_i &= d_{i0} + \sum_{j = 0}^2 d_{ij} w_j.
\end{align*}
The pair $(A',B')$ is determined by the choices of $c_{ij}$ for $1 \leq i,j \leq 3$ and $d_{ij}$ for $1 \leq i,j \leq 2$. We have:
\begin{align*}
X^{p_i} \asymp_M \lv v'_i \rv &\asymp_M \max_{0 \leq j < 4}\{\lv c_{ij}\rv\lv v_j \rv\} \asymp_M \max_{0 \leq j < n}\{\lv c_{ij}\rv X^{p_j}\} \\
X^{q_i} \asymp_M \lv w'_i \rv &\asymp_M \max_{0 \leq j < 3}\{\lv c_{ij}\rv\lv w_j \rv\} \asymp_M \max_{0 \leq j < n}\{\lv c_{ij}\rv X^{q_j}\}
\end{align*}
where $p_0 = q_0 = 0$.
If $j \leq i$, the number of choices for $c_{ij}$ and $d_{ij}$ is $\ll_M 1$. If $j > i$, the number of choices for $c_{ij}$ is $\ll_M X^{p_j-p_i}$ and the number of choices for $d_{ij}$ is $\ll_M X^{q_j-q_i}$. Hence, there are $\ll_M \exp[X,(p_3-p_2)+(p_3-p_1)+(p_2-p_1)+(q_2-q_1)]$ distinct integral $(A',B') \in MB(p,X)$ such that $(R',C') \simeq (R,C)$. 
\end{proof}

\subsection{Computing $\poly_4(G)$}
In this section, we will compute $\poly_4(G)$ and $\poly_4$. Let $p = (p_1,p_2,p_3,q_1,q_2) \in \RR^5$ be any point in $\RR^5$. Let $K$ denote a quartic \'etale algebra over $\QQ$ and $L$ denote the cubic resolvent algebra of $K$. Let $\varphi \colon K/\QQ \otimes K/\QQ \rightarrow L/\QQ$ denote the symmetric bilinear form given by the cubic resolvent map. A \emph{flag $(\cF,\cG)$ of $(K,L)$} is a pair of sequences $\cF = (F_0,F_1,F_2,F_3)$ and $\cG = (G_0,G_1,G_2)$ of $\QQ$-vector spaces such that
\[
	\QQ = F_0 \subset F_1 \subset F_2 \subset F_3 = K
\]
\[
	\QQ = G_0 \subset G_1 \subset G_2 = L
\]  
and $\dim_{\QQ} F_i = i+1$ and $\dim_{\QQ} G_i = i+1$. For a positive integer $n$, let $[n]$ refer to the set $\{0,\dots,n-1\}$. Associate to a flag $(\cF,\cG)$ the function
\begin{align*}
  T_{\cF,\cG} \colon [4]\times[4] & \longrightarrow [3] \\
  (i\;\;,\;\;j ) \; & \longmapsto \min\{k \in [3] \mid \varphi\la F_i \otimes F_j\ra \subseteq G_k\}.
\end{align*}

\begin{definition}
Let $P_{\cF,\cG}$ be the set of $p = (p_1,p_2,p_3,q_1,q_2) \in \RR^5$ satisfying \eqref{basicineq} and the following condition: for all $1 \leq i,j \leq 3$, if $T_{\cF,\cG}(i,j) > 0$ then
\[
	q_{T_{\cF,\cG}(i,j)} \leq p_i + p_j.
\]
\end{definition}

The proof of the following theorem is almost exactly analogous to that of \cite[Theorem 1.1.2]{me}; the reader is encouraged to refer to the proof of \cite[Theorem 1.1.2]{me} for more details. We will need the following definition.

\begin{theorem}
\label{quarticstructuretheorem}
We have $\poly_4(G) = \bigcup_{(\cF,\cG)}P_{\cF,\cG}$ where $(\cF,\cG)$ ranges across flags of pairs $(K,L)$ with the property that $\Gal(K) \simeq G$ as a permutation group.
\end{theorem}
\begin{proof}
We will first show that $\poly_4(G)$ is contained in $\bigcup_{(\cF,\cG)}P_{\cF,\cG}$. By Minkowski's second theorem and the definition of successive minima, $\poly_4(G)$ satisfies the linear inequalities given in \eqref{basicineq}. Let $R$ be any nondegenerate quartic $G$-ring and let $C$ be a cubic resolvent ring of $R$. Let $K = R \otimes \QQ$ and let $L = C \otimes \QQ$. Picking Minkowski bases $\{v_0=1,v_1,v_2,v_3\}$ and $\{w_0=1,w_1,w_2\}$ of $R$ and $C$ respectively gives rise to a flag $(\cF,\cG)$ of $(K,L)$ given by $F_i \coloneqq \QQ\la v_0,\dots,v_{i}\ra$ and $G_i \coloneqq \QQ\la w_0,\dots,w_{i}\ra$. For every $1 \leq i,j \leq 3$ such that $T_{\cF,\cG}(i,j) > 0$, we have
\[
	\varphi\la \ZZ\la v_0,\dots,v_i \ra \otimes \ZZ\la v_0,\dots,v_j \ra \ra \subseteq \ZZ\la w_0,\dots, w_{T_{\cF,\cG}(i,j)} \ra
\]
and 
\[
	\varphi\la \ZZ\la v_0,\dots,v_i \ra \otimes \ZZ\la v_0,\dots,v_j \ra \ra \not \subseteq \ZZ\la w_0,\dots, w_{T_{\cF,\cG}(i,j)-1} \ra.
\]
Therefore, there exists $0 \leq i' \leq i$ and $0 \leq j' \leq j$ such that $\la w_{T_{\cF,\cG}(i,j)} \ra \varphi(v_{i'} \otimes v_{j'}) \neq 0$. So,
\[
	\lv w_{T_{\cF,\cG}(i,j)} \rv \ll \lv \varphi(v_{i'} \otimes v_{j'}) \rv \ll \lv v_{i'}\rv\lv v_{j'} \rv \ll \lv v_i\rv\lv v_j\rv.
\]
This implies that $\poly_4(G)$ is contained in the set $\bigcup_{(\cF,\cG)}P_{\cF,\cG}$ where $(\cF,\cG)$ ranges across flags of $(K,L)$ with the property that $\Gal(K) \simeq G$.

Now let $(\cF,\cG)$ be a flag of a pair $(K,L)$ with the property that that $\Gal(K) \simeq G$. Choose bases $\{v_0=1,v_1,v_2,v_3\}$ and $\{w_0=1,w_1,w_2\}$ of $K$ and $L$ giving rise to the flag $(\cF,\cG)$. Writing the symmetric bilinear map $\varphi$ with respect to these bases gives a pair $(A,B) \in V_{\RR}$, which we write as
\[
	2(A,B) =\Bigg(\begin{bmatrix}
    2a_{11} & a_{12} & a_{13}  \\
    a_{12} & 2a_{22} & a_{23}  \\
    a_{13} & a_{23} & 2a_{33} 
  \end{bmatrix},
    \begin{bmatrix}
    2b_{11} & b_{12} & b_{13}  \\
    b_{12} & 2b_{22} & b_{23}  \\
    b_{13} & b_{23} & 2b_{33} 
  \end{bmatrix}\Bigg).
\]
For a positive integer $M$, define the pair of matrices $(A_M,B_M)$ to be 
\[
	2A_M = \Bigg(\begin{bmatrix}
    2a_{11}M^{p_1 + p_1 - q_1} & a_{12}M^{p_1 + p_2 - q_1} & a_{13}M^{p_1 + p_3 - q_1}  \\
    a_{12}M^{p_1 + p_2 - q_1} & 2a_{22}M^{p_2 + p_2 - q_1} & a_{23}M^{p_2 + p_2 - q_1}  \\
    a_{13}M^{p_1 + p_3 - q_1} & a_{23}M^{p_2 + p_3 - q_1} & 2a_{33}M^{p_3 + p_3 - q_1} 
  \end{bmatrix}
\]
\[
	2B_M = \begin{bmatrix}
    2b_{11}M^{p_1 + p_1 - q_2} & b_{12}M^{p_1 + p_2 - q_2} & b_{13}M^{p_1 + p_3 - q_2}  \\
    b_{12}M^{p_1 + p_2 - q_2} & 2b_{22}M^{p_2 + p_2 - q_2} & b_{23}M^{p_2 + p_3 - q_2}  \\
    b_{13}M^{p_1 + p_3 - q_2} & b_{23}M^{p_2 + p_3 - q_2} & 2b_{33}M^{p_3 + p_3 - q_2} 
  \end{bmatrix}\Bigg).
\]
For a rational point $p$ in the relative interior of $P_{\cF,\cG}$, one may check that there exists an infinite set $\cM \subseteq \ZZ_{>0}$ such that for all $M \in \cM$, the pair $(A_M,B_M)$ is contained in $V_{\ZZ}$. Each $M \in \cM$ gives rise to a distinct quartic ring and cubic resolvent ring $(R_M,C_M)$ with absolute discriminant $\Delta_M$; for each $M \in \cM$, write the bases of $R_M$ and $C_M$ obtained from $(A_M, B_M)$ by $\{1,v_{1,M},\dots,v_{3,M}\}$ and $\{1,w_{1,M},w_{2,M}\}$ respectively. Then for $M,N \in \cM$, we have:
\begin{align*}
N\Delta_M &= M\Delta_N \\
N^{p_i}v_{i,M} &= M^{p_i}v_{i,N} \\
N^{q_i}w_{i,M} &= M^{q_i}w_{i,N} 
\end{align*}
Therefore, 
\begin{align*}
\lim_{M \in \cM}\Big(\log_{ \Delta_M }\lambda_{1}(R_M),\dots,\log_{ \Delta_M }\lambda_{3}(R_M)\Big) &= (p_1,p_2,p_3) \\
\lim_{M \in \cM}\Big(\log_{ \Delta_M }\lambda_{1}(C_M),\log_{ \Delta_M }\lambda_{2}(C_M)\Big) &= (q_1,q_2)
\end{align*}
Therefore, $p \in \poly_4(G)$. Because $\poly_4(G)$ is defined to be a set of limit points, it must be closed; therefore $P_{\cF,\cG}$ is contained in $\poly_4(G)$, which completes the proof of the theorem.
\end{proof}

\begin{proof}[Proof of \cref{containmentlemma}]
\cref{quarticstructuretheorem} expresses $\poly_4(G)$ as a finite union of polytopes. A computation shows that these polytopes all have dimension $3$.
\end{proof}

\begin{proof}[Proof of \cref{c1polytopethm}]
Let $p \in \poly_4$ be any point. Suppose that $q_2 > p_1 + p_3$. By \cref{quarticstructuretheorem}, $\poly_4(G) = \bigcup_{(\cF,\cG)}P_{\cF,\cG}$ where $(\cF,\cG)$ ranges across flags of pairs $(K,L)$ where $K$ is a quartic \'etale algebra. Choose such a flag $(\cF,\cG)$ such that $P_{\cF,\cG}$ contains $p$. Choose bases of $K$ and $L$ giving rise to that flag; in these bases, the cubic resolvent map is a pair $(A,B) \in V_{\QQ}$ of the form
\[
	(A,B) = \Bigg(\begin{bmatrix}
    * & * & *  \\
    - & * & *  \\
    - & - & * 
  \end{bmatrix},
    \begin{bmatrix}
    0 & 0 & 0  \\
    - & * & *  \\
    - & - & * 
  \end{bmatrix}\Bigg).
\]
If $q_2 > p_2 + p_3$ then $b_{22} = b_{33} = 0$. If $q_1 > 2p_1$ then $a_{11}= 0$. In what ensues, we will show that neither of these events occur. 

Because $\Disc(A,B) \neq 0$, the binary cubic form $\dett(Bx + Ay)$ does not have a double root, and hence $\rank(B) > 1$. Therefore we cannot have both $b_{22} = 0$ and $b_{23} = 0$. Now let $Q_A$ and $Q_B$ be the conics in $\PP^2$ given by $A$ and $B$ respectively. Because $\Disc(A,B) \neq 0$, the intersection $Q_A \cap Q_B$ is $4$ geometric points in $\PP^2$ with the property that no $3$ are collinear. If $a_{11} = 0$ then $[1 \colon 0 \colon 0] \in Q_A \cap Q_B$; however because $\rank(B) = 2$, the conic $Q_B$ is a union of two lines intersecting at the point $[1 \colon 0 \colon 0]$, which forces three of the $4$ points in $Q_A \cap Q_B$ to be collinear, contradiction. Hence, $a_{11} \neq 0$. 

A computation shows that there exists $(A,B) \in V_{\QQ}$ of the form
\[
	(A,B) = \Bigg(\begin{bmatrix}
    * & * & *  \\
    - & * & *  \\
    - & - & * 
  \end{bmatrix},
    \begin{bmatrix}
    0 & 0 & 0  \\
    - & 0 & *  \\
    - & - & * 
  \end{bmatrix}\Bigg)
\]
generating a nondegenerate quartic ring. The pair $(A,B)$ gives rise to a flag $(\cF,\cG)$ via the explicit structure coefficients. It is easy to see that $P_{\cF,\cG}$ is precisely the first polytope in the theorem statement. Therefore, the first polytope specified in the theorem statement is contained in $\poly_4$ and the subset of $\poly_4$ with $q_2 > p_1 + p_3$ is contained in the first polytope specified in the theorem statement. 

Now suppose that $p \in \poly_4(S_4)$ is any point with $q_2 \leq p_1 + p_3$ and $q_1 > 2p_1$. Let $(A,B) \in V_{\QQ}$ be the set of matrices arising from an appropriate flag generating a nondegenerate quartic \'etale algebra; then $a_{11} = b_{11} = 0$. If  $q_2 > 2p_2$ then further $b_{12} = b_{22} = 0$. So, the quadric $Q_B$ contains the line $z = 0$, where $x,y,z$ are the coordinates of $\PP^2$; thus we may suppose two distinct points on the line $z = 0$ are contained in the intersection $Q_A \cap Q_B$; hence $a_{12} \neq 0$ so $q_1 \leq p_1 + p_2$. Therefore if $q_2 \leq p_1 + p_3$ and $q_1 > 2p_1$ then $q_1 \leq p_1 + p_2$ or $q_2 \leq 2p_2$. This shows that the part of $\poly_4(S_4)$ not lying in the first polytope stated in the theorem statement is contained in the union of the second polytope and third polytope stated in the theorem statement. Similarly, a computation shows that there exists $(A,B) \in V_{\QQ}$ of the form
\[
	(A,B) = \Bigg(\begin{bmatrix}
    0 & * & *  \\
    - & * & *  \\
    - & - & * 
  \end{bmatrix},
    \begin{bmatrix}
    0 & 0 & *  \\
    - & 0 & *  \\
    - & - & * 
  \end{bmatrix}\Bigg)
\]
and
\[
	(A,B) = \Bigg(\begin{bmatrix}
    0 & 0 & *  \\
    - & * & *  \\
    - & - & * 
  \end{bmatrix},
    \begin{bmatrix}
    0 & 0 & *  \\
    - & * & *  \\
    - & - & * 
  \end{bmatrix}\Bigg)
\]
generating a nondegenerate quartic ring, so \cref{quarticstructuretheorem} implies that the second and third polytopes in the theorem statement are contained in $\poly_4$. 
\end{proof}

\begin{proof}[Proof of \cref{d4polytopethm}]
As in the proof of \cref{c1polytopethm}, to show that $\poly_4(D_4)$ is contained in the given region, it suffices to show that for any element $(A,B) \in V_{\QQ}$ giving rise to a quartic \'etale algebra $K$, we have:
\begin{enumerate}
\item if $a_{11} = b_{11} = 0$, then $K$ is reducible;
\item if $b_{11} = b_{12} = b_{22} = 0$, then $K$ is reducible.
\end{enumerate}
Suppose $K$ is a quartic \'etale algebra given by a pair $(A,B) \in V_{\QQ}$, and let $Q_A$ and $Q_B$ be the quadratic forms in $\PP^2_{\QQ}$. Then
\[
	K \simeq \Gamma(V(Q_A,Q_B)).
\]
If $a_{11} = b_{11} = 0$ then $[1 \colon 0 \colon 0]$ is a common intersection of the conics defined by $A$ and $B$, which implies that $K$ is reducible. Similarly, if $b_{11} = b_{12} = b_{22} = 0$ then the conic by $B$ in $\PP^2$ is of the form $axz + bxz + cz^2 = z(ax + by + cz)$, where here $x,y,z$ are the coordinates of $\PP^2$; this implies that $K$ is reducible. Therefore, the set $\poly_4(D_4)$ is contained in the region specified in the theorem statement. 

A computation shows that there exists $(A,B) \in V_{\QQ}$ of the form
\[
	(A,B) = \Bigg(\begin{bmatrix}
    * & * & *  \\
    - & * & *  \\
    - & - & * 
  \end{bmatrix},
    \begin{bmatrix}
    0 & 0 & 0  \\
    - & * & *  \\
    - & - & * 
  \end{bmatrix}\Bigg)
\]
generating a $D_4$-quartic field, so \cref{quarticstructuretheorem} implies the region in the theorem statement is contained in $\poly_4(D_4)$.
\end{proof}

\begin{proof}[Proof of \cref{s4polytopethm}]
Let $p \in \poly_4(S_4)$ be any point. The proof of \cref{d4polytopethm} shows that $q_1 \leq 2p_1$ and $q_2 \leq 2p_2$. As in the proof of \cref{c1polytopethm}, to show that $\poly_4(S_4)$ is contained in the given region, it suffices to show that for any element $(A,B) \in V_{\QQ}$ giving rise to a quartic \'etale algebra $K$, we have that if $b_{11} = b_{12} = b_{13} = 0$ then $\Gal(K) \not \simeq S_4$. If $b_{11} = b_{12} = b_{13} = 0$ then $(1,0)$ is a root of the cubic resolvent polynomial $\dett(By - Ax)$, which implies that the cubic resolvent algebra of $K$ is reducible, which implies that  $\Gal(K) \not \simeq S_4$. Therefore, the set $\poly_4(S_4)$ is contained in the region specified in the theorem statement. 

A computation shows that there exists $(A,B) \in V_{\QQ}$ of the form
\[
	(A,B) = \Bigg(\begin{bmatrix}
    * & * & *  \\
    - & * & *  \\
    - & - & * 
  \end{bmatrix},
    \begin{bmatrix}
    0 & 0 & *  \\
    - & * & *  \\
    - & - & * 
  \end{bmatrix}\Bigg)
\]
generating a $S_4$-quartic field, so \cref{quarticstructuretheorem} implies the region in the theorem statement is contained in $\poly_4(S_4)$.
\end{proof}

\subsection{Computing density functions of quartic rings}
We now define some notation we use throughout this section to shorten our formulae. For a point $p \in \RR^5$, let
\[
	f(p) \coloneqq 1+\sum_{1 \leq i \leq j \leq 3, 1 \leq k \leq 2}\max\{0,q_k - p_i - p_j\}
\]
and
\[
	g(p) \coloneqq (p_3-p_2) + (p_3-p_1) + (p_2-p_1) + (q_2-q_1).
\]
As before, $\eps$ will denote a positive real number, $M,D$ will denote real numbers $\geq 1$, $X$ will denote a real number $> 1$, and $(A,B)$ will denote a nondegenerate pair in $V_{\ZZ}$. As usual, $p$ will denote a point in $\RR^5$.

\subsubsection{Upper and lower bounds, using Davenport's lemma}

\begin{lemma}
\label{davenportquartic}
For any $M$, $X$, and $p$ satisfying \eqref{basicineq}, we have
\[
	\exp[X,f(p)] \ll_p \#\{V_{\ZZ} \cap MB(p,X)\} \ll M^{12}\exp[X,f(p)].
\]
The first implicit constant is independent of $X$ and $M$ and the second implicit constant is independent of $p$, $X$, and $M$.
\end{lemma}
\begin{proof}
Let $S_k$ denote the set of pairs $(i,j)$ such that $q_k \geq p_i + p_j$ for $k = 1,2$. Then for every $(i,j) \in S_1$ and every lattice point $(A,B) \in MB(p,X)$, we have $\lv a_{ij} \rv \leq M$. Similarly, for every $(i,j) \in S_2$ and every lattice point $(A,B) \in MB(p,X)$, we have $\lv b_{ij} \rv \leq M$. Let $\pi \colon V_{\RR} \rightarrow \RR^{12-\#S_1 -\#S_2}$ denote the projection map away from $a_{ij}$ for every pair $(i,j) \in S_1$ and $b_{ij}$ for every pair $(i,j) \in S_2$. Then
\begin{align*}
\#\{\pi(B(p,X)) \cap \ZZ^{12-\#S_1-\#S_2}\} &\leq \#\{MB(p,X) \cap V_{\ZZ}\} \\
&\leq M^{\#S_1 + \#S_2} \#\{\pi(MB(p,X)) \cap \ZZ^{12-\#S_1-\#S_2}\}. 
\end{align*}
Davenport's lemma implies that
\begin{align*}
\#\{\pi(MB(p,X)) \cap \ZZ^{12-\#S_1 - \#S_2}\} = \Vol\big(\pi(MB(p,X))\big) + O\Bigg(\frac{\Vol(\pi(MB(p,X)))}{Y}\Bigg)
\end{align*}
where 
\[
	Y = M\min_{(i,j,k)}\{\exp[X,p_i + p_j - q_k] \; \mid \; (i,j) \notin S_k\}
\]
and the implicit constant is independent of $p$, $M$, and $X$. To finish, a computation shows that
\[
	\Vol(\pi(MB(p,X))) = M^{12-\#S_1 - \#S_2}\exp[X,f(p)].
\]
\end{proof}

\begin{lemma}
\label{upperboundquartic}
We have
\begin{align*}
\#\cF_4(p,\eps,X) \ll e^{36\eps}\exp[X,f(p)-g(p)]
\end{align*}
where the implicit constant does not depend on $X$, $p$, or $\eps$.
\end{lemma}
\begin{proof}
Let $S$ be the set of $(A,B)$ arising from Minkowski bases of pairs $(R,C) \in \cF_4(p,\eps,X)$. \cref{ublemma1} implies that there exists a constant $M$ such that $(A,B) \in Me^{3\eps}B(p,X)$ for all $(A,B) \in S$. By \cref{ublemma2}, there exists $D \geq 1$, independent of $X$, $p$, $M$, and $\eps$, such that
\[
\#\{(A',B') \in DMe^{3\eps}B(p,X) \mid (R,C) \simeq (R',C')\} \geq \exp[X, g(p)]
\]
for all elements of $S$. Therefore,
\[
\#\cF_4(p,\eps,X) \ll \frac{\#\{DMe^{3\eps}B(p,X) \cap V_{\ZZ}\}}{\exp[X, g(p)]}.
\]
Applying \cref{davenportquartic} to estimate the numerator completes the proof.
\end{proof}

\begin{lemma}
\label{lowerboundquartic}
There exists a constant $M$ such that for all $\eps > M$, we have:
\begin{enumerate}
\item if $p \in \poly_4$, then 
\begin{align*}
\#\cF_4(p,\eps,X) \gg_{p,\eps} \exp[X,f(p)-g(p)];
\end{align*}
\item if $p \in \poly_4(S_4)$, then 
\begin{align*}
\#\cF(p,\eps,S_4,X) \gg_{p,\eps} \exp[X,f(p)-g(p)];
\end{align*}
\item if $p \in \poly_4(D_4)$, then 
\begin{align*}
\#\cF(p,\eps,D_4,X) \gg_{p,\eps} \exp[X,1 - g(p) + (q_2-2p_1) + (q_2-p_1-p_2) + (q_2 - p_1 -p_3)].
\end{align*}
\end{enumerate}
\end{lemma}
\begin{proof}
We first prove $(1)$. Choose $p \in \poly_4$. \cref{lambdaboundlemmaquartic} shows that there exists a constant $M$ such that for all $\eps > M$, if $(A,B) \in B(p,X) \cap V_{\ZZ}$ and $X/2 \leq \Delta$, then $(A,B)$ is $(\eps,X)$-close to $p$. Let $S = \{(A,B) \in B(p,X) \cap V_{\ZZ} \mid X/2 \leq \Delta \leq X\}$.
Therefore
\begin{align*}
\#\cF_4(p,\eps,X) &\geq \frac{\#S}{\max_{(A,B) \in S}\#\{(A',B') \in B(p,X) \mid (R',C')\simeq (R,C)\}} && \text{} \\
&\asymp \frac{\#\{(A,B) \in B(p,X) \cap V_{\ZZ} \}}{\max_{(A,B)\in S}\#\{(A',B') \in B(p,X) \mid (R',C')\simeq (R,C)\}} && \text{by the definition of $B(p,X)$}.
\end{align*}
Applying \cref{davenportquartic} to estimate the numerator and applying \cref{reducedquarticcorol} to upper bound the denominator completes the proof of $(1)$.

The argument above also proves $(2)$, with the additional input that there exists $D \in \RR_{\geq 1}$ such that a positive proportion, as $X \rightarrow \infty$, of the lattice points in $DB(p,X) \cap V_{\ZZ}$ have Galois group $S_4$ and $\Delta \geq X/2$. In fact, we may replace the $\Delta \geq X/2$ condition with $\Delta \geq X/M$ for arbitrarily large $M$. Observe that
\[
	\lim_{M \rightarrow \infty} \lim_{X \rightarrow \infty} \frac{\#\{(A,B) \in DB(p,X) \cap V_{\ZZ} \; \mid \; \Delta > X/M\}}{\#(DB(p,X)\cap V_{\ZZ})} = 1.
\]
Thus it suffices to prove that that there exists $D \in \RR_{\geq 1}$ such that a positive proportion of the lattice points in $DB(p,X) \cap V_{\ZZ}$ have Galois group $S_4$ as $X \rightarrow \infty$. Consider the subset of lattice points of $DB(p,X)$ such that $\Delta \equiv 5 \Mod{25}$, the ring $R$ is irreducible modulo $7$, and the ring $C$ is irreducible modulo $11$; every element of this subset has Galois group $S_4$. If this set is nonempty for large enough $X$, then it forms a positive proportion of $DB(p,X)$. The nonemptiness follows from the fact that there exists a pair of integral ternary quadratic forms
\[
(A,B) = \Bigg(\begin{bmatrix}
    * & * & *  \\
    - & * & *  \\
    - & - & * 
  \end{bmatrix},
    \begin{bmatrix}
    0 & 0 & *  \\
    - & * & *  \\
    - & - & * 
  \end{bmatrix}\Bigg)
\]
such that $\Delta = 5 \Mod{25}$, the ring $R$ is irreducible modulo $7$, and the ring $C$ is irreducible modulo $11$. Choosing $D$ to be large enough such that $(A,B) \in DB(p,X)$ completes the proof of this case.

Similarly, to prove $(3)$ we will show that there exists $D \in \RR_{\geq 1}$ such that a positive proportion, as $X \rightarrow \infty$, of the lattice points $(A,B) \in B(p,X) \cap V_{\ZZ}$ satisfying $b_{11} = b_{12} = b_{13} = 0$ have Galois group $D_4$. First, it is easy to see that there exists a pair of integral ternary quadratic forms 
\[
	(A,B) = \Bigg(\begin{bmatrix}
    * & * & *  \\
    - & * & *  \\
    - & - & * 
  \end{bmatrix},
    \begin{bmatrix}
    0 & 0 & 0  \\
    - & * & *  \\
    - & - & * 
  \end{bmatrix}\Bigg)
\]
giving rise to a $D_4$-ring. Choose $D$ large enough so that $(A,B) \in DB(p,X)$. Let $S_k = \{(i,j) \colon q_k = p_i + p_j\}$ for $k = 1,2$. Let
\[
	Z = \{(A',B') \in V_{\ZZ} \mid \forall (i,j) \in S_1,\; \; a'_{ij} = a_{ij} \text{ and } \forall (i,j) \in S_2,\; \; b'_{ij} = b_{ij} \text{ and } b_{11} = b_{12} = b_{13} = 0\}.
\]
Observe that $Z \cap DB(p,X)$ contains a positive proportion of the lattice points in $DB(p,X)$. Further, observe that a general pair of integral ternary quadratic forms in $Z$ gives rise to a $D_4$-ring. Thus, by Hilbert's irreducibility theorem, the subset of $Z$ with Galois group $\neq D_4$ is a thin set; an application of quantitative Hilbert irreducibility (specifically Theorem~2.1 of \cite{cohen}) shows that the density, as $X \rightarrow \infty$, of lattice points in $Z$ with Galois group $\neq D_4$ is $0$.
\end{proof}

\subsubsection{Computing density functions}

\begin{proof}[Proof of \cref{fulldensitythmquartic}]
Clearly, the support of $d_{\eps,4}$ is contained in $\poly_4$. If $p \in \poly_4$, then combining \cref{upperboundquartic} and \cref{lowerboundquartic} shows that there exists a constant $M$ such that for all $\eps > M$, we have
\begin{align*}
\#\cF_4(p,\eps,X) \asymp_{\eps,p} \exp[X,1 - &(p_3 - p_2) - (p_3 - p_1) - (p_2 - p_1) - (q_2 - q_1) \\
&+\sum_{i \leq j,k}\max\{0, q_k-p_i-p_j\};
\end{align*}
\end{proof}

\begin{proof}[Proof of \cref{s4densitythm}]
The proof is analogous to that of \cref{fulldensitythmquartic}. Combine \cref{s4polytopethm}, \cref{upperboundquartic}, and \cref{lowerboundquartic}.
\end{proof}

\begin{proof}[Proof of \cref{d4densitythm}]
Again, the proof is analogous to that of \cref{fulldensitythmquartic}. Combine \cref{d4polytopethm}, \cref{upperboundquartic}, and \cref{lowerboundquartic}.
\end{proof}

\subsection{The generic points of $D_4$ and $S_4$}
In this section, we prove \cref{genericgaloisgroupscorollary}, first for $D_4$ and then for $S_4$. To prove \cref{genericgaloisgroupscorollary} for $D_4$, we will need the following lemma\footnote{This lemma can be strengthened significantly, but for our application this statement suffices.}.

\subsubsection{The generic point of $D_4$}
\begin{lemma}
\label{d4sievelemma}
Choose $p \in \poly_4(S_4)$ such that $p_i + p_j - q_k > 0$ for all $1 \leq i,j \leq 3$ and $1 \leq k \leq 2$. Then
\begin{equation*}
	\#\{(R,C) \mid \; \; \Delta \leq X \text{, $R$ is a $D_4$-ring, and $R$ is $(\eps,X)$-close to $p$}\} \ll_{\eps,\gamma} \exp[X,1 - \delta_{p} + \gamma]
\end{equation*}
where $\delta_{p} = \frac{1}{2}\max_{1 \leq i, j\leq 3, 1\leq k \leq 2}\{p_i + p_j - q_i\}$. Moreover, the implicit constant is independent of $p$.
\end{lemma}
\begin{proof}
By \cref{ublemma1}, it suffices to show that for any $M \geq 1$, we have
\begin{equation*}
\{(A,B) \in MB(p,X) \mid R \text{ is a $D_4$-ring and } \Delta \leq X\} \ll_{M,\gamma} X^{1-\delta_{p} + \gamma}.
\end{equation*}
Observe that if $(A,B)$ gives rise to a $D_4$-ring, then the binary cubic form $\dett(Ax + By)$ is reducible. We now use a version of quantitative Hilbert irreducibility (specifically an easy generalization of Theorem~2.1 of \cite{cohen}) and \cref{davenportquartic} to conclude.
\end{proof}

\begin{proof}[Proof of \cref{genericgaloisgroupscorollary} for $D_4$]
Let $p(D_4) = (0,1/4,1/4,0,1/2)$. We prove \eqref{correspondquartic1} for $D_4$ by demonstrating the following equivalent statement:
\begin{equation}
\label{D4allfar}
	\lim_{\eps \rightarrow \infty} \limsup_{X \rightarrow \infty}\frac{\#\{(R,C) \; \mid \; \Delta \leq X \text{ and $R$ is a $D_4$-ring that is not $(\eps,X)$-close to $p(D_4)$}\}}{\#\{(R,C) \; \mid \; \Delta \leq X \text{ and $R$ is a $D_4$-ring}\}} = 0.
\end{equation}
It is easy to bound the denominator; when ordered by absolute discriminant, $100\%$ of quadratic extensions of $\QQ(i)$ are quartic $D_4$-extensions, and there are $\asymp X$ quadratic extensions of $\QQ(i)$ of absolute discriminant $\leq X$. Hence  
\begin{equation}
\label{totald4count}
	\#\{(R,C) \; \mid \; \Delta \leq X \text{ and $R$ is a $D_4$-ring}\} \gg X.
\end{equation}

We now focus our attention on bounding the numerator of \eqref{D4allfar}. Let $\ell \leq k$ be positive integers whose values we fix their value later. Our proof strategy is very similar to that of \cref{genericcubiccorollary}; we count the number of $D_4$-rings of absolute discriminant $\leq X$ that are $(\eps/\ell,X)$-close to some \emph{other} point $p'$ on $\poly_4(D_4)$ with the property that $p'$ has distance at least $\eps/\log X$ from $p$. We then sum over an appropriate set of points $p'$; we evaluate this sum by splitting it into a \emph{main term} and a \emph{cusp}. We bound both the main term and the cusp to conclude.

To this end, there exists a constant $M$ such that for all $\eps > M$, we have:
\begin{align*}
& \#\{(R,C) \mid \Delta \leq X \text{, $R$ is a $D_4$-ring, and $R$ is not $(\eps,X)$-close to $p(D_4)$}\} \\
&\leq \MT + \Cusp
\end{align*}
where:
\begin{align*}
\MT & \coloneqq \#\{(R,C) \mid \; \; \Delta \leq X \text{, $R$ is a $D_4$-ring, and $R$ is $(\eps/k,X)$-close to a point in} \\
& \; \; \; \;\;\;\;\;\;\;\;\;\;\;\;\;\;\; \;\;\;\;\;\;\;\;\; \poly_4(D_4) \cap \{p \in \RR^5 \mid \frac{\eps}{\log X} \leq p_1 \leq 1/100\} \\
\Cusp & \coloneqq \#\{(R,C) \mid \; \; \Delta \leq X \text{, $R$ is a $D_4$-ring, and $R$ is $(\eps,X)$-close to a point in} \\
& \; \; \; \;\;\;\;\;\;\;\;\;\;\;\;\;\;\; \;\;\;\;\;\;\;\;\; \poly_4(D_4) \cap \{p \in \RR^5 \mid p_1 \geq 1/100\}
\end{align*} 
Therefore, to prove \eqref{D4allfar} it suffices to show that
\begin{equation}
\label{mtcuspzero}
	\lim_{\eps \rightarrow \infty}\limsup_{X \rightarrow \infty}\frac{\MT}{X} = \lim_{\eps \rightarrow \infty}\limsup_{X \rightarrow \infty}\frac{\Cusp}{X} = 0. 
\end{equation}
We begin by evaluating $\MT$. Choose values of $\ell$ and $k$ such that:
\begin{enumerate}
\item $\ell > 36$;
\item and for any point $(p_1,p_2,p_3,q_1,q_2) \in \poly_4$ and any pair $(R,C)$, if
\begin{align*}
\lv \log_{\Delta}\lambda_1(R) - p_1 \rv &\leq \frac{\eps}{k \log \Delta} \\
\lv \log_{\Delta}\lambda_2(R) - p_2 \rv &\leq \frac{\eps}{k \log \Delta} \\
\lv \log_{\Delta}\lambda_1(C) - q_1 \rv &\leq \frac{\eps}{k \log \Delta},
\end{align*}
then $(R,C)$ is $(\eps/\ell, \Delta)$-close to $(p_1,p_2,p_3,q_1,q_2)$.
\end{enumerate} 
We now define our appropriate set of points; they will form a grid whose points are spaced $\frac{\eps}{k\log X}$ distance apart and lie in the subset of $\poly_4(D_4)$ where $p_1 \leq 1/100$. For any integers $x$, $y$, and $z$, define the point
\[
	p_{xyz} \coloneqq (p_1(x),p_2(x,y),p_3(x,y),q_1(x,y,z),q_2(x,y,z))
\]
via the following formulae:
\begin{align*}
 p_1(x) &\coloneqq \frac{\eps}{\log X} + \frac{\eps x}{k \log X} \\
 p_2(x,y) &\coloneqq \bigg(\frac{1}{4} - \frac{1}{2}p_1(x)\bigg) - \frac{\eps y}{k \log X} \\
 p_3(x,y) &\coloneqq 1/2 - p_1(x) - p_2(x,y) \\
 q_1(x,y,z) &\coloneqq 1/2 - 2p_2(x,y) +  \frac{\eps z }{k\log X}\\
 q_2(x,y,z) &\coloneqq 1/2 - q_1(x,y,z)
\end{align*}
Let $M$ be the smallest positive integer such that 
\[
	p_1(M) \geq 1/100.
\]
For any nonnegative integer $x$, let $N_x$ be the smallest positive integer such that 
\[
	p_2(x,N_x) \leq \frac{1}{4} - p_1(x).
\]
For any nonnegative integers $x$ and $y$, let $L_{xy}$ be the smallest positive integer such that
\[
	q_1(x,y,L_{xy}) \geq 2p_1(x).
\]
The values $M$, $N_x$, and $L_{xy}$ indicate where the grid meets the boundaries of $\poly_4(D_4)$; see \cref{d4polytopethm} for an explicit description of the facets. Thus, we may write $\MT$ as the following triple sum\footnote{This is essentially Riemann integration on $\poly_4(D_4)$!}:
\begin{equation}
\label{triplesum}
	\MT \leq  \sum_{x = 0}^{M}\sum_{y = 0}^{N_{x}}\sum_{z = 0}^{L_{xy}}\#\{(R,C) \mid \Delta \leq X \text{, $R$ is a $D_4$-ring, and $R$ is $(\eps/\ell,X)$-close to $p_{xyz}$}\} 
\end{equation}
Without loss of generality, fix $\eps$ to be large enough so that $d_{\eps/k,4}$ does not depend on $\eps$; for conciseness, let $h = d_{\eps/k,4}$. We now bound the inner sum of \eqref{triplesum}; choose positive integers $x$ and $y$ such that $0 \leq  x \leq M$ and $0 \leq y \leq N_x$. Then:
\begin{align*}
&\sum_{z = 0}^{L_{xy}}\#\{(R,C) \mid \Delta \leq X \text{, $R$ is a $D_4$-ring, and $R$ is $(\eps/\ell,X)$-close to $p_{xyz}$}\} && \\
&\ll \sum_{z = 0}^{L_{xy}}e^{36\eps/\ell}\exp[X,h(p_{xyz})] && \text{by \cref{upperboundquartic}} \\
&= \sum_{z = 0}^{L_{xy}}e^{36\eps/\ell}\exp[X,h(p_{xy0})]X^{-\eps z/k\log X} && \text{by \cref{fulldensitythmquartic}} \\
&= e^{36\eps/\ell}\exp[X,h(p_{xy0})]\sum_{z = 0}^{L_{xy}}e^{-\eps z/k} \\
&\leq e^{36\eps/\ell}\exp[X,h(p_{xy0})]\frac{1}{1-e^{-\eps/k}}
\end{align*}
Moreover, the implicit constants are independent of $x$, $y$, and $\eps$. We now bound the second sum, also using \cref{upperboundquartic} and \cref{fulldensitythmquartic}. Fix a positive integer $0 \leq x \leq M$; there exists a positive real number $s$ such that: 
\begin{align*}
\sum_{y = 0}^{N_{x}} \exp[X,h(p_{xy0})] &= \exp[X,h(p_{x00})]\sum_{y = 0}^{N_{x}}X^{-\eps sy/k\log X} \\
&\leq \exp[X,h(p_{x00})]\frac{1}{1-e^{-\eps s/k}}.
\end{align*}
We bound the outer sum in exactly the same way; again \cref{upperboundquartic} implies that there exists a positive real number $t$ such that:
\begin{align*}
\sum_{x = 0}^{M} \exp[X,h(p_{x00})] &= \exp[X,h(p_{000})]\sum_{z = 1}^{M}X^{-\eps tx/k\log X} \\
&\leq \exp[X,h(p_{000})]\frac{1}{1-e^{-\eps t/k}}.
\end{align*}
Thus,
\begin{align*}
\frac{\MT}{X} &\ll e^{36\eps/\ell}X^{-1}\exp[X,h(p_{000})] \frac{1}{1-e^{-\eps t/k}}\frac{1}{1-e^{-\eps s/k}}\frac{1}{1-e^{\eps/k}}\\
&= e^{36\eps/\ell}X^{-\eps/\log X}\frac{1}{1-e^{-\eps t/k}}\frac{1}{1-e^{-\eps s/k}}\frac{1}{1-e^{-\eps/k}} \\
&= e^{36\eps/\ell - \eps}\frac{1}{1-e^{-\eps t/k}}\frac{1}{1-e^{-\eps s/k}}\frac{1}{1-e^{-\eps/k}}
\end{align*}
where the implicit constant is independent of $\eps$. Because $\ell > 36$, we have
\[
	\lim_{\eps \rightarrow \infty}e^{36\eps/\ell - \eps}\frac{1}{1-e^{-\eps t/k}}\frac{1}{1-e^{-\eps s/k}}\frac{1}{1-e^{\eps/k}} = 0.
\]
So,
\[
	\lim_{\eps \rightarrow \infty}\limsup_{X \rightarrow \infty}\frac{\MT}{X} = 0.
\]

We now bound $\Cusp$. For $p \in \poly_4$, \cref{upperboundquartic} implies that:
\begin{equation}
\label{d4cuspbound1}
	\#\{(R,C) \mid \; \; \Delta \leq X \text{, $R$ is a $D_4$-ring, and $R$ is $(\eps,X)$-close to $p$}\} \ll e^{36\eps}\exp[X,h(p)]
\end{equation}
where the implicit constant is independent of $\eps$ and $p$. Let $B$ be an open ball of radius $1/100$ around $(1/6,1/6,1/6,1/4,1/4)$. For $p \in B\cap \poly_4(S_4)$, \cref{d4sievelemma} shows that
\begin{equation}
\label{d4cuspbound2}
	\#\{(R,C) \mid \; \; \Delta \leq X \text{, $R$ is a $D_4$-ring, and $R$ is $(\eps,X)$-close to $p$}\} \ll_{\eps,\gamma} \exp[X,1 - \delta_{p} + \gamma]
\end{equation}
where $\delta_{p} = \frac{1}{2}\max_{1 \leq i, j\leq 3, 1\leq k \leq 2}\{p_i + p_j - q_i\}$ and the implicit constant is independent of $p$. Combining \eqref{d4cuspbound1} and \eqref{d4cuspbound2}, we conclude there exists a positive real number $\eta < 1$ such that for every $p \in \poly_4(D_4)$ with $p_1 \geq 1/100$, we have 
\[
	\#\{(R,C) \mid \; \; \Delta \leq X \text{, $R$ is a $D_4$-ring, and $R$ is $(\eps,X)$-close to $p$}\} \ll_{\eps} \exp[X,\eta],
\]
with an implicit constant independent of $p$. Summing over an appropriate grid in $\poly_4(D_4)$, we see that:
\[
	\Cusp \ll_{\eps} (\log X)^3\exp[X,\eta].
\]
Thus
\[	
	\lim_{\eps \rightarrow \infty}\lim_{X \rightarrow \infty}\frac{\Cusp}{X} = 0.
\]

We now show \eqref{correspondquartic2}, which translates to showing that
\[
	\liminf_{X \rightarrow \infty}\frac{\#\{(R,C) \; \mid \; \Delta \leq X \text{ and $(R,C)$ is a $D_4$-ring and is $(\eps,X)$-close to $p(D_4)$}\}}{\#\{(R,C) \; \mid \; \Delta \leq X \text{ and $(R,C)$ is $(\eps,X)$-close to $p(D_4)$}\}} > 0.
\]
Combining \cref{upperboundquartic} and \cref{lowerboundquartic} shows that the denominator is $\asymp_{\eps} X$.

We now show that the numerator is $\gg X$. When ordered by absolute discriminant, $100\%$ of quadratic extensions of $\QQ(i)$ are quartic $D_4$-extensions. Let $R$ be a maximal order in a quartic $D_4$-extension containing $\QQ(i)$. Then $i \in R$, so $\lambda_1(R) = 1$; moreover, let $v \in R\setminus \ZZ[i]$ be such that $\lv v \rv =  \lambda_2(R)$; then $\lv iv \rv = \lv v \rv$, so $\lambda_3(R) = \lambda_2(R)$. Moreover, $R$ has a unique cubic resolvent ring $C$; let $\varphi \colon R/\ZZ \rightarrow C/\ZZ$ be the cubic resolvent map. Then because $R$ is irreducible, $\varphi(i)$ must be a nonzero element of $C/\ZZ$; moreover, $\lambda_1(C) \leq \lv \varphi(i) \rv \ll 1$. Therefore there exists a positive real number $M$ such that for all $\eps > M$, every such ring $R$ is $(\eps,\Delta)$-close to $p(D_4)$. Observing that there are $\asymp X$ many quadratic extensions of $\QQ(i)$ of absolute discriminant $\leq X$ concludes the proof.
\end{proof}

\subsubsection{The generic point of $S_4$}
To prove \cref{genericgaloisgroupscorollary} when $G = S_4$, we appeal to upcoming work of H and Vincent. Given a nondegenerate rank $n$ ring $R$, we consider the rank $(n-1)$ sublattice $\cL_R \coloneqq \{x \in \ZZ + nR \; \mid \; \Tr(x) = 0\}$, up to isometry. It is naturally seen as an element of the double coset space
\[
	\mathcal{S}_{n-1} \coloneqq \GL_{n-1}(\ZZ)\backslash \GL_{n-1}(\RR)/\GO_{n-1}(\RR).
\]
The space $\mathcal{S}_{n-1}$ has a natural measure $\mu_{n-1}$, obtained from the Haar measure on $\GL_{n-1}(\RR)$ and $\GO_{n-1}(\RR)$. To a nondegenerate pair $(R,C)$ we associate an element $(\cL_R,\cL_C) \in \cS_{3} \times \cS_2$; we call this the \emph{shape} of the pair $(R,C)$.

\begin{theorem}[H, Vincent]
\label{equidistribofquartic}
When ordered by absolute discriminant, the shapes of pairs $(R,C)$ where $R$ is an $S_4$-ring are equidistributed with respect to $\mu_3\times \mu_2$.
\end{theorem}

\begin{proof}[Proof of \cref{genericgaloisgroupscorollary} when $G = S_4$]
Let $p(S_4) = (1/6,1/6,1/6,1/4,1/4)$. We will first prove \eqref{correspondquartic1}. 
For an integer $n\geq 2$, an element $\cL \in \cS_{n}$, and an integer $0 \leq i < n$, define $\lambda_i(\cL)$ to be the $i$-th successive minimum of a lattice representing $\cL$ with determinant $1$. For a positive real number $\eps$ and an integer $n \geq 2$, define
\[
	\cS_{n,\eps} \coloneqq \{\cL \in \cS_n \mid \max_{1 \leq i < n}\{\log\lambda_i(\cL)\} \leq \eps\}.
\]
Then 
\begin{align*}
&\lim_{\eps \rightarrow \infty} \liminf_{X \rightarrow \infty}\frac{\#\{(R,C) \; \mid \; \Delta \leq X \text{ and $R$ is a $S_4$-ring that is $(\eps,X)$-close to $p(S_4)$}\}}{\#\{(R,C) \; \mid \; \Delta \leq X \text{ and $R$ is a $S_4$-ring}\}} \\
&\geq \lim_{\eps \rightarrow \infty} \liminf_{X \rightarrow \infty}\frac{\#\{(R,C) \; \mid \Delta \leq X \text{ and $R$ is a $S_4$-ring and } (\cL_R,\cL_C) \in \cS_{3,\eps} \times \cS_{2,\eps}\}}{\#\{(R,C) \; \mid \; \Delta \leq X \text{ and $R$ is a $S_4$-ring}\}} \\
&= \lim_{\eps \rightarrow \infty}\frac{(\mu_3\times\mu_2)(\cS_{3,\eps} \times \cS_{2,\eps})}{(\mu_3\times\mu_2)(\cS_{3} \times \cS_2)} \\
&= 1
\end{align*}
We prove \eqref{correspondquartic2} with the following observation: \cref{upperboundquartic} and \cref{lowerboundquartic} combine to prove that there exists a positive real number $M$ such that for all $\eps > M$, a positive proportion of rings in $\cF(p(S_4),\eps,X)$ as $X \rightarrow \infty$ are $S_4$-rings.
\end{proof}

\subsection{Special families of quartic rings}
\subsubsection{Quartic rings with a basis of the form $1,x,y,xy$}

The following index formula of Bhargava \cite{BhargavaHCL3} is essential to our study of quartic rings equipped with a basis of the form $1,x,y,xy$. For a rank $n$ ring $R$ and $n$ elements $v_1,\dots,v_n \in R$, define $\Ind(v_1,\dots,v_n)$ to be the index of $\ZZ\la v_1,\dots,v_n\ra$ in $R$ if $ v_1,\dots,v_n$ are linearly independent, and $0$ otherwise.

\begin{proposition}[Bhargava \cite{BhargavaHCL3}, Lemma~9]
\label{1xyxyidentity}
If $R$ is a quartic ring and $C$ is a cubic resolvent ring of $R$, then for any $x,y \in R$, 
\[
	\Ind(1,x,y,xy) = \Ind(1,\varphi(x),\varphi(y)).
\]
\end{proposition}
The above proposition implies that any nondegenerate quartic ring with a basis $1,x,y,xy$ has a \emph{unique} cubic resolvent ring, namely $\ZZ \la 1,\varphi(x),\varphi(y) \ra$. Suppose $R$ is a nondegenerate quartic ring with a basis $1,x,y,xy$. Let $C$ be the unique cubic resolvent ring of $R$ and equip $C/\ZZ$ with the basis $\varphi(x),\varphi(y)$. Choose a basis $x,y,z$ for $R/\ZZ$. Then, the cubic resolvent map $\varphi \colon R/\ZZ \rightarrow C/\ZZ$, expressed in these bases, has the form:
\begin{equation}
\label{matrixform1xyxy}
  (A,B) = \begin{bmatrix}
    1 & \ast & \ast \\
    - &  0 & \ast  \\
    - & - & \ast
  \end{bmatrix},
    \begin{bmatrix}
    0 & \ast & \ast \\
    - &  1 & \ast  \\
    - & - & \ast
  \end{bmatrix}
\end{equation}
Conversely, a nondegenerate element of $V_{\ZZ}$ of the form \eqref{matrixform1xyxy} gives rise to the following data: a pair $(R,C)$, a basis $x,y,z$ of $R/\ZZ$, and a basis $\varphi(x),\varphi(y)$ of $C/\ZZ$. \cref{1xyxyidentity} implies that
\[
	\Ind(1,x,y,xy) = \Ind(1,\varphi(x),\varphi(y)) = 1,
\]
so $x,y,xy$ is also a basis of $R/\ZZ$. There are unique lifts $X,Y \in R$ of $x,y \in R/\ZZ$ and a unique value $\varepsilon = \pm 1$ such that $z \equiv \varepsilon XY \Mod{\ZZ}$. This discussion will prove \cref{main1xyxyparam}, which we state after making the following definitions.

Consider quadruples $(R,x,y,z)$ where $x,y,z \in R/\ZZ$ is a basis of $R/\ZZ$ and $1,X,Y,XY$ is a basis of $R$ for any lift $X,Y \in R$ of $x,y$. We say such a quadruple is \emph{nondegenerate} if $R$ is nondegenerate. Moreover, we say two such quadruples $(R,x,y,z)$ and $(R',x',y',z')$ are isomorphic if there exists a ring isomorphism $\alpha \colon R \rightarrow R'$ such that:
\begin{align*}
\alpha(x) &\equiv x' \Mod{\ZZ} \\
\alpha(y) &\equiv y' \Mod{\ZZ} \\
\alpha(z) &\equiv z' \Mod{\ZZ}
\end{align*} 

Similarly, now consider quadruples $(R,X,Y,\varepsilon)$ where $X,Y \in R$, $\varepsilon = \pm 1$, and $1, X,Y,\varepsilon XY$ is a basis of $R$. Again, we say such a quadruple is \emph{nondegenerate} if $R$ is nondegenerate. We say two such quadruples $(R,X,Y,\varepsilon)$ and $(R',X',Y',\varepsilon')$ are isomorphic if $\varepsilon = \varepsilon'$ and there exists a ring isomorphism $\alpha \colon R \rightarrow R'$ such that:
\begin{align*}
\alpha(X) &= X' \\
\alpha(Y) &= Y'
\end{align*} 
\begin{theorem}
\label{main1xyxyparam}
There is a canonical discriminant-preserving bijection between the following three sets:
\[
				\left \{ \begin{array}{c}
				\text{isomorphism classes of nondegenerate quadruples $(R,x,y,z)$ such that } \\
				\text{$x,y,z \in R/\ZZ$  and $x,y,z$ is a basis of $R/\ZZ$ and $1, X,Y,XY$ } \\
				\text{is a basis of $R$ for any lifts $X,Y \in R$ of $x,y \in R/\ZZ$}
			\end{array}	 \right \}
			\]	
\[
				\left \{ \begin{array}{c}
				\text{isomorphism classes of nondegenerate quadruples $(R,X,Y, \varepsilon)$} \\
				\text{such that $X,Y \in R$, $\varepsilon = \pm 1$, and $1, X,Y,\varepsilon XY$ is a basis of $R$}
			\end{array}	 \right \}
			\]	
\[
			\left \{ \begin{array}{c}
				\text{nondegenerate pairs of integral ternary } \\
				\text{quadratic forms of the form \eqref{matrixform1xyxy}}
				\end{array}
			\right \}
			\]	

\end{theorem}

Let $H$ be the subgroup of $\GL_2(\ZZ) \times \GL_3(\ZZ)$ acting on pairs of matrices of the form \eqref{matrixform1xyxy}; $H$ is generated by the following matrices:
\begin{equation}
\label{groupformmatrixform1xyxy}
\begin{split}
	\begin{bmatrix}
    1 & 0  \\
    0 &  1
  \end{bmatrix}&,
    \begin{bmatrix}
    \pm 1 & 0 & 0 \\
    0 &  \pm 1 & 0  \\
    \ast & \ast & \pm 1
  \end{bmatrix}  \\
  \begin{bmatrix}
    0 & 1  \\
    1 &  0
  \end{bmatrix}&,
    \begin{bmatrix}
    0 & 1 & 0 \\
    1 &  0 & 0  \\
    0 & 0 & 1
  \end{bmatrix}
\end{split}
\end{equation}

Suppose $R$ is a nondegenerate quartic ring. Consider pairs $(x,y) \in R$ such that $1,x,y,xy$ is a basis of $R$. Consider the equivalence relation on such pairs generated by $(x,y) \sim (y,x)$ and $(x,y) \sim (\pm x + a,y)$ for any $a \in \ZZ$. We define a \emph{digenization of $R$} to be such an equivalence class $(x,y)$. A \emph{digenized quartic ring} is defined to be a nondegenerate quartic ring equipped with a digenization\footnote{The notion of digenized quartic rings is a natural variant of the concept of a \emph{monogenized} quartic ring.}. We say two digenized quartic rings $(R,x,y)$ and $(R',x',y')$ are isomorphic if there exists a ring isomorphism $\alpha \colon R \rightarrow R'$ such that $(\alpha(x),\alpha(y))\sim (x',y')$. Then we obtain:
\begin{theorem}
There is a canonical discriminant-preserving bijection between:
			\[
				\left \{ \begin{array}{c}
				\text{isomorphism classes of} \\
				\text{digenized quartic rings} \\
			\end{array}	 \right \}
			\longleftrightarrow
			\left \{ \begin{array}{c}
				\text{nondegenerate orbit classes of} \\
				\text{pairs of integral ternary quadratic forms} \\
				\text{of the form \eqref{matrixform1xyxy}} \\
				\text{under the action of $H$} 
				\end{array}
			\right \}
			\]	
\end{theorem}

A theorem of Bhargava \cite{monogenizations} proves that a nondegenerate quartic ring has at most $2760$ many monogenizations; it is natural to ask if the number of digenizations of a quartic ring is uniformly bounded! However, it is clear that a monogenic quartic ring has infinitely digenizations. It is interesting to ask if there is a positive integer $N$ such that for every nondegenerate quartic ring $R$:
\begin{enumerate}
\item if $R$ has infinitely many digenizations, then it is monogenic;
\item else, $R$ has at most $N$ digenizations.
\end{enumerate}

We now consider matrices of the form \eqref{matrixform1xyxyresolventintro}, which we repeat here:
\begin{equation*}
  (A,B) = \begin{bmatrix}
    1 & 0 & \ast \\
    - &  0 & \ast  \\
    - & - & \ast
  \end{bmatrix},
  \begin{bmatrix}
    0 & 0 & \ast \\
    - &  1 & \ast  \\
    - & - & \ast
  \end{bmatrix}.
\end{equation*}

We obtain the following theorem as a corollary of \cref{main1xyxyparam}.
\onexyxyparam*

Recall that $\cF^{x,y,xy}$ is defined to be the set of isomorphism classes of triples $(R,x,y)$ where $R$ is a nondegenerate quartic ring and $x,y \in R$ are such that $1,x,y,xy$ is a basis of $R$ and $\varphi(x+y) = \varphi(x) + \varphi(y)$. Let $W_{\ZZ}$ denote the set of matrices of the form \eqref{matrixform1xyxyresolventintro} and let $\psi$ be the map from triples $(R,x,y) \in \cF^{x,y,xy}$ to pairs of integral ternary quadratic forms referred to in \cref{1xyxyparam} given by setting $\varepsilon = 1$. We now prove \cref{1xyxythm}. Let $p = (1/8,1/8,1/4,1/4,1/4)$. 

\begin{proof}[Proof of \cref{1xyxythm}]
Let $T = (R,x,y) \in \cF^{x,y,xy}$ be any such triple of absolute discriminant $\Delta$ and let $C$ be the unique cubic resolvent ring. A computation shows that $H(T) \leq X^{1/4}$ if and only if $\psi(T) \in B(p,X)$; thus for any real numbers $M \geq 1$ and $X > 1$, an easy generalization of \cref{lambdaboundlemmaquartic} implies that if $\Delta \geq X/M$ and $H(T)^4 \leq X$, then
\begin{align*}
\lambda_i(R) &\asymp_M \Delta^{p_i} \\
\lambda_i(C) &\asymp_M \Delta^{q_i}.
\end{align*}
So, for each $M \geq 1$, there exists a positive real number $\eps_M$ such that if  $\Delta \geq X/M$ and $H(T)^4 \leq X$, we have that $\psi(T)$ is $(\eps_M, X)$-close to $p$. Moreover, we may suppose $\lim_{M \rightarrow \infty}\eps_M = \infty$. We now prove \eqref{correspondquartic1}. 
We have:
\begin{align*}
&\lim_{\eps \rightarrow \infty} \liminf_{X \rightarrow \infty}\frac{\#\{T \in \cF^{x,y,xy} \; \mid \; H(T) \leq X \text{ and $\psi(T)$ is $(\eps,X)$-close to $p$}\}}{\#\{T \in \cF^{x,y,xy}\; \mid \; H(T) \leq X \}} \\
&=\lim_{\eps \rightarrow \infty} \liminf_{X \rightarrow \infty}\frac{\#\{T \in \cF^{x,y,xy} \; \mid \; H(T) \leq X^{1/4} \text{ and $\psi(T)$ is $(\eps,X^{1/4})$-close to $p$}\}}{\#\{T \in \cF^{x,y,xy}\; \mid \; H(T) \leq X^{1/4} \}} \\
&=\lim_{\eps \rightarrow \infty} \liminf_{X \rightarrow \infty}\frac{\#\{T \in \cF^{x,y,xy} \; \mid \; H(T)^4 \leq X \text{ and $\psi(T)$ is $(4\eps,X)$-close to $p$}\}}{\#\{T \in \cF^{x,y,xy}\; \mid \; H(T)^4 \leq X \}} \\
&=\lim_{\eps \rightarrow \infty} \liminf_{X \rightarrow \infty}\frac{\#\{T \in \cF^{x,y,xy} \; \mid \; H(T)^4 \leq X \text{ and $\psi(T)$ is $(\eps,X)$-close to $p$}\}}{\#\{T \in \cF^{x,y,xy}\; \mid \; H(T)^4 \leq X \}} \\
&\geq \lim_{M \rightarrow \infty} \liminf_{X \rightarrow \infty}\frac{\#\{T  \in \cF^{x,y,xy} \; \mid \; H(T)^4 \leq X \text{ and $\Delta \geq X/M$}\}}{\#\{T \in \cF^{x,y,xy}\; \mid \; H(T)^4 \leq X \}}.
\end{align*}
Davenport's lemma gives:
\begin{align*}
{\#\{T \in \cF^{x,y,xy} \; \mid \; H(T)^4 \leq X \}} &= \frac{1}{2}\#\{(A,B) \in W_{\ZZ} \cap B(p,X) \mid \Disc(A,B) \neq 0\} \\
&=\frac{X}{2}\Vol\{(A,B) \in W_{\RR} \cap B(p,1)\} + O(X^{7/8}).
\end{align*}
We now compute the numerator:
Let 
\[
	g_X = \Bigg(\begin{bmatrix}
    X^{1/4} & 0  \\
    0 &  X^{1/4} 
  \end{bmatrix},
  \begin{bmatrix}
    X^{-1/8} & 0 & 0 \\
    0 &  X^{-1/8} & 0  \\
    0 & 0 & X^{-1/4}
  \end{bmatrix}\Bigg).
\]
The element $g_X$ maps $B(p,X)$ to $B(p,1)$. By Davenport's lemma, we have
\begin{align*}
&\#\{T \in \cF^{x,y,xy} \; \mid \; H(T)^4 \leq X \text{ and } \Delta \geq X/M\} \\
&= \frac{1}{2}\#\{(A,B) \in W_{\ZZ} \cap B(p,X) \; \mid \;  \lv \disc(A,B) \rv \geq X/M\} \\
&= \frac{1}{2}\#\{(A,B) \in W_{\ZZ} \cap B(p,X) \; \mid \;  \lv \disc(g(A,B)) \rv \geq 1/M\} \\
&= \frac{1}{2} X\Vol\{(A,B) \in W_{\RR} \cap B(p,1) \; \mid \; \lv \disc(A,B) \rv \geq 1/M\} + O(X^{7/8}).
\end{align*}
Note that $W_{\RR}$ lies on an affine linear subspace, so the ``volume'' above is taken relatively. Now observe that
\[
	\lim_{M\rightarrow \infty}\Vol\{(A,B) \in W_{\RR} \cap B(p,1) \; \mid \; \lv \disc(A,B) \rv \geq 1/M\} = 1
\]
so 
\[
\lim_{\eps \rightarrow \infty} \liminf_{X \rightarrow \infty}\frac{\#\{T \in \cF^{x,y,xy} \; \mid \; H(T) \leq X \text{ and $\psi(T)$ is $(\eps,X)$-close to $p$}\}}{\#\{T \in \cF^{x,y,xy} \; \mid \; H(T) \leq X \}} = 1.
\]

We obtain Condition \eqref{cond2cubicspecialfamily} via the proof of \cref{upperboundquartic}, using the additional observation that a positive proportion (depending on $\eps$) of lattice points in $B(p,X)$ lie on the affine linear subspace  defined by $a_{11} = b_{22} = 1$ and $a_{12} = b_{12} = a_{22} = b_{11} = 0$.
\end{proof}

\subsubsection{Quartic rings with a basis of the form $1,x,y,x^2$}

As in the case of quartic rings with a basis of the form $1,x,y,xy$, our proofs crucially rely on an \emph{index formula}.

\begin{proposition}
\label{indexprop2}
If $R$ is a quartic ring and $C$ is a cubic resolvent ring of $R$, then for any $x,y \in R$, 
\[
	\Ind(1,x,y,x^2) = \Ind(1,\varphi(x),\varphi(x+y)-\varphi(x)-\varphi(y)).
\]
\end{proposition}
\begin{proof}
The proof proceeds exactly as that of \cref{1xyxyidentity}. Namely, the assertion of the lemma is equivalent to the following identity:
\[
	\begin{array}{|c c c c|}
1 & 1 & 1 & 1 \\
x & x' & x'' & x''' \\
y & y' & y'' & y''' \\
x^2 & (x')^2 & (x'')^2 & (x''')^2
\end{array}
= 
	\begin{array}{|c c c|}
1 & 1 & 1 \\
xx' + x''x''' & xx'' + x'x''' & xx''' + x'x'' \\
f & g & h 
\end{array}
\]
where
\begin{align*}
f &= (x + y)(x' + y') + (x'' + y'')(x''' + y''') - (xx' + x''x''') - (yy' + y''y''') \\
g &= (x + y)(x'' + y'') + (x' + y')(x''' + y''') - (xx'' + x'x''') - (yy'' + y'y''') \\
h &= (x + y)(x''' + y''') + (x' + y')(x'' + y'') - (xx''' + x'x'') - (yy''' + y'y'').
\end{align*}
The identity may be verified by direct calculation.
\end{proof}

From the above proposition, we see that any nondegenerate quartic ring with a basis of the form $1,x,y,x^2$ has a \emph{unique} cubic resolvent ring, namely the ring spanned by the vectors $1,\varphi(x),\varphi(x+y) - \varphi(x) - \varphi(y)$. Let $R$ be a nondegenerate quartic ring equipped with the basis $1,x,y,x^2$ and let $C$ be the unique cubic resolvent ring; equip $C/\ZZ$ with the basis $\varphi(x),\varphi(x+y) - \varphi(x)-\varphi(y)$. Choose a basis $x,y,z$ for $R/\ZZ$. Then, the cubic resolvent map $\varphi \colon R/\ZZ \rightarrow C/\ZZ$, expressed in these bases, has the form:
\begin{equation}
\label{matrixform1xyx2}
  (A,B) = \begin{bmatrix}
    1 & 0 & \ast \\
    - &  \ast & \ast  \\
    - & - & \ast
  \end{bmatrix},
    \begin{bmatrix}
    0 & 1/2 & \ast \\
    - &  \ast & \ast  \\
    - & - & \ast
  \end{bmatrix}.
\end{equation}

Conversely, a nondegenerate element of $V_{\ZZ}$ of the form \eqref{matrixform1xyx2} gives rise to the following data: a pair $(R,C)$, a basis $x,y,z$ of $R/\ZZ$, and a basis $\varphi(x),\varphi(x+y)-\varphi(x)-\varphi(y)$ of $C/\ZZ$. \cref{indexprop2} implies that
\[
	\Ind(1,x,y,x^2) = \Ind(1,\varphi(x),\varphi(x+y)-\varphi(x)-\varphi(y)) = 1
\]
so $1,X,Y,X^2$ is a basis of $R$ for any lifts $X,Y \in R$ of $x,y \in R/\ZZ$.  This discussion proves the following theorem.
\begin{theorem}
\label{firstparam}
There is a canonical discriminant-preserving bijection between:
			\[
				\left \{ \begin{array}{c}
				\text{isomorphism classes of} \\
				\text{nondegenerate quadruples $(R,x,y,z)$} \\
				\text{where $x,y,z \in R/\ZZ$ are such that } \\
				\text{$1,X,Y,X^2$ is a basis of $R$} \\
				\text{for any lifts $X,Y \in R$ of $x,y \in R/\ZZ$}
			\end{array}	 \right \}
			\longleftrightarrow
			\left \{ \begin{array}{c}
				\text{nondegenerate pairs of} \\
				\text{integral ternary quadratic forms} \\
				\text{of the form \eqref{matrixform1xyx2}}
				\end{array}
			\right \}
			\]	
\end{theorem}
Let $H$ be the largest subgroup of $\GL_2(\ZZ) \times \GL_3(\ZZ)$ acting on matrices of the form \eqref{matrixform1xyx2}; $H$ is generated by of elements of the form:
\begin{equation}
\label{groupformmatrixform1xyxy}
\begin{split}
	\begin{bmatrix}
    1 & -2u  \\
    0 &  1
  \end{bmatrix}&,
    \begin{bmatrix}
    1 & 0 & 0 \\
    u &  1 & 0  \\
    \ast & \ast & 1
  \end{bmatrix} \\
  \begin{bmatrix}
    1 & 0  \\
    0 &  1
  \end{bmatrix}&,
   \begin{bmatrix}
    1 & 0 & 0 \\
    0 &  1 & 0  \\
    0 & 0 & \pm 1
  \end{bmatrix} \\
   \begin{bmatrix}
    1 & 0  \\
    0 &  -1
  \end{bmatrix}&,
   \begin{bmatrix}
    -1 & 0 & 0 \\
    0 &  -1 & 0  \\
    0 & 0 & 1
  \end{bmatrix} 
\end{split}
\end{equation}

Suppose $R$ is a nondegenerate quartic ring. Consider pairs $(x,y) \in R$ such that $1,x,y,x^2$ is a basis of $R$. Consider the equivalence relation on all pairs generated by $(x,y) \sim (\pm x + a,  b + cx + \pm y)$ for integers $a,b,c \in \ZZ$. Now consider triples $(R,x,y)$ where $x\in R/\ZZ$, $y \in R/\ZZ\la 1,x\ra$ are such that $R = \ZZ\la 1,x,y,x^2\ra$. We say the triple is \emph{nondegenerate} if $R$ is nondegenerate. We say two such triples $(R,x,y)$ and $(R',x',y)$ are isomorphic if there exists a ring isomorphism $\alpha \colon R \rightarrow R'$ such that $\alpha(x) \equiv x' \Mod{\ZZ}$ and $\alpha(y) \equiv y' \Mod{\ZZ\la 1,x \ra}$. We obtain:
\begin{theorem}
There is a canonical discriminant preserving bijection between:
			\[
				\left \{ \begin{array}{c}
				\text{isomorphism classes of} \\
				\text{nondegenerate triples $(R,x,y)$} \\
				\text{where $x\in R/\ZZ$, $y \in R/\ZZ\la 1,x\ra$ are such that } \\
				\text{$R = \ZZ\la 1,X,Y,X^2\ra$ for any lifts} \\
				\text{$X,Y \in R$ of $x,y$}
			\end{array}	 \right \}
			\longleftrightarrow
			\left \{ \begin{array}{c}
				\text{nondegenerate orbit classes of} \\
				\text{pairs of integral ternary quadratic forms} \\
				\text{of the form \eqref{matrixform1xyx2}} \\
				\text{under the action of $H$} 
				\end{array}
			\right \}
			\]	
\end{theorem}

We now restrict our attention to matrices of the form \eqref{matrixform1xyx2intro}, which we repeat here:
\begin{equation*}
  (A,B) = \begin{bmatrix}
    1 & 0 & \ast \\
    - &  \ast & \ast  \\
    - & - & \ast
  \end{bmatrix},
    \begin{bmatrix}
    0 & 1/2 & 0 \\
    - &  \ast & \ast  \\
    - & - & \ast
  \end{bmatrix}.
\end{equation*}

We consider quintuples $(R,x,y,\varepsilon,\nu)$ where $x \in R$, $y \in R/\ZZ$, $\varepsilon = \pm 1$, $\nu \in \{0,1\}$, and $x,y,x^2$ is a basis of $R/\ZZ$. We such a quintuple is \emph{nondegenerate} if $R$ is nondegenerate. We say two such quintuples $(R,x,y,\varepsilon,\nu)$ and $(R',x',y',\varepsilon',\nu')$ are isomorphic if $\varepsilon = \varepsilon'$, $\nu = \nu'$, and there exists a ring isomorphism $\alpha \colon R \rightarrow R'$ such that $\alpha(x)=x'$ and $\alpha(y)  = y' \Mod{\ZZ}$.

\onexyxtwoparam*
\begin{proof}
Given such a nondegenerate pair of integral ternary quadratic forms, \cref{firstparam} gives a nondegenerate quartic ring with a basis $x,y,z$ of $R/\ZZ$ with the property that $1,X,Y,X^2$ is a basis of $R$ for any lifts $X,Y \in R$ of $x,y \in R/\ZZ$. There is a unique lift $X \in R$ of $x$ and a unique value $\varepsilon = \pm 1$ and $\nu \in \{0,1\}$ such that $z \equiv \varepsilon X^2 + \nu X \Mod{\ZZ\la 1,y\ra}$. We take the quintuple $(R,X,y,\varepsilon,\nu)$. Conversely, given such a quintuple $(R,x,y,\epsilon,\nu)$, there exists a unique integer $a$ such that 
\[
	\varphi(x + (\varepsilon x^2 + \nu x + ay)) - \varphi(\varepsilon x^2 + \nu x + ay) \equiv 0 \Mod{\ZZ\la 1,\varphi(x)\ra}.
\]
Equip $R/\ZZ$ with the basis $x,y,\varepsilon x^2 + \nu x + ay$. Let $C$ be the unique cubic resolvent ring of $R$ and equip $C/\ZZ$ with the basis $1,\varphi(x),\varphi(x+y)-\varphi(x)-\varphi(y)$; taking the cubic resolvent map with respect to these bases gives an element of \eqref{matrixform1xyx2intro}. A computation shows that these two maps are inverse to each other.
\end{proof}

Recall that $\cF^{x,y,x^2}$ is defined to be the set of isomorphism classes of triples $(R,x,y)$ where $R$ is a nondegenerate quartic ring and $x \in R$ and $y \in R/\ZZ$ is such that $x,y,x^2$ forms a basis of $R/\ZZ$. Let $W_{\ZZ}$ denote the set of matrices of the form \eqref{matrixform1xyx2intro} and $\psi$ be the map from triples $(R,x,y) \in \cF^{x,y,x^2}$ to pairs of integral ternary quadratic forms referred to in \cref{1xyx2param} given by setting $\varepsilon = 1$ and $\nu = 0$. We now prove \cref{1xyx2thm}. Let $p = (1/10,1/5,1/5,1/5,3/10)$. 

\begin{proof}[Proof of \cref{1xyx2thm}]
The proof is almost identical to that of \cref{1xyxythm}. Let $T = (R,x,y) \in \cF^{x,y,x^2}$ be any triple of absolute discriminant $\Delta$ and let $C$ be the unique cubic resolvent ring. A computation shows that $H(T) \leq X^{1/5}$ if and only if $\psi(T) \in B(p,X)$; thus for any positive real numbers $M$ and $X$, \cref{lambdaboundlemmaquartic} implies that if $\Delta \geq X/M$, then
\begin{align*}
\lambda_i(R) &\asymp_M \Delta^{p_i} \\
\lambda_i(C) &\asymp_M \Delta^{q_i}.
\end{align*}
So, for each $M$ there exists a constant $\eps_M$ such that if $\Delta \geq X/M$ and $H(T)^5 \leq X$, we have that $T$ is $(\eps_M,X)$-close to $p$. Moreover, we may suppose $\lim_{M \rightarrow \infty}\eps_M = \infty$. We first prove \eqref{correspondquartic1}. 
\begin{align*}
&\lim_{\eps \rightarrow \infty} \liminf_{X \rightarrow \infty}\frac{\#\{T \in \cF^{x,y,x^2} \; \mid \; H(T) \leq X \text{ and $\psi(f)$ is $(\eps,X)$-close to $p$}\}}{\#\{T \in \cF^{x,y,x^2} \; \mid \; H(T) \leq X \}} \\
&\geq \lim_{M \rightarrow \infty} \liminf_{X \rightarrow \infty}\frac{\#\{T \in \cF^{x,y,x^2} \; \mid \; H(T)^5 \leq X \text{ and $\Delta \geq X/M$}\}}{\#\{T \in \cF^{x,y,x^2}\; \mid \; H(T)^5 \leq X \}}.
\end{align*}
Davenport's lemma gives:
\begin{align*}
{\#\{T \in \cF^{x,y,s^2} \; \mid \; H(T)^5 \leq X \}} &\asymp \frac{1}{4}\#\{(A,B) \in W_{\ZZ} \cap B(p,X) \mid \Disc(A,B) \neq 0\} \\
&=\frac{X}{4}\Vol\{(A,B) \in W_{\RR} \cap B(p,1)\} + O(X^{9/10}).
\end{align*}
We now compute the numerator. Let 
\[
	g_X = \Bigg(\begin{bmatrix}
    X^{1/5} & 0  \\
    0 &  X^{3/10} 
  \end{bmatrix},
  \begin{bmatrix}
    X^{-1/10} & 0 & 0 \\
    0 &  X^{-1/5} & 0  \\
    0 & 0 & X^{-1/5}
  \end{bmatrix}\Bigg).
\]

By Davenport's lemma, we have
\begin{align*}
&\#\{T \in \cF^{x,y,x^2}\; \mid \; H(T)^5 \leq X \text{ and } \Delta \geq X/M\} \\
&\asymp \frac{1}{4}\#\{(A,B) \in W_{\ZZ} \cap B(p,X) \; \mid \;  \lv \disc(A,B) \rv \geq X/M\} \\
&= \frac{1}{4}\#\{(A,B) \in W_{\ZZ} \cap B(p,X) \; \mid \;  \lv \disc(g(A,B)) \rv \geq 1/M\} \\
&= \frac{1}{4}X\Vol\{(A,B) \in W_{\RR} \cap B(p,1) \; \mid \; \lv \disc(A,B) \rv \geq 1/M\} + O(X^{9/10}).
\end{align*}
Now observe that 
\[
	\lim_{M \rightarrow \infty}\Vol\{(A,B) \in W_{\RR} \cap B(p,1) \; \mid \; \lv \disc(A,B) \rv \geq 1/M\} = 1
\]
so 
\[
\lim_{\eps \rightarrow \infty} \liminf_{X \rightarrow \infty}\frac{\#\{T \in \cF^{x,y,x^2} \; \mid \; H(T) \leq X \text{ and $\psi(f)$ is $(\eps,X)$-close to $p$}\}}{\#\{T \in \cF^{x,y,x^2}\; \mid \; H(T) \leq X \}} = 1
\]
We now prove Condition \eqref{correspondquartic2}. Observe that a positive proportion of lattice points in $B(p,X)$ satisfy the conditions $a_{11} = b_{12} = 1$ and $a_{12} = b_{11} = 0$. Moreover, there exists $C$ such that for every lattice point in $B(p,X)$ with $a_{11} = b_{12} = 1$ and $a_{12}=b_{11} = 0$, there exists a $\GL_2(\ZZ)$-translate with $a_{11} = b_{12} = 1$ and $a_{12} = b_{11} = b_{13} = 0$ in $CB(p,X)$. Thus, a positive proportion of lattice points in $CB(p,X)$ arise from elements of $\cF^{1,x,y,x^2}$; combining this observation with the proof of \cref{upperboundquartic} shows that for $\eps$ sufficiently large, we have:
\[
	\#(\cF(p,\eps,X)\cap \psi(\cF^{1,x,y,x^2})) \ll_{\eps} X^{4/5}
\]
Conversely, \cref{lowerboundquartic} shows that
\[
	\#\cF(p,\eps,X) \gg X^{4/5}.
\]
\end{proof}

\subsubsection{Binary quartic forms}
Let $f = f_0x^4 + f_1x^3y + f_2x^2y^2 + f_3xy^3 + f_4y^4 \in \Sym^4\ZZ^2$ be a binary quartic form. Wood defines a map $\Sym^4 \ZZ^2 \rightarrow \ZZ^2 \otimes \Sym^2\ZZ^3$ which canonically attaches to a binary quartic form a based quartic ring equipped with a cubic resolvent ring (\cite{woodquartic}). In this thesis we use a slight variant of Wood's map\footnote{Our map and the map of Wood send a form $f$ to isomorphic quartic rings and cubic resolvent rings. The only difference is a choice of basis of the quartic ring and cubic resolvent ring.}. In particular, we send a form $f$ to the following element of $\ZZ^2 \otimes \Sym^2\ZZ^3$:
\[
   \begin{bmatrix}
    f_0 & f_1/2 & 0 \\
    - &  f_2 & f_3/2  \\
    - & - & f_4
  \end{bmatrix},
    \begin{bmatrix}
    0 & 0 & 1/2 \\
    - &  1 & 0  \\
    - & - & 0
  \end{bmatrix}.
\]

Let $p = (1/6,1/6,1/6,1/6,1/3)$ and let $\psi$ denote the map from $\Sym^4 \ZZ^2 \rightarrow \ZZ^2 \otimes \Sym^2 \ZZ^3$.  Observe that for any form $f \in \Sym^4 \ZZ^2$, we have $H_c(f) \leq X^{1/6}$ if and only if $\psi(f) \in B(p,X)$. 

\begin{proof}[Proof of \cref{binaryquartictheorem}]
Again, the proof is almost identical to that of \cref{1xyxythm}. Let $M$ be any constant. Let $f \in \cF_4^{\bin}$ and let $(R,C)$ be the corresponding rings. By \cref{lambdaboundlemmaquartic}, if $\Delta \geq X/M$ and $\psi(f) \in B(p,X)$, then
\begin{align*}
\lambda_i(R) &\asymp_M \Delta^{p_i} \\
\lambda_i(C) &\asymp_M \Delta^{q_i}.
\end{align*}
To prove \eqref{correspondquartic1}, it suffices to show that
\[
\lim_{M \rightarrow \infty} \liminf_{X \rightarrow \infty} \frac{\#\{f \in \cF_4^{\bin} \; \mid \; H_c(f) \leq X \text{ and } \Delta \geq X^6/M\}}{\#\{f \in \cF_4^{\bin} \; \mid \; H_c(f) \leq X\}} = 1.
\]
The denominator is clearly $X^5 + o(X^5)$. By Davenport's lemma, we have:
\begin{align*}
&\#\{f \in \cF_4^{\bin}  \; \mid \; H_c(f) \leq X \text{ and } \Delta > X^6/M\} \\
&= \#\{(f_0,\dots,f_4) \in \ZZ^5 \; \mid \; \max\{\lv f_i \rv\}\leq X \text{ and } \lv \disc(f_0,\dots,f_4) \rv > X^6/M\} \\
&= \Vol\{(f_0,\dots,f_4) \in \RR^5 \; \mid \; \max\{\lv f_i \rv\}\leq X \text{ and } \lv \disc(f_0,\dots,f_4) \rv > X^6/M\} + O(X^{4}) \\
&= X^5\Vol\{(f_0,\dots,f_5) \in \RR^5  \; \mid \; \max\{\lv f_i \rv \}\leq 1 \text{ and } \lv \disc(f_0,\dots,f_4) \rv > 1/M\} + O(X^{4})
\end{align*}

We now prove \eqref{correspondquartic2}. Observe that by \cref{upperboundquartic}, there exists a constant $M$ such that for all $\eps > M$, there are $\ll_{\eps} X^{5/6}$ isomorphism classes of tuples $(R,C)$ such that $\Delta \leq X$ and $(R,C)$ is $(\eps,X)$-close to $p$. \cref{evertse} shows that there are at most $2^{80}$ $\GL_2(\ZZ)$-equivalence classes of binary quartic forms giving rise to the same quartic ring. Therefore, it suffices to show that there are $\gg X^{5/6}$ $\GL_2(\ZZ)$-equivalence classes of binary quartics such that $\Delta \leq X$ and the corresponding rings $(R,C)$ are $(\eps,X)$-close to $p$; this is precisely the statement of \cref{binthm} when $r = (1/24,1/24)$.
\end{proof}

\subsubsection{Monic quartic forms}
Let $p = (1/12,1/6,1/4,1/6,1/3)$.  Observe that for any monic integral quartic form $f$, we have $H_r(f) \leq X^{1/12}$ if and only if $H_r(f) \in B(p,X)$. Let $\cF_4^{\mon}$ denote the set of monic quartic forms with integral coefficients and nonzero discriminant.

\begin{proof}[Proof of \cref{monogenicquartictheorem}]
Again, the proof is almost identical to that of \cref{1xyxythm}. Let $M$ be any constant. For any $f \in \cF_4^{\mon}$ let $(R,C)$ be the corresponding rings. By \cref{lambdaboundlemmaquartic}, if $\Delta \geq X/M$ and $\psi(f) \in B(p,X)$, then
\begin{align*}
\lambda_i(R) &\asymp_M \Delta^{p_i} \\
\lambda_i(C) &\asymp_M \Delta^{q_i}.
\end{align*}
To prove \eqref{correspondquartic1}, it suffices to show that
\[
\lim_{M \rightarrow \infty} \liminf_{X \rightarrow \infty} \frac{\#\{f \in \cF_4^{\mon} \; \mid \; H_r(f) \leq X \text{ and } \Delta \geq X^{12}/M\}}{\#\{f \in \cF_4^{\mon} \; \mid \; H_r(f) \leq X\}} = 1.
\]
The denominator is clearly $X^{10} + o(X^{10})$. By Davenport's lemma, we have:
\begin{align*}
&\#\{f \in \cF_4^{\mon}  \; \mid \; H_r(f) \leq X \text{ and } \Delta > X^{12}/M\} \\
&= \#\{(f_1,\dots,f_4) \in \ZZ^4 \; \mid \; \max\{\lv f_i \rv^{1/i}\}\leq X \text{ and } \lv \disc(1,f_1,\dots,f_4) \rv > X^{12}/M\} \\
&= \Vol\{(f_1,\dots,f_4) \in \RR^4 \; \mid \; \max\{\lv f_i \rv^{1/i}\}\leq X \text{ and } \lv \disc(1,f_1,\dots,f_4) \rv > X^{12}/M\} + O(X^{9}) \\
&= X^{10}\Vol\{(f_1,\dots,f_4) \in \RR^4  \; \mid \; \max\{\lv f_i \rv \}\leq 1 \text{ and } \lv \disc(1,f_1,\dots,f_4) \rv > 1/M\} + O(X^{9})
\end{align*}

We now prove \eqref{correspondquartic2}. Observe that by \cref{upperboundquartic}, there exists a constant $M$ such that for all $\eps > M$, there are $\ll_{\eps} X^{3/4}$ isomorphism classes of tuples $(R,C)$ such that $\Delta \leq X$ and $(R,C)$ is $(\eps,X)$-close to $p$. \cref{evertse} shows that there are at most $2^{80}$ $\GL_2(\ZZ)$-equivalence classes of binary quartic forms giving rise to the same quartic ring. Therefore, it suffices to show that there are $\gg X^{3/4}$ $\GL_2(\ZZ)$-equivalence classes of binary quartics such that $\Delta \leq X$, the binary quartic represents $1$, and the corresponding rings $(R,C)$ are $(\eps,X)$-close to $p$; this follows from \cref{binthm} when $r = (0,1/12)$.
\end{proof}

\subsection{Pictures of $\poly_4$, $\poly_4(S_4)$, and $\poly_4(D_4)$}

\begin{figure}[!htb]
\captionsetup{width=.3\linewidth}
\minipage{0.32\textwidth}
 \includegraphics[width=\linewidth]{quarticFullColorCORRECT.png}
 \caption{$\poly_4$ projected onto $p_1$, $p_2$, and $q_1$. The color indicates the value of the density function $d_{\eps,4}$; blue is the highest value.}\label{fig:awesome_image1}
\endminipage\hfill
\minipage{0.32\textwidth}
 \includegraphics[width=\linewidth]{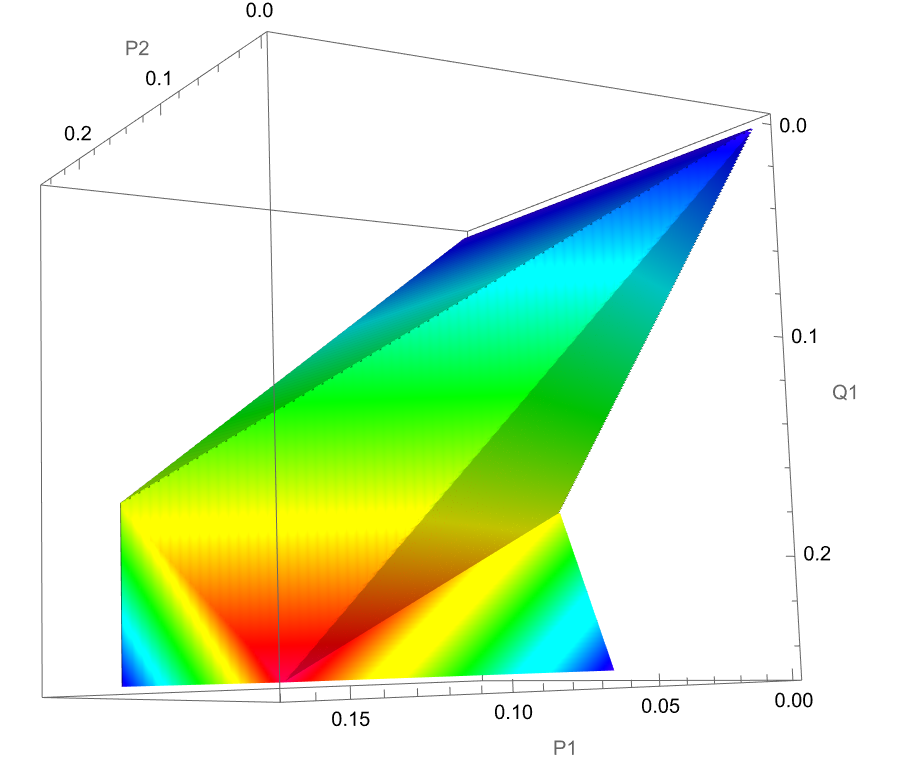}
 \caption{A rotated version of the previous image.}\label{fig:awesome_image2}
\endminipage\hfill
\minipage{0.32\textwidth}%
\includegraphics[width=\linewidth]{d4FullColor.png}
 \caption{$\poly_4(D_4)$ projected onto $p_1$, $p_2$, and $q_1$. The color indicates the value of the lower bound on the density function $d_{\eps,D_4}$ given in \cref{d4densitythm}; blue is the highest value.}\label{fig:awesome_image1}
\endminipage
\end{figure}

\begin{figure}[!htb]
\captionsetup{width=.3\linewidth}
\minipage{0.32\textwidth}
 \includegraphics[width=\linewidth]{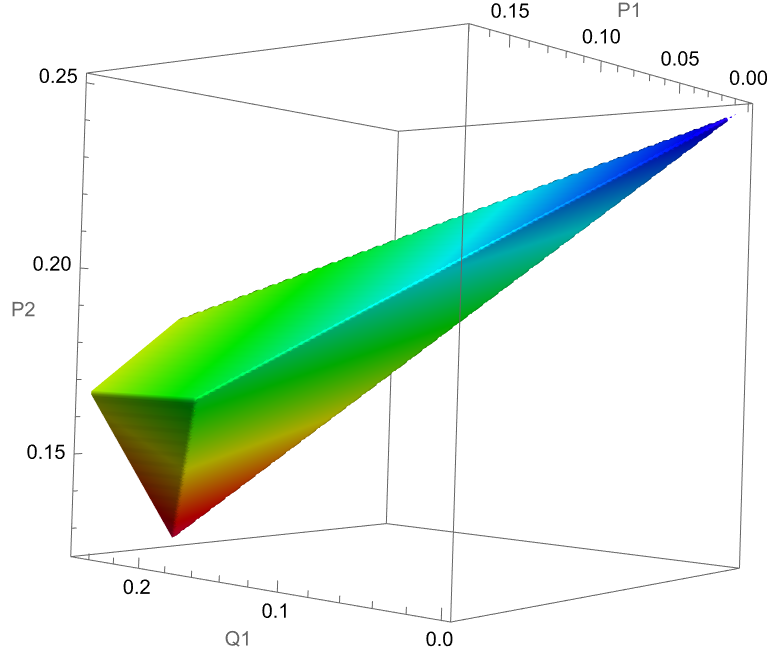}
 \caption{A rotated version of the previous image.}\label{fig:awesome_image2}
\endminipage\hfill
\minipage{0.32\textwidth}
 \includegraphics[width=\linewidth]{s4FullColor.png}
 \caption{$\poly_4(S_4)$ projected onto $p_1$, $p_2$, and $q_1$. The color indicates the value of the density function $d_{\eps,S_4}$; blue is the highest value.}\label{fig:awesome_image1}
\endminipage\hfill
\minipage{0.32\textwidth}%
\includegraphics[width=\linewidth]{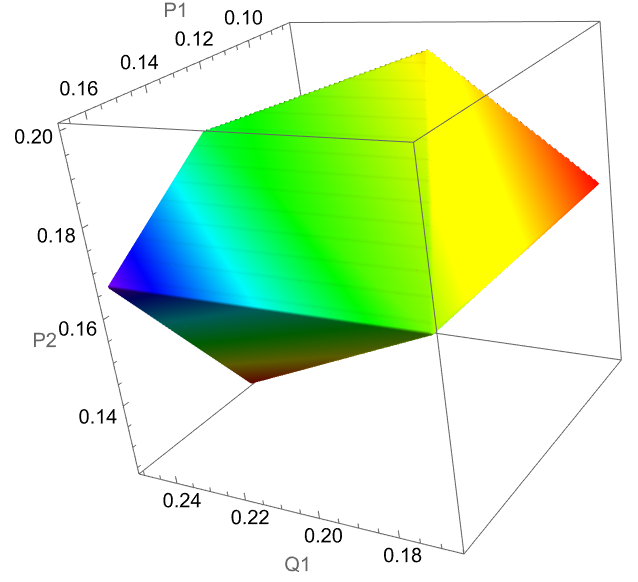}
 \caption{A rotated version of the previous image.}\label{fig:awesome_image2}
\endminipage
\end{figure}

\begin{figure}[!htb]
\captionsetup{width=.45\linewidth}
\minipage{0.5\textwidth}
 \includegraphics[width=\linewidth]{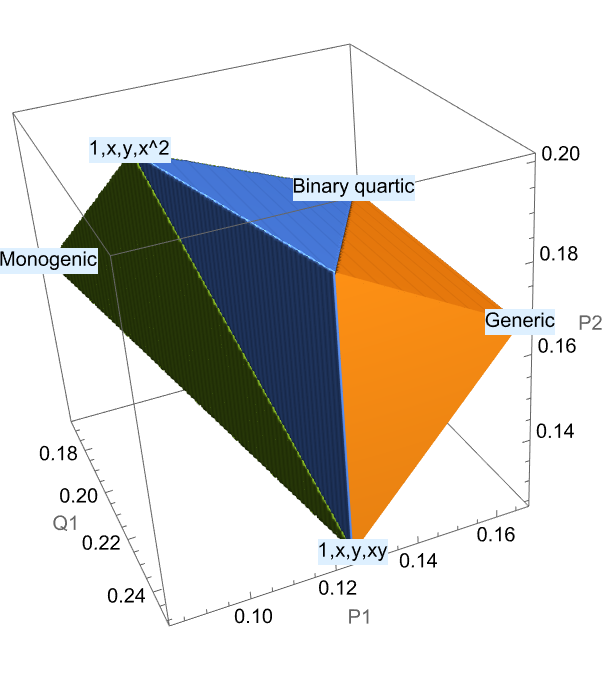}
 \caption{$\poly_4(S_4)$ projected onto $p_1$, $p_2$, and $q_1$. The colored regions are the regions on which $d_{\eps,S_4}$ is linear. The vertices are labelled by the family of rings ``corresponding'' to them.}\label{fig:awesome_image1}
\endminipage\hfill
\minipage{0.5\textwidth}
 \includegraphics[width=\linewidth]{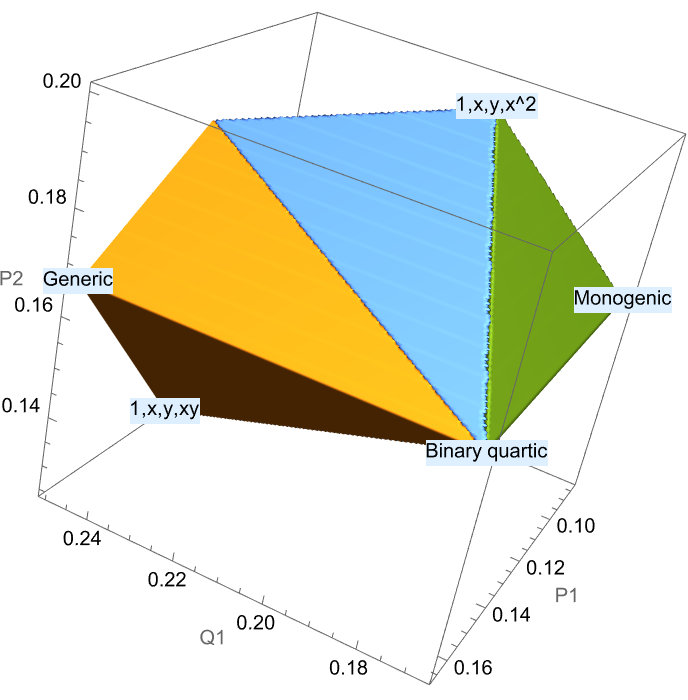}
 \caption{A rotated version of the previous image.}\label{fig:awesome_image2}
\endminipage
\end{figure}

\section{Successive minima of quintic rings}
\label{quintic}
We now investigate density functions of quintic rings. In this section, let $R$ denote a nondegenerate quintic ring with absolute discriminant $\Delta$ and let $S$ denote a sextic resolvent ring of $R$.

\subsection{Additional results on quintic rings}
\subsubsection{Theorems on $\poly_5(G)$}
\begin{theorem}
\label{quinticcontainmentlemma}
For every permutation group $G \subseteq S_5$, the set $\poly_5(G)$ is a finite union of polytopes of dimension $7$. 
\end{theorem}

In fact, the facets of these polytopes can be defined quite explicitly; see \cref{quinticstructuretheorem} for an explicit description of these polytopes. \cref{quinticcontainmentlemma} is a natural generalization of \cite[Theorem 1.1.2]{me} in the quintic case. We now explicitly describe this finite union of polytopes when $G = S_5$. Let
\[
	p = (p_1,\dots,p_4,q_1,\dots,q_5)\in \RR^9
\]
refer to a point of $\RR^9$. By the definition of successive minima and Minkowski's theorem, the set $\poly_5(G)$ satisfies the linear inequalities
\begin{equation}
\label{basicineqquintic}
\begin{tabular}{c c c c}
$0 \leq p_1 \leq \dots \leq p_4$, & $1/2 = \sum_i p_i$, & $0 \leq q_1 \leq \dots \leq q_5$, & and $3/2 = \sum_i q_i$.
\end{tabular}
\end{equation}
 
\begin{theorem}
\label{s5polytopetheorem}
The set $\poly_5(S_5)$ is the set of points $p \in \RR^{9}$ such that Equation~\eqref{basicineqquintic} is satisfied and
\begin{center}
\begin{tabular}{ c c c c }
 $1/2+p_1 \geq q_4+q_1$, & $1/2+p_1\geq q_3+q_2$, & $1/2+p_2\geq q_1+q_5$, & $1/2+p_2\geq q_2+q_4$, \\
 $1/2+p_3 \geq q_5+q_2$, & $1/2+p_3 \geq q_3+q_4$, & and $1/2+p_4\geq q_5+q_3$. & \text{}
\end{tabular}
\end{center}
Moreover, $\poly_5(S_5) = \cup_G \poly_5(G)$ as $G$ ranges across transitive subgroups of $S_5$.
\end{theorem}

\subsubsection{Successive minima of resolvent rings}
We now make a slight detour and discuss successive minima of resolvent rings. A sextic resolvent ring of a quintic $S_5$-ring is a sextic $\psi(S_5)$-ring, where $\psi$ is a transitive embedding $\psi \colon S_5 \xhookrightarrow{} S_6$; the embedding $\psi$ is unique up to outer automorphism. It is a natural question to ask if the successive minima of sextic resolvent rings satisfy the same constraints as those of sextic $\psi(S_5)$-rings. More precisely, is the set of limit points of the multiset
\[
\{(\log_{\lv \Disc(S) \rv}\lambda_1(S),\dots,\log_{\lv \Disc(S) \rv}\lambda_{5}(S))\} \subseteq \RR^{5}
\]
as $S$ ranges across sextic resolvent rings of quintic $S_5$-rings equal to set of limit points of the multiset
\[
\{(\log_{\lv \Disc(S) \rv}\lambda_1(S),\dots,\log_{\lv \Disc(S) \rv}\lambda_{5}(S))\} \subseteq \RR^{5}
\]
as $S$ ranges across sextic $\psi(S_5)$-rings? Let $R_1$ denote the first set of limit points and let $R_2$ denote the second.

\begin{corollary}
\label{resolventsextic}
$R_1$ is strictly contained in $R_2$.
\end{corollary}
\begin{proof}
The first region $R_1$ is $1/3$ times the projection of $\poly_5(S_5)$ onto the last $5$ coordinates of $\RR^9$. By \cref{onlyifthmshorten}, the second region $R_2$ is contained in of the set of points $q = (q_1,\dots,q_5) \in \RR^5$ such that:
\begin{enumerate}
\item $1/2 = q_1 + \dots + q_5$;
\item $0 \leq q_1 \leq \dots \leq q_5$;
\item and $q_{i + j} \leq q_i + q_j$ for all $1 \leq i \leq j < i + j < 6$.
\end{enumerate}
Notice that this region is precisely $\Len(6)$; by \cite[Theorem 1.1.2]{me}, $\Len(6)$ is contained in $R_2$, as $\Len(6)$ is the polytope of the flag generated by a monogenic basis of number field. In particular, the second region contains the point $\frac{1}{30}(1,2,3,4,5)$ but the first does not.
\end{proof}

\subsubsection{Results on density functions of quintic rings}
Let $p = (p_1,\dots,p_4,q_1,\dots,q_5) \in \RR^9$ denote a point. For any $\eps \in \RR_{>0}$, $X \in \RR_{> 1}$, and permutation group $G \subseteq S_5$, define $\cF(p,\eps,G,X)$ to be the set of isomorphism classes of pairs $(R,S)$ such that $R$ is a $G$-ring, $\Delta \leq X$, and 
\begin{align}
\label{condquinticintro}
\max_{i = 1,2,3,4}\{\lv \log_{\Delta}\lambda_i(R) - p_i \rv\} &\leq \frac{\eps}{\log X} \\
\label{condquinticintro2}
\max_{i = 1,2,3,4,5}\{\lv \log_{\Delta}\lambda_i(S) - q_i \rv\} &\leq \frac{\eps}{\log X}.
\end{align}
Define the density function $d_{\eps,G} \colon \RR^9 \rightarrow \RR \cup \{\ast\}$ which assigns to $p \in \RR^9$ the value 
\[
	\lim_{X \rightarrow \infty}\frac{\log (\#\cF(p,\eps,G,X)+1)}{\log X}
\]
if it exists and else assigns $\ast$. Similarly, let $\cF(p,\eps,X)$ be the set of isomorphism classes of pairs $(R,S)$ such that $\Delta \leq X$, \eqref{condquinticintro}, and \eqref{condquinticintro2}. Let $d_{\eps, 5}$ denote the corresponding density function. We now compute the total density function for quintic rings and the $S_5$-density function. In this section, we prove \cref{quintictotaldensity} and \cref{s5densitythm}.

\subsubsection{Special families of quintic rings}
Again, we see that quintic rings with extremal successsive minima tend to also be extremal  in an \emph{algebraic} sense. In the quartic case, we considered vertices of $\poly_4(S_4)$ and showed that many of these vertices corresponded to interesting families of quintic rings. The polytope $\poly_5(S_5)$ is significantly more complicated than $\poly_4(S_4)$; the polytope $\poly_5(S_5)$ has $18$ vertices, and the author believes that many of them likely give rise to interesting and natural families of quartic rings. In this section, we discuss the following $3$ of these vertices:
\begin{align*}
&(1/8,1/8,1/8,1/8,3/10,3/10,3/10,3/10,3/10) \\
&(1/8, 1/8, 1/8, 1/8, 1/4, 1/4, 1/4, 3/8, 3/8) \\
&(1/20,2/20,3/20,4/20,4/20,5/20,6/20,7/20,8/20)
\end{align*}
We'll show that the first ``corresponds'' to generic $S_5$-rings, the second ``corresponds'' to binary quintic forms, and the third ``corresponds'' to monogenic quintic rings. We now formalize our definition of ``correspond''.

We say a pair $(R,S)$ is $(\eps,X)$-close to $p$ if
\begin{align}
\label{condquintic}
\max_{1\leq i \leq 4}\{\lv \log_{\Delta}\lambda_i(R) - p_i \rv\} &\leq \frac{\eps}{\log X} \\
\label{condquintic2}
\max_{1\leq i \leq 5}\{\lv \log_{\Delta}\lambda_i(S) - q_i \rv\} &\leq \frac{\eps}{\log X}.
\end{align}

Let $\cF$ be an infinite set equipped with a map $H \colon \cF \rightarrow \RR_{> 0}$ such that for all $X$, we have $\#\{f \in \cF \mid H(f) \leq X \} < \infty$. We call $H$ the \emph{height function} of $\cF$. Further suppose that $\psi$ is a map from $\cF$ to isomorphism classes of nondegenerate pairs of quintic rings equipped with a sextic resolvent. For a point $p \in \poly_5$, we say the triple \emph{$(\cF,H,\psi)$ corresponds to $p$} if 
\begin{equation}
\label{correspondquintic1}
\lim_{\eps \rightarrow \infty} \liminf_{X \rightarrow \infty}\frac{\#\{f \in \cF \; \mid \; H(f) \leq X \text{ and $\psi(f)$ is $(\eps,X)$-close to $p$}\}}{\#\{f \; \mid \; H(f) \leq X \}} = 1
\end{equation}
and there exists a constant $C$ such that for all $\eps > C$, we have
\begin{equation}
\label{correspondquintic2}
\liminf_{X \rightarrow \infty}\frac{\#\{(R,S)  \in \cF(p,\eps, X)\; \mid \; (R,S) \in \psi(\cF)\}}{\#\cF(p,\eps, X)} > 0.
\end{equation}
When $H$ or $\psi$ are implicit, we simply write that $\cF$ corresponds to $p$. The following is a corollary of upcoming work of H.

\begin{corollary}
\label{generics5thm}
The set of isomorphism classes of nondegenerate $S_5$-rings, where $H$ is the absolute discriminant, corresponds to
\[
	(1/8,1/8,1/8,1/8,3/10,3/10,3/10,3/10,3/10).
\]
\end{corollary}

The root height of a monic degree $n$ polynomial $f(x) = x^n + a_{1}x^2 + \dots + a_n \in \ZZ[x]$ is defined to be $H_r(f) \coloneqq \max_i\{\lv a_i\rv^{1/i}\}$. Given a binary $n$-ic form $f(x,y) = a_0x^n + \dots + a_ny^n$, define the \emph{coefficient height of $f$} to be $H_c(f) \coloneqq \max_i\{\lv a_i\rv\}$. Fess \cite{fess} constructed a map from $\Sym^5 \ZZ^2 \rightarrow \ZZ^4 \otimes \wedge^2 \ZZ^5$; via Bhargava's parametrization of quintic rings \cite{BhargavaHCL4}, this map canonically associates to an integral binary quintic form the data of a quintic ring and a sextic resolvent ring. Fess' construction is compatible with that of Birch and Merriman \cite{birch}; both constructions give isomorphic quintic rings from a binary quintic form. Let $\cF_5^{\bin}$ be the set of binary quintic forms with integral coefficients and nonzero discriminant. 

\begin{theorem}
\label{binaryquintictheorem}
Then $(\cF_5^{\bin}, H_c)$ corresponds to
\[
	(1/8, 1/8, 1/8, 1/8, 1/4, 1/4, 1/4, 3/8, 3/8).
\] 
\end{theorem}

As in the case of monogenic cubics, monogenic quartics, and quartic rings coming from binary quartic forms (\cref{monogeniccubictheorem}, \cref{monogenicquartictheorem}, \cref{binaryquartictheorem}), there is a good reason to expect that $\cF_5^{\bin}$ (equipped with the coefficient height) corresponds to the point above; this reason comes from Fess' characterization of pairs $(R,S)$ arising from binary quintic forms. For a rank $n$ ring $R$, let $\widetilde{R} \coloneqq \Hom(R/\ZZ,\ZZ)$. Recall that a pair $(R,S)$ comes equipped with a \emph{sextic resolvent map} $\varphi \colon \widetilde{S} \otimes \widetilde{S} \rightarrow \widetilde{R}$. For an basis $\{v_0=1,v_1,\dots,v_{n-1}\}$ of a nondegenerate rank $n$ ring, let $\{v^*_0,v^*_1,\dots,v^*_{n-1}\}$ denote the dual basis with respect to the trace pairing.

\begin{theorem}[Fess \cite{fess}, Theorem~5.1.4]
Let $(R,S)$ be a nondegenerate pair. There exists $f \in \Sym^5 \ZZ^2$ giving rise to the pair $(R,S)$ if and only if $S$ has a basis $\{1,w_1,\dots,w_5\}$ such that
\[
	(w_1,w_2,w_3) \equiv 8\Disc(R)((w^*_5)^2,2w^*_4w^*_5,(w^*_4)^2) \Mod{\QQ}
\]
and $\varphi(w_4^*,w_5^*) = 0$.
\end{theorem}

From this classification, one might naively guess the following: if a pair $(R,S)$ arises from a nondegenerate binary quintic form and $\{1,w_1,\dots,w_5\}$ is a Minkowski basis of $S$, then it is likely that
\begin{align*}
\lv w_1 \rv &\asymp \Delta \lv w_5\rv^{-2} \\
\lv w_2 \rv &\asymp \Delta \lv w_4\rv^{-1}\lv w_5\rv^{-1} \\
\lv w_3 \rv &\asymp \Delta \lv w_4\rv^{-2}.
\end{align*}
For lack of other information, one might also guess that $\lv w_4 \rv$ is of similar size to $\lv w_5 \rv$ and that the successive minima of the quintic ring $R$ are roughly equal as well. \cref{binaryquintictheorem} formalizes this heuristic. A similar heuristic informs the monogenic case. Let $\cF_5^{\mon}$ be the set of binary quintic forms with integral coefficients and nonzero discriminant. 

\begin{theorem}
\label{monogenicquintictheorem}
Then $(\cF_5^{\mon}, H_r)$ corresponds to
\[
	(1/20,2/20,3/20,4/20,4/20,5/20,6/20,7/20,8/20).
\] 
\end{theorem}
Observe that, as case of \cref{binaryquintictheorem}, the vertex
\[
	(1/20,2/20,3/20,4/20,4/20,5/20,6/20,7/20,8/20)
\]
satisfies the conditions:
\begin{align*}
q_1 &= 1 - 2q_5 \\
q_2 &= 1 - q_4 - q_5 \\
q_3 &= 1 - 2q_4
\end{align*}

\subsection{The parametrization of quintic rings}
\label{quinticbackground}
Our primary tool will be Bhargava's parametrization of quintic rings; we now briefly summarize this parametrization. Let $V_{\RR} = \RR^4 \otimes \wedge^2 \RR^5$. We write an element $\mathcal{A} = (A_1,A_2,A_3,A_4)$ as a quadruple of $5\times 5$ real alternating matrices as follows:
\[
  A_k = \begin{bmatrix}
    0 & a^k_{12} & a^k_{13} & a^k_{14} & a^k_{15} \\
    -a^k_{12} & 0 & a^k_{23} & a^k_{24} & a^k_{25} \\
    -a^k_{13} & - & 0 & a^k_{34} & a^k_{35} \\
    -a^k_{14} & -a^k_{24} & -a^k_{34} & 0 & a^k_{45} \\
    -a^k_{15} & -a^k_{25} & -a^k_{35} & -a^k_{45} & 0 
  \end{bmatrix}.
\]
An element $\mathcal{A}$ is said to be \emph{integral} if $a^k_{ij} \in \ZZ$, and the lattice of integral elements is denoted $V_{\ZZ}$. The group $\GL_4(\RR) \times \GL_5(\RR)$ acts on $V_{\RR}$ in a natural way; an element $(g_4,g_5)$ sends $\mathcal{A}$ to
\[
	g_4\begin{bmatrix}
	g_5A_1g_5^t \\
	g_5A_2g_5^t \\
	g_5A_3g_5^t \\
	g_5A_4g_5^t
	\end{bmatrix}.
\]
The subgroup $G_{\ZZ} = \GL_4(\ZZ) \times \GL_5(\ZZ)$ acts on $V_{\RR}$ and preserves the lattice $V_{\ZZ}$. The action of $G_{\ZZ}$ on $V_{\ZZ}$ has a unique polynomial invariant, denoted $\Disc(\mathcal{A})$. The orbits of $G_{\ZZ}$ acting on $V_{\ZZ}$ parametrize quintic rings equipped with a sextic resolvent ring; see \cite{BhargavaHCL4} for more details.

\begin{theorem}[Bhargava, \cite{BhargavaHCL4}]
There is a canonical, discriminant-preserving bijection between $V_{\ZZ}$ and isomorphism classes of pairs $(R,S)$ where $R$ is a based quintic ring and $S$ is a based sextic resolvent ring of $R$. Under the action of $G_{\ZZ}$, this bijection descends to a bijection between the set of $G_{\ZZ}$ orbits of $V_{\ZZ}$ and the set of isomorphism classes of pairs $(R,S)$ where $R$ is a quintic ring and $S$ is a sextic resolvent ring of $R$. 
\end{theorem}

We now describe a few important details of this bijection. Let $\widetilde{R}$ and $\widetilde{S}$ denote $\Hom(R/\ZZ,\ZZ)$ and $\Hom(S/\ZZ,\ZZ)$ respectively. A sextic resolvent ring $S$ of a quintic ring comes equipped with an alternating map $\widetilde{S} \otimes \widetilde{S} \rightarrow \widetilde{R}$. Choosing bases for $R/\ZZ$ and $S/\ZZ$ gives rise to an element of $V_{\ZZ}$. Conversely, given an element $\mathcal{A} \in V_{\ZZ}$, we obtain bases for $R$ and $S$ via certain explicit polynomials in the $a^k_{ij}$ that form the structure coefficients for these bases. 

\subsection{Reduction theory of quintic rings}

We now fix some notation we use throughout this section. Unless otherwise specified, let $\mathcal{A}=(A_1,A_2,A_3,A_4)$ denote an element of $V_{\ZZ}$ and let $(R,S)$ denote the corresponding quintic ring and sextic resolvent ring.  Let $\Delta = \lv \Disc(\mathcal{A}) \rv = \lv \Disc(R) \rv$. For a basis $\{v_0=1,v_1,\dots,v_{n-1}\}$ of a nondegenerate rank $n$ ring $R$, let $\{v^*_0,v^*_1,\dots,v^*_{n-1}\}$ denote the dual basis with respect to the trace pairing and let $R^* = \Hom(R, \ZZ)$. Unless otherwise specified, $\{1=v_0,\dots,v_4\}$ and $\{1=w_0,\dots,w_5\}$ will denote bases of the quintic ring $R$ and the sextic resolvent ring $S$ respectively. Given a pair $(R,S)$ and bases $\{1=v_0,\dots,v_4\}$ and $\{1=w_0,\dots,w_5\}$ respectively, we obtain a tensor $\mathcal{A}$ by writing the resolvent $\varphi \colon \widetilde{S}\otimes \widetilde{S} \rightarrow \widetilde{R}$ map in terms of the bases $\{v^*_4,\dots,v^*_1\}$ and $\{w^*_5,\dots,w^*_1\}$ of $\widetilde{R}$ and $\widetilde{S}$ respectively\footnote{This ordering is very important! We encourage the reader to keep this ordering of the basis elements in mind throughout this section.}. In this section, we develop the relationship between:
\begin{enumerate}
\item pairs $(R,S)$ equipped with ``small'' bases;
\item and tensors $\mathcal{A}$ with ``small'' coefficients.
\end{enumerate}

\subsubsection{``Short'' bases of $R,S$ gives rise to matrices $\mathcal{A}$ with ``small'' coefficients}
We will need to relate the lengths of $\{v^*_1,\dots,v^*_4\}$ and $\{w^*_1,\dots,w^*_5\}$ to the lengths of $\{1,v_1,\dots,v_4\}$ and $\{1,w_1,\dots,w_5\}$, when the latter two bases are ``short''; the content of the following lemma consists of the relationship between these two lengths. 
 
\begin{lemma}
\label{succminlemma}
Let $R$ be any nondegenerate rank $n$ ring and let $\{1=v_0,\dots,v_{n-1}\}$ be a basis of $R$. Let $\theta$ be a positive real number such that the angle between any $v_i$ and $v_j$ in $\RR^n$ is bounded below by $\theta$ in the Minkowski embedding of $R$. For all $0 \leq i \leq n-1$, we have
\[
	\lv v_i\rv \asymp_{n,\theta} \lambda_i(R) \asymp_{n,\theta} \lambda_{n-1-i}(R^*)^{-1} \asymp_{n,\theta} \lv v_i^{*}\rv^{-1}.
\]
\end{lemma}
\begin{proof}
Let $r$ be the number of nonzero homomorphisms of $R$ into $\RR$ and let $s$ be the number of complex conjugate pairs of nonzero homomorphisms of $R$ into $\CC$ whose image does not lie in $\RR$. The Minkowski embedding of $R$ into $\RR^n$ is given via the composition of $R\rightarrow \RR^r \oplus \CC^s$ with the map
\begin{align*}
\RR^r \oplus \CC^s &\rightarrow \RR^n \\
(y_1,\dots,y_r,z_1,\dots,z_s) &\rightarrow (y_1,\dots,y_r,\Ree(z_1),\Imm(z_1),\dots,\Ree(z_s),\Imm(z_s)).
\end{align*}
Because all norms on finite dimensional vector spaces are equivalent, it suffices to suppose that all lengths are taken with respect to the Euclidean norm on $\RR^n$, i.e. let $\lv \cdot \rv$ denote the Euclidean norm. Note that $v_0,\dots,v_{n-1}$ form a set of linearly independent vectors in $\RR^n$ via the Minkowski embedding; let $\{ v_0^{\vee},\dots, v_{n-1}^{\vee}\}$ denote the dual basis with respect to the standard Euclidean norm. We will show that the length of $v_i^*$ is comparable to the length of $v_i^{\vee}$. Afterwards, we compute the lengths of $v_i^{\vee}$.

On $\RR^n$, the trace form is given by the quadratic form
\[
	q(x_1,\dots,x_n) = \sum_{1\leq i \leq r}x_i^2 + \sum_{1\leq i \leq s}(2x_{r + 2i-1}^2-2x_{r + 2i}^2).
\]
Let $\sigma \colon \RR^n \rightarrow \RR^n$ be the map sending
\[
	(y_1,\dots,y_r,z_1,\dots,z_s) \rightarrow (y_1,\dots,y_r,\frac{1}{2}\overline{z_1},\dots,\frac{1}{2}\overline{z_s}).
\]
For $0 \leq i \leq n-1$ write $v_i = (v_{1,i},\dots,v_{n,i})$ and $v^{\vee}_i = (v^{\vee}_{1,i},\dots,v^{\vee}_{n,i})$. 
\begin{align*}
q(\sigma(v_i^{\vee}),v_j) &= \sum_{1 \leq k \leq r}v^{\vee}_{k,i}v_{k,j} + \sum_{1 \leq k \leq s}(2\frac{1}{2}v^{\vee}_{r+2k-1,i}v_{r+2k-1,j} - 2\frac{-1}{2}v^{\vee}_{r+2k,i}v_{r+2k,j}) \\
&= \sum_{1 \leq k \leq r}v^{\vee}_{k,i}v_{k,j} + \sum_{1 \leq k \leq s}(v^{\vee}_{r+2k-1,i}v_{r+2k-1,j} + v^{\vee}_{r+2k,i}v_{r+2k,j}) \\
&= \delta_{ij}
\end{align*}
Thus, $\sigma(v_i^{\vee}) = v_i^*$ for all $0 \leq i \leq n-1$. Moreover, the angle between $v_i^{\vee}$ and $v_i$ is $\leq c_{\theta} \pi/2$ where $0 \leq c_{\theta} < 1$ and $c_{\theta}$ is dependent only on $\theta$. Hence
\begin{align*}
1 &= v_i^{\vee}\cdot v_i && \text{because $v_i^{\vee}\cdot v_j = \delta_{ij}$}\\
&\asymp_{\theta} \lv v_i^{\vee} \rv \lv v_i \rv && \text{because of the angle bound $c_{\theta}$}\\
&= \lv \sigma^{-1}(v_i^{*}) \rv \lv v_i \rv && \text{because $\sigma(v_i^{\vee}) = v_i^*$} \\
&\asymp_n \lv v_i^{*} \rv \lv v_i \rv.
\end{align*}
Moreover,
\[
	\Covol(R)^{-1} \asymp_{n,\theta} \prod_{i = 0}^{n-1}\lv v_i \rv^{-1} \asymp_{n,\theta} \prod_{i = 0}^{n-1}\lv v_i^{*} \rv \gg_n \prod_{i = 0}^{n-1} \lambda_i(R^*) \asymp_{n} \Covol(R^*) \asymp_{n} \Covol(R)^{-1}.
\]
Hence, $\lv v_i^* \rv \asymp_{n,\theta} \lambda_{n - 1 - i}(R^{\vee})$, which completes the proof.
\end{proof}

We now show that if the sextic resolvent map is expressed with respect to ``short'' bases of $R$ and $S$, then the coefficients of the map are ``small''. To do so, we will need the following explicit description of the sextic resolvent map. Let $\sigma_1,\dots,\sigma_6$ be the $6$ nonzero homomorphisms of $S$ into the complex numbers. The sextic resolvent map $\varphi \colon \widetilde{S} \otimes \widetilde{S} \rightarrow \widetilde{R}$ is given by the formula:
\[
	\varphi(u \otimes v) = \frac{\sqrt{\Disc(S)}}{48\Disc(R)}\;\; \begin{array}{|c c c|}
1 & 1 & 1 \\
\sigma_1(u) + \sigma_2(u) & \sigma_3(u) + \sigma_6(u) & \sigma_4(u) + \sigma_5(u) \\
\sigma_1(v) + \sigma_2(v) & \sigma_3(v) + \sigma_6(v) & \sigma_4(v) + \sigma_5(v) 
\end{array}
\]
\begin{lemma}
\label{matrixboundquintic}
If $\{1,v_1,\dots,v_4\}$ and $\{1,w_1,\dots,w_5\}$ are Minkowski bases of $R$ and $S$ respectively, then for any $1 \leq i,j \leq 4$ and any $1 \leq k \leq 5$, we have:
\[
	\lv a^k_{ij} \rv \ll \frac{\Delta^{1/2}\lv v_{5-k}\rv}{\lv w_{6-i}\rv\lv w_{6-j}\rv}
\]
\end{lemma}
\begin{proof}
From the definition of the sextic resolvent map, we have that:
\begin{align*}
\lv a^k_{ij} \rv &\ll \frac{\lv \varphi(w^*_{6-i} \otimes w^*_{6-j})\rv}{\lv v^*_{5-k}\rv} \\
&  \ll \frac{\Delta^{1/2}\lv w^*_{6-i}\rv\lv w^*_{6-j}\rv}{\lv v^*_{5-k}\rv} && \text{by the explicit description of $\varphi$} \\
&\ll \frac{\Delta^{1/2}\lv v_{5-k}\rv}{\lv w_{6-i}\rv\lv w_{6-j}\rv} && \text{by \cref{succminlemma}}
\end{align*}

\end{proof}

Let $p = (p_1,\dots,p_4,q_1,\dots,q_5) \in \RR^9$ be any point. 

\begin{definition}
Let $B(p,X)$ denote the set of elements $\mathcal{A} \in V_{\RR}$ such that
\[
	\lv a^k_{ij} \rv \leq \exp[X,1/2 + p_{5-k} - q_{6-i} - q_{6-j}].
\]    
\end{definition}

We now fix more notation we use throughout this section. Let $\eps$ denote a positive real number, let $C,D$ denote real numbers $\geq 1$, and let $X$ denote a real number $>1$. Recall that we say the pair $(R,S)$ is $(\eps,X)$-close to $p$ if:
\begin{align*}
\max_{1 \leq i \leq 4}\{\lv \log_{\Delta}\lambda_i(R) - p_i \rv\} &\leq \frac{\eps}{\log X} \\
\max_{1 \leq i \leq 5}\{\lv \log_{\Delta}\lambda_i(S) - q_i \rv\} &\leq \frac{\eps}{\log X}
\end{align*}

\begin{definition}
For an integral element $\mathcal{A}$, we say $(\mathcal{A},p)$ are $(\eps,X)$-close if the corresponding pair $(R,S)$ is $(\eps,X)$-close to $p$. We say an integral element $\mathcal{A}$ \emph{comes from Minkowski bases} if there exist Minkowski bases of the corresponding rings giving rise to $\mathcal{A}$.    
\end{definition}

\begin{lemma}
\label{ublemmaquintic}
There exists $C$ such that for all $\eps$, $\mathcal{A}$ and $p$ the following statement holds: if $\Delta > 1$ and $(\mathcal{A},p)$ are $(\eps, \Delta)$-close and $\mathcal{A}$ comes from Minkowski bases, then $\mathcal{A} \in Ce^{3\eps}B(p,\Delta)$.
\end{lemma}
\begin{proof}
Let $\{1,v_1,\dots,v_4\}$ and $\{1,w_1,\dots,w_5\}$ be the Minkowski bases of $(R,S)$ giving rise to $\mathcal{A}$. 
We have:
\begin{align*}
    \lv a^k_{ij} \rv &\ll \frac{\Delta^{1/2}\lv v_{5-k}\rv}{\lv w_{6-i}\rv\lv w_{6-j}\rv} && \text{by \cref{matrixboundquintic}} \\
    &\asymp \frac{\Delta^{1/2}\lambda_{5-k}(R)}{\lambda_{6-i}(S)\lambda_{6-j}(S)} && \text{by \cref{succminlemma}}\\
    &\leq e^{3\eps}\exp[X,1/2 + p_{5-k} - q_{6-i} - q_{6-j}] && \text{because $(\mathcal{A},p)$ is $(\eps, \Delta)$-close}
\end{align*}
\end{proof}

\subsubsection{Matrices $\mathcal{A}$ with ``small'' coefficients gives rise to ``short'' bases of $R,S$}
Choose $\mathcal{A} \in V_{\ZZ}$ and let $\{v_0=1,v_1,\dots,v_4\}$ and $\{w_0=1,w_1,\dots,w_5\}$ be the explicit bases of $R$ and $S$ obtained from $\mathcal{A}$ via the explicit multiplication coefficients in \cite{BhargavaHCL4}.
\begin{lemma}
\label{quinticboundlemma}
For all $\mathcal{A}$, $C$, $p$, and $X$, the following statement holds: if $p$ satisfies Equation \eqref{basicineqquintic} and $\mathcal{A} \in CB(p,X)$, then $\lv v_i \rv \ll_C X^{p_i}$ for $i = 1,2,3,4$ and $\lv w_i \rv \ll_C X^{q_i}$ for $i = 1,2,3,4,5$.
\end{lemma}
\begin{proof}
Using the fact that $\mathcal{A} \in B(p,X)$, we may bound the size of the multiplication coefficients given in \cite{BhargavaHCL4}. We obtain:
\[
	\lv v_i\rv^2 \ll_C \max_{0 \leq j \leq 4}\{X^{2p_i-p_j}\lv v_j \rv\}
\]
\[
	\lv w_i\rv^2 \ll_C \max_{0 \leq j \leq 5}\{X^{2q_i-q_j}\lv w_j \rv\}
\]
Applying \cref{generalboundlemma} completes the proof.
\end{proof}

\begin{lemma}
\label{lambdaboundlemmaquintic}
For all $p$ satisfying Equation \eqref{basicineqquintic}, all $X,C$, and all $\mathcal{A}$, the following statement holds: if $\mathcal{A} \in V_{\ZZ} \cap CB(p,X)$ and $\frac{X}{2} \leq \Delta$ then 
\[
	 \lambda_i(R) \asymp_C \Delta^{p_i}
\]
for $i = 1,2,3,4$ and
\[
	 \lambda_i(C) \asymp_C \Delta^{q_i}
\]
for $i = 1,2,3,4,5$.
\end{lemma}
\begin{proof}
We have
\begin{align*}
\prod_{i = 1}^4\lambda_i(R) &\ll \prod_{i = 1}^4\lv v_i\rv && \text{by definition of successive minima} \\
&\ll_C \prod_{i = 1}^4X^{p_i} && \text{by \cref{quinticboundlemma}}\\
&= X^{1/2} && \text{because $p_1 + p_2 + p_3 + p_4 = 1/2$}\\
&\asymp \Delta^{1/2} && \text{because $\Delta \geq X/2$} \\
&\asymp \prod_{i = 1}^4\lambda_i(R). && \text{by Minkowski's second theorem}
\end{align*}
Thus $\lambda_i(R) \asymp_C \Delta^{p_i}$ for $i = 1,2,3,4$. The same argument shows, for the sextic resolvent, that $\lambda_i(S) \asymp_C \Delta^{q_i}$ for $i = 1,2,3,4,5$.
\end{proof}

\subsection{Upper and lower bounding the number of $\mathcal{A} \in B(p,X)$ giving rise to isomorphic $(R,S)$}
\begin{lemma}
\label{quinticlblemma}
For all $p$ satisfying Equation \eqref{basicineqquintic}, all $X,C$, and all $\mathcal{A}$, the following statement holds: if $\mathcal{A} \in V_{\ZZ} \cap CB(p,X)$ and $X/2 \leq \Delta \leq X$, then there are
\[
	\#\{\mathcal{A}' \in CB(p,X) \cap V_{\ZZ} \mid \; (R,S) \simeq (R',S') \} \ll_C \exp[X,\sum_{1 \leq i < j \leq 4}(p_j-p_i)+\sum_{1 \leq i < j \leq 5}(q_j-q_i)].
\]
\end{lemma}
\begin{proof}
Let $\{v_0=1,v_1,v_2,v_3,v_4\}$ and $\{w_0=1,w_1,\dots,w_5\}$ be the explicit bases of $R$ and $S$ obtained from $\mathcal{A}$. By \cref{quinticboundlemma}, we have $\lv v_i \rv \ll_C X^{p_i}$ for $i = 1,2,3,4$ and  $\lv w_i \rv \ll_C X^{q_i}$ for $i = 1,2,3,4,5$. Because
\[
	X^{1/2} \asymp \Delta^{1/2} \ll \prod_{i = 1}^4\lv v_i \rv \ll_C X^{\sum_{i = 1}^4 p_i} = X^{1/2}
\]
and
\[
	X^{3/2} \asymp \Delta^{3/2} \ll \prod_{i = 1}^5\lv w_i \rv \ll_C X^{\sum_{i = 1}^5 q_i} = X^{3/2}
\]
we have $\lv v_i \rv \asymp_C X^{p_i}$ for $i = 1,2,3,4$ and  $\lv w_i \rv \asymp_C X^{q_i}$ for $i = 1,2,3,4,5$.

Suppose $\mathcal{A}' \in CB(p,X) \cap V_{\ZZ}$ is such that $(R',S') \simeq (R,S)$. Let $\{1,v'_1,\dots,v'_4\}$ and $\{1,w'_1,\dots,w'_5\}$ denote the bases of $R'$ and $S'$ given by $\mathcal{A}'$. By \cref{quinticboundlemma}, we have $\lv v'_i  \rv \asymp_C X^{p_i}$ for $i = 1,2,3,4$ and  $\lv w'_i \rv \asymp_C X^{q_i}$ for $i = 1,2,3,4,5$. Moreover, write
\begin{align*}
v'_i &= c_{i0} + \sum_{j = 1}^4 c_{ij} v_j \\
w'_i &= d_{i0} + \sum_{j = 1}^5 d_{ij} w_j.
\end{align*}
The tensor $\mathcal{A}'$ is determined by the choices of $c_{ij}$ and $d_{ij}$. Letting $p_0 = q_0 = 0$, we have:
\begin{align*}
X^{p_i} &\asymp_C \lv v'_i \rv \asymp_C \max_{0 \leq j < 5}\{\lv c_{ij}\rv \lv v_j \rv\} \asymp_C \max\{0 \leq j < 5\}\{\lv c_{ij}\rv X^{p_j}\} \\
X^{q_i} &\asymp_C \lv w'_i \rv \asymp_C \max_{0 \leq j < 6}\{\lv d_{ij}\rv \lv w_j \rv\} \asymp_C \max\{0 \leq j < 6\}\{\lv d_{ij}\rv X^{q_j}\}
\end{align*}
Thus if $j \leq i$, the number of choices for $c_{ij}$ and $d_{ij}$ is $\ll_C 1$. If $j > i$, the number of choices for $c_{ij}$ is $\ll_C X^{p_j-p_i}$ and the number of choices for $d_{ij}$ is $\ll_C X^{q_j-q_i}$. Hence, there are $\ll_C \exp[X,\sum_{1 \leq i < j \leq 4}(p_j-p_i)+\sum_{1 \leq i < j \leq 5}(q_j-q_i)]$ distinct integral $\mathcal{A}' \in CB(p,X)$ such that $(R',S') \simeq (R,S)$. 
\end{proof}

\begin{lemma}
\label{ublemma3}
There exists $D$ such that for all $p$, $X$, $C$ and $\mathcal{A}$, the following statement holds: if $p$ satisfies Equation \eqref{basicineqquintic} and $\mathcal{A} \in CB(p,X)$ then
\[
	\#\{\mathcal{A}' \in DCB(p,X) \cap V_{\ZZ} \mid \; (R,S) \simeq (R',S') \} \geq \exp[X,\sum_{1\leq i < j \leq 4}(p_j-p_i)+\sum_{1\leq i < j \leq 5}(q_j-q_i)]\}.
\]
where here $(R',S')$ denotes the quintic ring and sextic resolvent ring arising from tensor $\mathcal{A}'$.
\end{lemma}
\begin{proof}
Recall that $V_{\ZZ}$ admits an action of $G_{\ZZ} = \GL_4(\ZZ)\times \GL_5(\ZZ)$. Act by elements of $G_{\ZZ}$ of the form
\[
   \Bigg(\begin{bmatrix}
    1 & n_{12} & n_{13} & n_{14}  \\
    0 & 1 & n_{23} & n_{24}  \\
    0 & 0 & 1 & n_{34}  \\
    0 & 0 & 0 & 1  \\
  \end{bmatrix}, \begin{bmatrix}
    1 & 0 & 0 & 0 & 0 \\
    m_{12} & 1 & 0 & 0 & 0  \\
    m_{13} & m_{23} & 1 & 0 & 0  \\
    m_{14} & m_{24} & m_{34} & 1 & 0  \\
    m_{15} & m_{25} & m_{35} & m_{45}  & 1\\
  \end{bmatrix}\Bigg)
\]
where $\lv n_{ij}\rv \leq \lceil \exp[X,p_{5-i}-p_{5-j}] \rceil$ and $\lv m_{ij}\rv \leq \lceil \exp[X,q_{6-i}-q_{6-j}] \rceil$.
\end{proof}

\subsection{Computing $\poly_5(G)$}
Let $K$ denote a quintic \'etale algebra over $\QQ$ and $L$ denote the sextic resolvent algebra of $K$. Let $K_0$ denote the subspace of $K$ with trace $0$, and let $L_0$ denote the subspace of $L$ with trace $0$. Let $\varphi \colon L_0 \otimes L_0 \rightarrow K_0$ denote the alternating bilinear form given by the sextic resolvent map. A \emph{flag $(\cF,\cG)$ of $(K,L)$} is a pair of sequences $\cF = (F_0,\dots,F_5)$ and $\cG = (G_0,\dots,G_4)$ of $\QQ$-vector spaces such that
\[
	\{0\} = F_0 \subset F_1 \subset F_2 \subset \dots \subset F_5 = L_0
\]
\[
	\{0\} = G_0 \subset G_1 \subset G_1 \subset \dots \subset G_4 = K_0
\]  
and $\dim_{\QQ} F_i = i$ and $\dim_{\QQ} G_i = i$. Recall that for an integer $n > 1$, the set $[n] = \{0,\dots,n-1\}$. Associate to the flag $(\cF,\cG)$ the function
\begin{align*}
  T_{\cF,\cG} \colon [6] \times [6] & \longrightarrow [5] \\
  (i\;\;,\;\;j ) \; & \longmapsto \min\{k \in [5] \mid \varphi\la F_i \otimes F_j\ra \subseteq G_k\}.
\end{align*}

\begin{definition}
The polytope $P_{\cF,\cG}$ is the set of $p = (p_1,\dots,p_4,q_1,\dots,q_5) \in \RR^9$ satisfying Equation \eqref{basicineqquintic} and the following condition: for all $1 \leq i \leq j \leq 5$, if $T_{\cF,\cG}(i,j) > 0$ then
\[
	p_{5-T_{\cF,\cG}(i,j)} + 1/2 \geq q_{6-i} + q_{6-j}.
\]
\end{definition}

The proof of the following theorem is almost exactly analogous to that of \cref{quarticstructuretheorem}.

\begin{theorem}
\label{quinticstructuretheorem}
We have $\poly_5(G) = \bigcup_{(\cF,\cG)}P_{\cF,\cG}$ where $(\cF,\cG)$ ranges across flags of pairs $(K,L)$ with the property that $\Gal(K) \simeq G$ as a permutation group.
\end{theorem}
\begin{proof}
We first show that $\poly_5(G)$ is contained in $\bigcup_{(\cF,\cG)}P_{\cF,\cG}$. By Minkowski's second theorem and the definition of successive minima, $\poly_5(G)$ satisfies the linear inequalities given in Equation \eqref{basicineqquintic}. Let $R$ be any nondegenerate quintic $G$-ring and let $S$ be a sextic resolvent ring of $G$. Let $K = R \otimes \QQ$ and let $L = S \otimes \QQ$. Picking Minkowski bases $\{v_0=1,\dots,v_4\}$ and $\{w_0=1,\dots,w_5\}$ of $R$ and $S$  gives rise to dual bases $\{v^*_0,\dots,v^*_4\}$ and $\{w^*_0,\dots,w^*_5\}$ of $R$ and $S$ respectively. These give rise to a flag $(\cF,\cG)$ of $(K,L)$ given by $F_0 = G_0 \coloneqq \{0\}$ and
\begin{align*}
	F_i &\coloneqq \QQ\la w^*_5, \dots, w^*_{6-i} \ra \\
	G_i &\coloneqq \QQ\la v^*_4, \dots, v^*_{5-i} \ra.
\end{align*}
Choose $1 \leq i \leq j \leq 5$ and suppose $T_{\cF,\cG}(i,j) > 0$. Then if $T_{\cF,\cG}(i,j) = 1$ then
\[
	0 \neq \varphi\la \ZZ\la w^*_5,\dots,w^*_{6-i} \ra \otimes \ZZ\la w^*_5,\dots,w^*_{6-j} \ra \ra \subseteq \ZZ\la v^*_4,\dots,v^*_{5-T_{\cF,\cG}(i,j)} \ra.
\]
If $T_{\cF,\cG}(i,j) > 1$ then
\[
	\varphi\la \ZZ\la w^*_5,\dots,w^*_{6-i} \ra \otimes \ZZ\la w^*_5,\dots,w^*_{6-j} \ra \ra \subseteq \ZZ\la v^*_4,\dots,v^*_{5-T_{\cF,\cG}(i,j)} \ra
\]
and 
\[
	\varphi\la \ZZ\la w^*_5,\dots,w^*_{6-i} \ra \otimes \ZZ\la w^*_5,\dots,w^*_{6-j} \ra \ra \not\subseteq \ZZ\la v^*_4,\dots,v^*_{5-T_{\cF,\cG}(i,j)+1} \ra
\]
In either case we obtain:
\[
	\lv v^*_{5-T_{\cF,\cG}(i,j)} \rv \ll \lv w^*_{6-i}\rv\lv w^*_{6-j}\rv\Delta^{1/2}
\] 
which implies by \cref{succminlemma} that
\[
	\lv v_{5-T_{\cF,\cG}(i,j)} \rv^{-1} \ll \lv w_{6-i}\rv^{-1}\lv w_{6-j}\rv^{-1}\Delta^{1/2}.
\] 
Multiplying out the inverses, we obtain
\[
	\lv w_{6-i}\rv\lv w_{6-j} \rv \ll \lv v_{5-T_{\cF,\cG}(i,j)} \rv\Delta^{1/2}.
\]
So, $\poly_5(G)$ is contained in the set $\bigcup_{(\cF,\cG)}P_{\cF,\cG}$ where $(\cF,\cG)$ ranges across flags of $(K,L)$ with the property that $\Gal(K) \simeq G$.

Now let $(\cF,\cG)$ be a flag of a pair $(K,L)$ with the property that $\Gal(K) \simeq G$. Choose bases $\{v_0=1,v_1,\dots,v_4\}$ and $\{w_0=1,w_1,\dots,w_5\}$ of $K$ and $L$ giving rise to bases $\{v^*_4,\dots,v^*_1\}$ and $\{w^*_5,\dots,w^*_1\}$ of $K_0$ and $L_0$. These bases give rise to a flag $(\cF,\cG)$ of $(K,L)$. Moreover, the dual bases give rise to a tensor $\mathcal{A} \in V_{\RR}$. For a positive integer $M$, define $g_M \in \GL_4(\RR) \times \GL_5(\RR)$ to be
\[
	g_M = \Bigg(
	\begin{bmatrix}
		M^{p_4} & 0 & 0 & 0 \\
		0 & M^{p_3} & 0 & 0 \\
		0 & 0 & M^{p_2} & 0 \\
		0 & 0 & 0 & M^{p_1} \\	
	\end{bmatrix},
	\begin{bmatrix}
		M^{-q_5} & 0 & 0 & 0 & 0 \\
		0 & M^{-q_4} & 0 & 0 & 0 \\
		0 & 0 & M^{-q_3} & 0 & 0 \\
		0 & 0 & 0 & M^{-q_2} & 0 \\
		0 & 0 & 0 & 0 & M^{-q_1} \\
	\end{bmatrix}\Bigg)
\]
Define $\mathcal{A}_M = \sqrt{M}g_M\mathcal{A}$. For a rational point $p$ in the relative interior of $P_{\cF,\cG}$, one may check that there exists an infinite set $\cM \subseteq \ZZ_{>0}$ such that for all $M \in \cM$ the tensor $\mathcal{A}_M \in V_{\ZZ}$. Each $M \in \cM$ gives rise to a distinct quintic ring and sextic resolvent ring $(R_M,S_M)$ with absolute discriminant $\Delta_M$; for each $M \in \cM$, write the bases of $R_M$ and $S_M$ obtained from $\mathcal{A}_M$ by $\{1,v_{1,M},\dots,v_{4,M}\}$ and $\{1,w_{1,M},\dots,w_{5,M}\}$ respectively. Then for $M,N \in \cM$, we have: 
\begin{align*}
N\Delta_M &= M\Delta_N \\
N^{p_i}v_{i,M} &= M^{p_i}v_{i,N} \\
N^{q_i}w_{i,M} &= M^{q_i}w_{i,N} 
\end{align*}
Therefore, 
\begin{align*}
\lim_{M \in \cM}\Big(\log_{ \Delta_M }\lambda_{1}(R_M),\dots,\log_{ \Delta_M }\lambda_{4}(R_M)\Big) &= (p_1,p_2,p_3,p_4) \\
\lim_{M \in \cM}\Big(\log_{ \Delta_M }\lambda_{1}(S_M),\dots,\log_{ \Delta_M }\lambda_{5}(S_M)\Big) &= (q_1,q_2,q_3,q_4,q_5)
\end{align*}
Therefore, $p \in \poly_5(G)$. Because $\poly_5(G)$ is defined to be a set of limit points, it must be closed; therefore $P_{\cF,\cG}$ is contained in $\poly_5(G)$, which completes the proof of the theorem.
\end{proof}

\begin{proof}[Proof of \cref{quinticcontainmentlemma}]
\cref{quinticstructuretheorem} expresses $\poly_5(G)$ as a finite union of polytopes. A computation shows that these polytopes have dimension $7$.
\end{proof}

To compute $\poly_5(S_5)$, we will need the following lemma.

\begin{lemma}[Bhargava \cite{Bhquintic}, Lemma~10]
\label{quinticzerolemma}
Let $\mathcal{A} \in V_{\QQ}$ be an element such that all the variables in one of the following sets vanish:
\begin{enumerate}
\item $\{a^4_{12},a^4_{13},a^4_{14},a^4_{15},a^4_{23},a^4_{24},a^4_{25}\}$
\item $\{a^4_{12},a^4_{13},a^4_{14},a^4_{23},a^4_{24},a^4_{34}\}$
\item $\{a^4_{12},a^4_{13},a^4_{14},a^4_{15}\} \cup \{a^3_{12},a^3_{13},a^3_{14},a^3_{15}\}$
\item $\{a^4_{12},a^4_{13},a^4_{14},a^4_{23},a^4_{24}\} \cup \{a^3_{12},a^3_{13},a^3_{14},a^3_{23},a^3_{24}\}$
\item $\{a^4_{12},a^4_{13},a^4_{14}\}\cup \{a^3_{12},a^3_{13},a^3_{14}\} \cup \{a^2_{12},a^2_{13},a^2_{14}\}$
\item $\{a^4_{12},a^4_{13},a^4_{23}\}\cup \{a^3_{12},a^3_{13},a^3_{23}\} \cup \{a^2_{12},a^2_{13},a^2_{23}\}$
\item  $\{a^4_{12},a^4_{13}\}\cup \{a^3_{12},a^3_{13}\} \cup \{a^2_{12},a^2_{13}\} \cup \{a^1_{12},a^1_{13}\}$
\end{enumerate}
Then $\mathcal{A}$ is reducible.
\end{lemma}

\begin{proof}[Proof of \cref{s5polytopetheorem}]
Let $(\cF,\cG)$ be any flag of a pair $(K,L)$ such that $\Gal(K)$ is a transitive subgroup of $S_5$. Then the seven statements of \cref{quinticzerolemma} imply that:
\begin{enumerate}
\item $T_{\cF,\cG}(2,5) = 4$
\item $T_{\cF,\cG}(3,4) = 4$
\item $T_{\cF,\cG}(1,5) \geq 3$
\item $T_{\cF,\cG}(2,4) \geq 3$
\item $T_{\cF,\cG}(1,4) \geq 2$
\item $T_{\cF,\cG}(2,3) \geq 2$
\item $T_{\cF,\cG}(1,3) \geq 1$
\end{enumerate}
The $7$ statements above imply that $P_{\cF,\cG}$ satisfies the linear inequalities:
\begin{enumerate}
\item $p_1 + 1/2 \geq q_4 + q_1$
\item $p_1 + 1/2 \geq q_3 + q_2$
\item $p_2 + 1/2 \geq q_5 + q_1$
\item $p_2 + 1/2 \geq q_4 + q_2$
\item $p_3 + 1/2 \geq q_5 + q_2$
\item $p_3 + 1/2 \geq q_4 + q_3$
\item $p_4 + 1/2 \geq q_5 + q_3$
\end{enumerate} 
In total, this proves that $\poly_5(\Gal(K))$ is contained in the region stated in the theorem statement. The converse is shown via computer calculation by producing an example of $\mathcal{A} = (A_1,A_2,A_3,A_4) \in V_{\QQ}$ of the form
\[
	\Bigg(\begin{bmatrix}
		- & 0 & * & * & *\\
		- & - & * & * & *\\
		- & - & - & * & *\\
		- & - & - & - & *\\
		- & - & - & - & -\\	
	\end{bmatrix},
	\begin{bmatrix}
		- & 0 & 0 & * & *\\
		- & - & * & * & *\\
		- & - & - & * & *\\
		- & - & - & - & *\\
		- & - & - & - & -\\	
	\end{bmatrix},
	\begin{bmatrix}
		- & 0 & 0 & 0 & *\\
		- & - & 0 & * & *\\
		- & - & - & * & *\\
		- & - & - & - & *\\
		- & - & - & - & -\\	
	\end{bmatrix},
	\begin{bmatrix}
		- & 0 & 0 & 0 & 0\\
		- & - & 0 & 0 & *\\
		- & - & - & * & *\\
		- & - & - & - & *\\
		- & - & - & - & -\\	
	\end{bmatrix}\Bigg)
\]
defining an $S_5$-quintic field.
\end{proof}

\subsection{Computing density functions of quintic rings}
To make our formulas shorter, we use the following notation in this section. Define:
\[
	f(p) \coloneqq 1+\sum_{1 \leq i < j \leq 5, 1 \leq k \leq 4}\max\{0,-1/2 + q_i + q_j - p_k\}
\]
\[
	g(p) \coloneqq  \sum_{1 \leq i < j \leq 4}(p_j-p_i) + \sum_{1 \leq i < j \leq 5}(q_j-q_i)
\]

\begin{lemma}
\label{davenportquintic}
For any $C,X$ and $p \in \RR^9$ satisfying Equation \eqref{basicineqquintic}, we have
\[
	\{V_{\ZZ} \cap CB(p,X) \} \asymp_C \exp[X,f(p)].
\]
\end{lemma}
\begin{proof}
The argument is identical to that of \cref{davenportquartic}.
\end{proof}

\begin{lemma}
\label{upperboundquintic}
For any $\eps$, $X$, and $p \in \poly_5$, we have
\[
	\#\cF_5(p,\eps,X) \ll_{\eps} \exp[X,f(p)-g(p)].
\]
\end{lemma}
\begin{proof}
The proof is again completely identical to that of \cref{upperboundquartic}.
\end{proof}

\begin{lemma}
\label{lowerboundquintic}
There is a constant $C$ such that for all $\eps > C$ we have:
\begin{enumerate}
\item if $p \in \poly_5$, then 
\begin{align*}
\#\cF_4(p,\eps,X) \gg_{p,\eps} \exp[X,f(p)-g(p)];
\end{align*}
\item if $p \in \poly_5(S_5)$, then 
\begin{align*}
\#\cF(p,\eps,S_5,X) \gg_{p,\eps} \exp[X,f(p)-g(p)];
\end{align*}
\end{enumerate}
\end{lemma}
\begin{proof}
The proof of $(1)$ is completely analogous to that of $(1)$ in \cref{lowerboundquartic}. The argument in \cref{lowerboundquartic} also proves $(2)$, with the additional input that there exists $D,M \in \RR_{\geq 1}$ such that a positive proportion, as $X \rightarrow \infty$, of the lattice points in $DB(p,X) \cap V_{\ZZ}$ have Galois group $S_5$ and $\Delta \geq X/M$; it suffices to drop the latter condition and simply prove that  there exists $D \in \RR_{\geq 1}$ such that a positive proportion of the lattice points in $DB(p,X) \cap V_{\ZZ}$ have Galois group $S_5$ as $X \rightarrow \infty$. Consider the subset of lattice points of $DB(p,X)$ such that $\Delta \equiv 5 \Mod{25}$, the ring $R$ is irreducible modulo $7$, and the ring $S$ is irreducible modulo $11$; every element of this subset has Galois group $S_5$. If this set is nonempty for large enough $X$, then it forms a positive proportion of $DB(p,X)$. The nonemptiness follows from the fact that there exists $\mathcal{A}$ of the form
\[
	\Bigg(\begin{bmatrix}
		- & 0 & * & * & *\\
		- & - & * & * & *\\
		- & - & - & * & *\\
		- & - & - & - & *\\
		- & - & - & - & -\\	
	\end{bmatrix},
	\begin{bmatrix}
		- & 0 & 0 & * & *\\
		- & - & * & * & *\\
		- & - & - & * & *\\
		- & - & - & - & *\\
		- & - & - & - & -\\	
	\end{bmatrix},
	\begin{bmatrix}
		- & 0 & 0 & 0 & *\\
		- & - & 0 & * & *\\
		- & - & - & * & *\\
		- & - & - & - & *\\
		- & - & - & - & -\\	
	\end{bmatrix},
	\begin{bmatrix}
		- & 0 & 0 & 0 & 0\\
		- & - & 0 & 0 & *\\
		- & - & - & * & *\\
		- & - & - & - & *\\
		- & - & - & - & -\\	
	\end{bmatrix}\Bigg)
\]
such that $\Delta = 5 \Mod{25}$, the ring $R$ is irreducible modulo $7$, and the ring $S$ is irreducible modulo $11$. Choosing $D$ to be large enough such that $\mathcal{A} \in DB(p,X)$ completes the proof of this case. 
\end{proof}

\begin{proof}[Proof of \cref{quintictotaldensity}]
This follows from \cref{lowerboundquintic} and \cref{upperboundquintic}.
\end{proof}

\begin{proof}[Proof of \cref{s5densitythm}]
This follows from \cref{lowerboundquintic} and \cref{upperboundquintic}.
\end{proof}

\subsection{Special families of quintic rings}

\subsubsection{Binary quintic forms}
Let $f(x,y) = f_0x^5 + f_1x^4y + f_2x^3y^2 + f_3x^2y^3 + f_4xy^4 + f_5y^5 \in \Sym^5 \ZZ^2$ be a binary quintic form. In \cite{fess}, Fess defines a map $\Sym^5 \ZZ^2 \rightarrow \ZZ^4 \otimes \wedge^2\ZZ^5$; we modify this map slightly and send a form $f$ to the matrices\footnote{Our map and the map of Fess send a form $f$ to isomorphic quintic rings and sextic resolvent rings. The only difference is a choice of basis of the quintic ring and sextic resolvent ring.}:
\[
	\begin{bmatrix}
		- & 0 & 0 & 0 & 0\\
		- & - & -1 & 0 & 0\\
		- & - & - & 0 & 0\\
		- & - & - & - & -f_5\\
		- & - & - & - & -\\	
	\end{bmatrix},
	\begin{bmatrix}
		- & 0 & 1 & 0 & 0\\
		- & - & 0 & 1 & 0\\
		- & - & - & 0 & -f_3\\
		- & - & - & - & -f_4\\
		- & - & - & - & -\\	
	\end{bmatrix},
	\begin{bmatrix}
		- & 0 & 0 & -1 & 0\\
		- & - & 0 & 0 & -1\\
		- & - & - & -f_1 & -f_2\\
		- & - & - & - & 0\\
		- & - & - & - & -\\	
	\end{bmatrix},
	\begin{bmatrix}
		- & 0 & 0 & 0 & 1\\
		- & - & 0 & 0 & 0\\
		- & - & - & -f_0 & 0\\
		- & - & - & - & 0\\
		- & - & - & - & -\\	
	\end{bmatrix}
\]

Let $p = (1/8,1/8,1/8,1/8,1/4,1/4,1/4,3/8,3/8)$. Call the map above $\psi$. Observe that for any form $f \in \Sym^5 \ZZ^2$, we have $H_c(f) \leq X^{1/8}$ if and only if $\psi(f) \in B(p,X)$. 

\begin{proof}[Proof of \cref{binaryquintictheorem}]
The proof is almost identical to that of \cref{binaryquartictheorem}. Let $C \in \RR_{\geq 1}$. Let $f \in \cF_5^{\bin}$ and let $(R,S)$ be the corresponding rings. By \cref{lambdaboundlemmaquintic}, if $\Delta \geq X/C$ and $\psi(f) \in B(p,X)$, then
\begin{align*}
\lambda_i(R) &\asymp_C \Delta^{p_i} \\
\lambda_i(S) &\asymp_C \Delta^{q_i}.
\end{align*}
To prove \eqref{correspondquintic1}, it suffices to show that
\[
\lim_{C \rightarrow \infty} \liminf_{X \rightarrow \infty} \frac{\#\{f \in \cF_5^{\bin} \; \mid \; H_c(f) \leq X \text{ and } \Delta \geq X^8/C\}}{\#\{f \in \cF_5^{\bin} \; \mid \; H_c(f) \leq X\}} = 1.
\]
The denominator is clearly $X^6 + o(X^6)$. By Davenport's lemma, we have:
\begin{align*}
&\#\{f \in \cF_5^{\bin}  \; \mid \; H_c(f) \leq X \text{ and } \Delta > X^8/C\} \\
&= \#\{(f_0,\dots,f_5) \in \ZZ^6 \; \mid \; \max\{\lv f_i \rv\}\leq X \text{ and } \lv \disc(f_0,\dots,f_5) \rv > X^8/C\} \\
&= \Vol\{(f_0,\dots,f_5) \in \RR^6 \; \mid \; \max\{\lv f_i \rv\}\leq X \text{ and } \lv \disc(f_0,\dots,f_5) \rv > X^8/C\} + O(X^{5}) \\
&= X^6\Vol\{(f_0,\dots,f_5) \in \RR^5  \; \mid \; \max\{\lv f_i \rv \}\leq 1 \text{ and } \lv \disc(f_0,\dots,f_5) \rv > 1/C\} + O(X^{5})
\end{align*}

We now prove \eqref{correspondquintic2}. Observe that by \cref{upperboundquintic}, there exists a constant $C$ such that for all $\eps > C$, there are $\ll_{\eps} X^{3/4}$ isomorphism classes of tuples $(R,S)$ such that $\Delta \leq X$ and $(R,S)$ is $(\eps,X)$-close to $p$. \cref{evertse} shows that there are at most $O(1)$ $\GL_2(\ZZ)$-equivalence classes of binary quintic forms giving rise to the same quintic ring. Therefore, it suffices to show that there are $\gg X^{3/4}$ $\GL_2(\ZZ)$-equivalence classes of binary quintics such that $\Delta \leq X$ and the corresponding rings $(R,S)$ are $(\eps,X)$-close to $p$; this follows from \cref{binthm} when $r = (1/40,1/40)$.
\end{proof}

\subsubsection{Monic quintic forms}
Let $\cF_5^{\mon}$ denote the set of monic quintic forms with integral coefficients and nonzero discriminant. Let
\[
	p = (1/20,2/20,3/20,4/20,4/20,5/20,6/20,7/20,8/20).
\]
Observe that for any monic integral quintic form $f$, we have $H_r(f) \leq X^{1/20}$ if and only if $\psi(f) \in B(p,X)$. 

\begin{proof}[Proof of \cref{monogenicquintictheorem}]
Again, the proof is almost identical to that of \cref{monogenicquartictheorem}. Let $C \in \RR_{\geq 1}$.  For any $f \in \cF_5^{\mon}$ let $(R,S)$ be the corresponding rings. By \cref{lambdaboundlemmaquintic}, if $\Delta \geq X/C$ and $\psi(f) \in B(p,X)$, then
\begin{align*}
\lambda_i(R) &\asymp_C \Delta^{p_i} \\
\lambda_i(S) &\asymp_C \Delta^{q_i}.
\end{align*}
To prove \eqref{correspondquintic1}, it suffices to show that
\[
\lim_{C \rightarrow \infty} \liminf_{X \rightarrow \infty} \frac{\#\{f \in \cF_5^{\mon} \; \mid \; H_r(f) \leq X \text{ and } \Delta \geq X^{20}/C\}}{\#\{f \in \cF_5^{\mon} \; \mid \; H_r(f) \leq X\}} = 1.
\]
The denominator is clearly $X^{15} + o(X^{15})$. By Davenport's lemma, we have:
\begin{align*}
&\#\{f \in \cF_5^{\mon}  \; \mid \; H_r(f) \leq X \text{ and } \Delta > X^{20}/C\} \\
&= \#\{(f_1,\dots,f_5) \in \ZZ^5 \; \mid \; \max\{\lv f_i \rv^{1/i}\}\leq X \text{ and } \lv \disc(1,f_1,\dots,f_5) \rv > X^{20}/C\} \\
&= \Vol\{(f_1,\dots,f_5) \in \RR^5 \; \mid \; \max\{\lv f_i \rv^{1/i}\}\leq X \text{ and } \lv \disc(1,f_1,\dots,f_5) \rv > X^{20}/C\} + O(X^{14}) \\
&= X^{15}\Vol\{(f_1,\dots,f_5) \in \RR^5  \; \mid \; \max\{\lv f_i \rv \}\leq 1 \text{ and } \lv \disc(1,f_1,\dots,f_5) \rv > 1/C\} + O(X^{14})
\end{align*}

We now prove \eqref{correspondquintic2}. Observe that by \cref{upperboundquintic}, there exists a constant $M$ such that for all $\eps > M$, there are $\ll_{\eps} X^{7/10}$ isomorphism classes of tuples $(R,S)$ such that $\Delta \leq X$ and $(R,S)$ is $(\eps,X)$-close to $p$. \cref{evertse} shows that there are at most $O(1)$ $\GL_2(\ZZ)$-equivalence classes of binary quintic forms giving rise to the same quintic ring. Therefore, it suffices to show that there are $\gg X^{7/10}$ $\GL_2(\ZZ)$-equivalence classes of binary quintics such that $\Delta \leq X$, the binary quintic represents $1$, and the corresponding rings $(R,S)$ are $(\eps,X)$-close to $p$; this follows from \cref{binthm} when we let $r = (0,1/20)$.
\end{proof}

\subsubsection{The generic point of $S_5$}
We now prove \cref{generics5thm} by appealing to upcoming work of H. Given a nondegenerate rank $n$ ring $R$, we consider the rank $(n-1)$ lattice $\{x \in \ZZ + nR \; \mid \; \Tr(x) = 0\}$, up to isometry. It is naturally seen as an element of the double coset space
\[
	\mathcal{S}_{n-1} \coloneqq \GL_{n-1}(\ZZ)\backslash \GL_{n-1}(\RR)/\GO_{n-1}(\RR).
\]
The space $\mathcal{S}_{n-1}$ has a natural measure $\mu_{n-1}$, obtained from the Haar measure on $\GL_{n-1}(\RR)$ and $\GO_{n-1}(\RR)$. To a nondegenerate pair $(R,S)$, we can associate an element $(\cL_R,\cL_S) \in \cS_{4} \times \cS_5$; we call this the \emph{shape} of the pair $(R,S)$. 

\begin{theorem}[H]
\label{equidistribofquintic}
When ordered by discriminant, the shapes of pairs $(R,S)$, where $R$ is an $S_5$-ring, are equidistributed with respect to $\mu_4\times \mu_5$.
\end{theorem}

\begin{proof}[Proof of \cref{generics5thm}]
Let $p(S_5) = (1/8,1/8,1/8,1/8,3/10,3/10,3/10,3/10,3/10)$. We will first prove \eqref{correspondquintic1}. For an integer $n\geq 2$, an element $\cL \in \cS_{n}$, and an integer $0 \leq i < n$, define $\lambda_i(\cL)$ to be the $i$-th successive minimum of a lattice representing $\cL$ with determinant $1$. For a positive real number $\eps$ and an integer $n \geq 2$, define
\[
	\cS_{n,\eps} \coloneqq \{\cL \in \cS_n \mid \max_{1 \leq i \leq n}\{\log\lambda_i(\cL)\} \leq \eps\}.
\]
Then:
\begin{align*}
&\lim_{\eps \rightarrow \infty} \liminf_{X \rightarrow \infty}\frac{\#\{(R,S) \; \mid \; \Delta \leq X \text{ and $R$ is a $S_5$-ring that is $(\eps,X)$-close to $p(S_5)$}\}}{\#\{(R,S) \; \mid \; \Delta \leq X \text{ and $R$ is a $S_5$-ring}\}} \\
&\geq \lim_{\eps \rightarrow \infty} \liminf_{X \rightarrow \infty}\frac{\#\{(R,S) \; \mid \Delta \leq X \text{ and $R$ is a $S_5$-ring and } (\cL_R,\cL_S) \in \cS_{4,\eps} \times \cS_{5,\eps}\}}{\#\{(R,S) \; \mid \; \Delta \leq X \text{ and $R$ is a $S_5$-ring}\}} \\
&= \lim_{\eps \rightarrow \infty}\frac{(\mu_4\times\mu_5)(\cS_{4,\eps} \times \cS_{5,\eps})}{(\mu_4\times\mu_5)(\cS_{4} \times \cS_5)} \\
&= 1
\end{align*}
We prove \eqref{correspondquintic1} with the following observation: \cref{upperboundquintic} and \cref{lowerboundquintic} combine to prove that there exists a positive real number $C$ such that for all $\eps > C$, a positive proportion of rings in $\cF(p(S_5),\eps,X)$ as $X \rightarrow \infty$ are $S_5$-rings.
\end{proof}

\section{Appendix}
We will need the following lemma on directed graphs.
\begin{lemma}
\label{cyclelemma}
Let $G$ be any directed graph on $\geq 2$ vertices. If every vertex of $G$ has an outgoing edge, then $G$ has a cycle.
\end{lemma}
\begin{proof}
Start at a vertex and travel along an outgoing edge to another vertex. Continuing this way, because the graph has finitely many vertices, we will eventually return to a vertex that we have been at previously. Thus the graph contains a cycle.
\end{proof}

\begin{lemma}
\label{generalboundlemma}
Let $n \in \ZZ_{\geq 1}$, let $a_1,\dots,a_{n},C$ be any positive real numbers, let $X \in \RR_{\geq 1}$, and let $p=(p_1,\dots,p_n) \in \RR^n$. If
\begin{equation}
\label{maineqngen}
a_i^2 \leq C\max_{1\leq j\leq n}\{X^{2p_i-p_j} a_j \}
\end{equation}
for all $1 \leq i \leq n$, then $a_i \leq CX^{p_i}$ for all $1 \leq i \leq n$.
\end{lemma}
\begin{proof}
We proceed by induction on $n$. If $n = 1$ then $a_1^2 \leq CX^{p_1}a_1$ so $a_1 \leq CX^{p_1}$.

Now fix $n\geq 2$ and suppose the lemma is true for all $1 \leq m < n$. Let 
\[
	S = \{1 \leq i \leq n \mid a_i \leq CX^{p_i}\}.
\]
Assume for the sake of contradiction that $S = \emptyset$; equivalently, assume that
\begin{equation}
\label{hereasump}
a_i > CX^{p_i}
\end{equation}
for all $1 \leq i \leq n$. Define the directed graph $G$ on the set of vertices $\{1,2,\dots,n\}$ as follows: draw an edge from $i$ to $j$ if
\begin{equation}
\label{edgeeqn}
a_i^2 \leq CX^{2p_i-p_j}a_j.
\end{equation} 

\eqref{maineqngen} and \eqref{hereasump} together imply that every vertex of $G$ has an outgoing edge. \cref{cyclelemma} implies that $G$ contains a cycle; choose any cycle of $G$ and let $T$ denote the set of vertices of that cycle. Multiplying \eqref{edgeeqn} together for every edge in the cycle and dividing both sides by $\prod_{i \in T}a_i$, we obtain
\[
	\prod_{i \in T}a_i \leq C^{\#T}\prod_{i \in T}X^{p_i}.
\]
However, \eqref{hereasump} implies that
\[
	C^{\#T}\prod_{i \in T}X^{p_i} < \prod_{i \in T}a_i,
\]
and hence we have arrived at a contradiction. Thus, $S \neq \emptyset$. 

For all $i \notin S$ and for all $j \in S$, we have
\begin{align*}
a_i^2 &> C^2X^{2p_i} &&\text{because $i \notin S$} \\
&\geq CX^{2p_i-p_j}a_j. && \text{because $j \in S$}
\end{align*}
Therefore for all $i \notin S$, we may rewrite \eqref{maineqngen} as
\[
	a_i^2 \leq C\max_{k\notin S}\{X^{2p_i-p_k} a_k \}.
\]
Because $S \neq \emptyset$ the induction assumption implies that $a_i \leq CX^{p_i}$  for all $i \notin S$.

\end{proof}

\end{document}